\documentclass[10pt,oneside,leqno]{amsart}
\usepackage[latin9]{inputenc}
\usepackage{amsthm}
\usepackage{amssymb}
\usepackage{esint}
\usepackage[unicode=true]
 {hyperref}
\makeatletter
\numberwithin{equation}{section}
\numberwithin{figure}{section}
\theoremstyle{plain}
\newtheorem{thm}{\protect\theoremname}
  \theoremstyle{plain}
  \newtheorem{cor}[thm]{\protect\corollaryname}
  \theoremstyle{plain}
  \newtheorem{lem}[thm]{\protect\lemmaname}
 \newtheorem{rem}[thm]{Remark}
\makeatother
\newcommand{\bea}{\begin{eqnarray}}
\newcommand{\eea}{\end{eqnarray}}
\newcommand{\bg}{\begin{gathered}}
\newcommand{\eg}{\end{gathered}}
  \providecommand{\corollaryname}{Corollary}
  \providecommand{\lemmaname}{Lemma}
\providecommand{\theoremname}{Theorem}
\newcommand{\gauss}[2]{\genfrac{[}{]}{0pt}{}{#1}{#2}_q}
\renewcommand{\b}{\beta}
\renewcommand{\a}{\alpha}

\newcommand{\f}{\phi}

\renewcommand{\C}{\mathbb{C}}
\begin{document}

\title{On Some $2D$ Orthogonal $q$-Polynomials}

\author{Mourad E. H. Ismail }

\address[Mourad E. H. Ismail]{Department of Mathematics\\
 University of Central Florida \\
 Orlando, Florida 32816 USA \\
 and King Saud University, Riyadh, Saudi Arabia }

\email{\textit{mourad.eh.ismail@gmail.com}}

\author{Ruiming Zhang}

\address[Ruiming Zhang]{College of Science\\
Northwest A\&F University\\
Yangling, Shaanxi 712100\\
P. R. China.}
\email[Corresponding author]{\textit{ruimingzhang@gmail.com}}

\thanks{Research supported by the DSFP at King Saud University in Riyadh. Supported also by  Research Grants Council of Hong Kong  contract \# 1014111}
\begin{abstract}
We introduce two $q$-analogues of the $2D$-Hermite polynomials which are functions 
of two complex variables. We derive explicit formulas, orthogonality relations, raising 
and lowering operator relations, generating functions, and  Rodrigues formulas for both 
families.  We also introduce a $q-2D$ analogue of the disk polynomials (Zernike polynomials) 
and  derive similar formulas for them as well including evaluating certain  connection 
coefficients. Some of the generating functions may be related to Rogers--Ramanujan type 
identities. 
\end{abstract}

\subjclass[2000]{Primary 33C50, 33D50, Secondary 33C45, 33D45.}

\keywords{Disc polynomials, Zernike polynomials, 2$D$-Hermite polynomials, $q$-2$D$-Hermite polynomials,  generating functions, 
ladder operators, $q$-Sturm--Liouville equations, q-integrals, $q$-Zernike polynomials, Ramanujan's beta integrals, large degree asymptotics, scaled asymptotics, connection relations, 
Askey-Roy integral, Rogers--Ramanujan identities.}

\maketitle

\noindent{\bf filename} q-2DHermite13.tex 

\noindent{\bf Date}: 8/19/2015

 \noindent{\bf Journal}: Transactions of the American  Mathematical Society, to appear. 
  \tableofcontents

\section{Introduction}

The $2D$-Hermite (or complex Hermite) polynomials 
 $\left\{ H_{m,n}(z_{1},z_{2})\right\} _{m,n=0}^{\infty}$, 
\begin{equation}
H_{m,n}(z_{1},z_{2})=\sum_{k=0}^{m\wedge n}(-1)^{k}k!{m \choose k}{n \choose k}z_{1}^{m-k}z_{2}^{n-k}.\label{1.1}
\end{equation}
were introduced in \cite{Ito}. Recently several mathematical physicists studied these polynomials 
from mathematical and physical points of view, \cite{Ali:Bag:Hon},  \cite{Cot:Gaz:Gor}, 
\cite{Gha}--\cite{Gha2}, \cite{Int:Int}.  Their combinatorics were studied in 
 \cite{Ism:Sim}, \cite{Ism:Zen},  and in 
\cite{Ism4}. Ismail \cite{Ism4} proved a Kibble-Slepian type multilinear generating function for 
these polynomials while the present authors gave a new proof together with a proof of the original 
Kibble-Slepian formula for Hermite polynomials in the forthcoming work \cite{Ism:Zha}.  
Relevant references on the $2D$-Hermite polynomials are \cite{Shi}, 
\cite{Thi:Hon:Krz}, \cite{Wun}--\cite{Wun2}, and \cite{Van:Mey}. 

This work introduces two $q$-analogues of the $2D$-Hermite polynomials denoted by 
$\{H_{m,n}(z_{1},z_{2}|q)\}$, and  $\{h_{m,n}(z_{1},z_{2}|q)\}$ and a $q-2D$ sequence of ultraspherical  polynomials.  The polynomials $H_{m,n}(z_{1},z_{2}|q)\}$, and  $\{h_{m,n}(z_{1},z_{2}|q)\}$  transform to each other  as 
$q \to 1/q$.  We produce one orthogonality measure for the first family   while we give infinity many orthogonality measures for the second family.  An orthogonality measure is given for the 
 $q-2D$ ultraspherical polynomials. 
 We also find the  raising and lowering operators for both families of $q-2D$ Hermite polynomials together with   Sturm-Liouville equations which they satisfy.  

In Section 2 we collect all the preliminary results used throughout the paper. Section 3 treats the 
polynomials $\{H_{m,n}(z_{1},z_{2}|q)\}$ while Section 4 treats the second family 
$\{h_{m,n}(z_{1},z_{2}|q)\}$. In Section 5 we first give the definition  of  a  set or 
orthogonal polynomials in two variables, 
the disk polynomials \cite[\S 2.3]{Dun:Xu}. They are  also known as Zernike polynomials 
\cite{Zer}. The rest of Section 5 
contains the definition and properties of the $q-2D$ ultraspherical polynomials denoted by $\{p_{m,n}(z\bar z;b|q)\}$. These are $q$-analogues of the disk polynomials. They constitute a   $q$-analogue which is  different from the one  introduced by Floris in 
\cite{Flo}--\cite{Flo2}, see also \cite{Flo:Koe}.  
      Section 6  has several applications of the results obtained in the earlier sections including multilinear generating 
functions.  In Section 7    we establish moment type representations for 
$\{H_{m,n}(z_{1},z_{2}|q)\}$, and  $\{h_{m,n}(z_{1},z_{2}|q)\}$ and give closed form expressions for the connection coefficients in the expansion of $\{H_{m,n}(z_{1},z_{2}|q)\}$,  (respectively  $\{h_{m,n}(z_{1},z_{2}|q)\}$) in  $\{h_{m,n}(z_{1},z_{2}|q)\}$, (respectively 
$\{H_{m,n}(z_{1},z_{2}|q)\}$).     In addition  we give a two dimensional $q$-analogue  of the generating function \cite[(4.6.29)]{Ism}.  
\begin{eqnarray}
\sum_{n=0}^\infty H_{m+n}(x) \frac{t^n}{n!} = \exp(2xt -t^2) H_m(x-t).
\end{eqnarray}
A formula that may have ramifications on the theory of partitions is formula \eqref{eqhasPGF}. 
  In Section 8 we show that the zero sets of $\{H_{m,n}(z, \bar z|q)\}$, $\{h_{m,n}(z, \bar z|q)\}$
 and $\{p_{m,n}(z, \bar z; b|q)\}$ are concentric circles  in $\C$ centered  at $z =0$.  We also show that the limiting distribution of the zeros of $\{H_{m,n}(z, \bar z|q)\}$ 
 and $\{p_{m,n}(z, \bar z; b|q)\}$ coincide with the support of their measures of orthogonality. The polynomials $\{h_{m,n}(z, \bar z|q)\}$ are orthogonal on an bounded sets with respect to different measures. We  describe their zero sets as $m,n \to \infty$.  The asymptotics involves the zeros of the Ramanujan function to be defined in \eqref{eqAq}. This is similar to the one variable $q$-polynomials  in \cite{Ism2}. In Section 
 9  we show that certain matrices whose entries are formed by 2$D$-polynomials are positive definite. 

This is the first part in a series of papers on the subject of $2D$ orthogonal polynomials where we study several new families of orthogonal polynomials.

 \section{Preliminaries} 
 In this section we collect all the formulas used in the later sections and mention some of the notation.  We 
 shall follow the notation and terminology for special functions and $q$-series in \cite{And:Ask:Roy}, \cite{Gas:Rah}, \cite{Ism}, and 
\cite{Koe:Swa}.  We assume the reader is familiar with the notations of $q$-shifted factorials as 
well as the unilateral and bilateral 
 basic hypergeometric functions ${}_r\phi_s$ and ${}_r\psi_r$.  Moreover we use the notations 
 \begin{eqnarray}
 \label{eqdeffrac}
&{}&  \{x\} = \textup{the fractional part of} \; x, \quad \lfloor{x}\rfloor  = x-  \{x\}\\
&{}&  m \wedge n = \textup{min}\; \{m,n\}. 
\end{eqnarray}
 
 The $q$-difference and dilation operators are  
\begin{eqnarray}
(D_q f)(z)=    \frac{f(z)-f(qz)}{z-qz}, z \ne 0, \quad \textup{and} \quad (\eta_q f)(x) = f(qx), 
\label{eqdefDq}
\end{eqnarray}
respectively. 
If the dependence on $z$ is important we shall use $D_{q,z}$and $\eta_{q,z}$  instead of $D_q$ 
and $\eta_{q}$, respectively.   
The Leibniz rule for $D_q$ is 
\begin{eqnarray}
\label{eqLeib}
D_q^n (fg)(x) = \sum_{k=0}^n \gauss{n}{k} D_q^k f(x) \eta_q^k D_q^{n-k} g(x).
\end{eqnarray}

The $q$-binomial theorem is \cite[(II.3)]{Gas:Rah}
\begin{eqnarray}
\sum_{n=0}^\infty \frac{(a;q)_n}{(q;q)_n} \; z^n = \frac{(az;q)_\infty}{(z;q)_\infty}, \quad |z| < 1. 
\label{eqqBT}
\end{eqnarray}
The terminating case is 
\begin{eqnarray}
\label{eqqbt}
\sum_{k=0}^n \gauss{n}{k} q^{\binom{k}{2}} (-z)^k = (z;q)_n. 
\end{eqnarray}
Two important special and limiting case are the Euler identities \cite[(II.1)--(II.2)]{Gas:Rah}
\begin{eqnarray}
\sum_{n=0}^\infty \frac{z^n}{(q;q)_n} &=& \frac{1}{(z;q)_\infty}, \label{eqEuler1}\\
\sum_{n=0}^\infty \frac{(-z)^n}{(q;q)_n}\; q^{\binom{n}{2}} &=& (z;q)_\infty. \label{eqEuler2}
\end{eqnarray}

The $q$-integral is, \cite[\S 1.11]{Gas:Rah} 
\begin{eqnarray}
\int_0^\infty f(x) d_qx = (1-q) \sum_{n=-\infty}^\infty q^n f(q^n).
\label{eqdefqInt}
\end{eqnarray}

The $q$-Laguerre polynomials are \cite[3.21.1)]{Koe:Swa}
\begin{eqnarray}
\label{eqqLag}
\bg
L_n^{(\alpha)}(x;q)= \frac{(q^{\alpha+1};q)_n}{(q;q)_n} {}_1\phi_1\left(\left.  \begin{array}{c}
 q^{-n}  \\
 q^{\alpha+1}
     \end{array} \right| q, -q^{n+\alpha+1}x  \right) 
      \qquad \\ \qquad    
     = \frac{1}{(q;q)_n} {}_2\phi_1\left(\left.  \begin{array}{c}
 q^{-n}, -x  \\
 q^{\alpha+1}
     \end{array} \right| q, q^{n+\alpha+1}  \right).
\eg
\end{eqnarray}
Their moment problem is indeterminate, that is there are infinitely many orthogonality measures with respect to which the $q$-Laguerre polynomials are orthogonal. For a treatment of the 
$q$-Laguerre polynomials and the corresponding 
moment problem we refer the interested reader to \cite[\S 21.8]{Ism}. The little $q$-Laguerre, also known as Wall polynomials are defined by \cite[(3.20.1)]{Koe:Swa}
\begin{eqnarray}
\label{eqlittqL}
p_n(x;a|q) =  {}_2\phi_1\left(\left.  \begin{array}{c}
 q^{-n}, 0  \\
aq
     \end{array} \right| q,  qx  \right) 
     = \frac{1}{(q^{-n}/a;q)_n} {}_2\phi_0\left(\left.  \begin{array}{c}
 q^{-n}, 1/x  \\
 -
     \end{array} \right| q, \frac{x}{a}  \right).
\end{eqnarray}
See also \S 11 of Chapter VI in \cite{Chi}.  

The $q$-Bessel functions $J_\nu^{(2)}$ and $I_\nu^{(2)}$ are defined by   
\begin{eqnarray}
\label{eqJnu2}
J_\nu^{(2)}(z;q) &=& \frac{(q^{\nu+1};q)_\infty}{(q;q)_\infty} 
\sum_{n=0}^\infty \frac{(-1)^n q^{n(n+\nu)}}{(q, q^{\nu+1};q)_n} \left(\frac{z}{2}\right)^{\nu+2n}, \\
\label{eqInu2}
I_\nu^{(2)}(z;q) &=& \frac{(q^{\nu+1};q)_\infty}{(q;q)_\infty} 
\sum_{n=0}^\infty \frac{q^{n(n+\nu)}}{(q, q^{\nu+1};q)_n} 
\left(\frac{z}{2}\right)^{\nu+2n},  
\end{eqnarray}
respectively, \cite{Ism}, \cite{Gas:Rah}.   
The Ramanujan function is \cite{Ism}
\begin{eqnarray}
\label{eqAq}
A_q(z) = \sum_{n=0}^\infty \frac{q^{n^2}}{(q;q)_n} (-z)^n.
\end{eqnarray}
It had has only positive zeros. It appeared in Ramanujan's Lost Note Book \cite{Ram} with some statements about the asymptotics of its zeros. 

Garrett, Ismail, and Stanton \cite{Gar:Ism:Sta} generalized the Rogers--Ramanujan identities to 
\begin{equation}
\label{eqGIS}
\sum_{n=0}^\infty\frac{q^{n^2+mn}}{(q;q)_n}
=\frac{(-1)^m q^{-\binom{m}2} a_m(q)}{(q,q^4;q^5)_\infty}
+\frac{(-1)^{m+1} q^{-\binom{m}2}b_m(q)}{(q^2,q^3;q^5)_\infty}
\end{equation}
where
\begin{eqnarray}
\label{eqSchurP}
\begin{gathered}
a_m(q)=\sum_{j} q^{j^2+j}\gauss{m-j-2}{j},  \qquad 
b_m(q) =\sum_{j} q^{j^2}\gauss{m-j-1}{j}.
\end{gathered}
\end{eqnarray}
The polynomials $a_m(q)$ and $b_m(q)$ were considered by Schur in conjunction with his proof of the Rogers--Ramanujan identities,  see
\cite{And2}  and \cite{Gar:Ism:Sta} 
  for details.  We shall refer to $a_m(q)$ and $b_m(q)$ as the Schur polynomials.

Let $q=e^{-2k^{2}}$and $\left|q\right|<1$, the Ramanujan's identities are 

\begin{equation}
\begin{aligned} & \int_{-\infty}^{\infty}e^{-x^{2}+2mx}\left(-aqe^{2kx},-bqe^{-2kx};q\right)_{\infty}dx
=\frac{\sqrt{\pi}\left(abq;q\right)_{\infty}e^{m^{2}}}{\left(ae^{2mk}\sqrt{q},be^{-2mk}\sqrt{q};q\right)_{\infty}},
\end{aligned}
\label{eq:ram1}
\end{equation}
\cite[Ex 6.15)(i)]{Gas:Rah},  and
\begin{equation}
\begin{aligned} & \int_{-\infty}^{\infty}\frac{e^{-x^{2}+2mx}dx}{\left(ae^{2ikx}\sqrt{q},be^{-2ikx}\sqrt{q};q\right)_{\infty}}
 =\frac{\sqrt{\pi}e^{m^{2}}\left(-aqe^{2imk},-bqe^{-2imk};q\right)_{\infty}}{\left(abq;q\right)_{\infty}},
\end{aligned}
\label{eq:ram2}
\end{equation}
\cite[Ex 6.15)(ii)]{Gas:Rah}.
For $0<q<1$, $\Re\left(a+c\right)>0$ and $\Re\left(b-c\right)>0$,
Ramanjuan extended the beta integral on $(0,\infty)$  to the following integrals, 
\begin{equation}
\begin{aligned} & \int_{0}^{\infty}\frac{\left(-tq^{b},-q^{a+1}/t;q\right)_{\infty}t^{c-1}d_{q}t}{\left(-t,-q/t;q\right)_{\infty}\left(1-q\right)}
  =\frac{\left(q,-q^{c},-q^{1-c},q^{a+b};q\right)_{\infty}}{\left(-1,-q,q^{a+c},q^{b-c};q\right)_{\infty}},
\end{aligned}
\label{eq:rambeta1}
\end{equation}
\cite[Ex 6.17(i)]{Gas:Rah}
and
\begin{equation}
\begin{aligned}  \int_{0}^{\infty}\frac{\left(-tq^{b},-q^{a+1}/t;q\right)_{\infty}t^{c-1}dt}{\left(-t,-q/t;q\right)_{\infty}}
  =\frac{\Gamma\left(c\right)\Gamma\left(1-c\right)\left(q^{c},q^{1-c},q^{a+b};q\right)_{\infty}}{\left(q,q^{a+c},q^{b-c};q\right)_{\infty}},
\end{aligned}
\label{eq:rambeta2}
\end{equation}
\cite[Ex 6.17(ii)]{Gas:Rah}.  Then, 

\begin{equation}
\begin{aligned}\int_{0}^{\infty}\frac{t^{c-1}d_{q}t}{\left(-t,-q/t;q\right)_{\infty}\left(1-q\right)} & =\frac{\left(q,-q^{c},-q^{1-c};q\right)_{\infty}}{\left(-1,-q;q\right)_{\infty}}\end{aligned}
\label{eq:rambeta3}
\end{equation}
and
\begin{equation}
\begin{aligned}\int_{0}^{\infty}\frac{t^{c-1}dt}{\left(-t,-q/t;q\right)_{\infty}} & =\frac{\Gamma\left(c\right)\Gamma\left(1-c\right)\left(q^{c},q^{1-c};q\right)_{\infty}}{\left(q;q\right)_{\infty}}.\end{aligned}
\label{eq:rambeta4}
\end{equation}

The Askey-Roy integral is \cite[(4.11.1)]{Gas:Rah}
\begin{eqnarray} 
\label{eqAsk:Roy}
\quad  \int_{-\pi}^{\pi}\frac{\left(ce^{i\theta}/\beta,qe^{i\theta}/c\alpha,c\alpha e^{-i\theta},q\beta e^{-i\theta}/c;q\right)_{\infty}}{\left(ae^{i\theta},be^{i\theta},\alpha e^{-i\theta},\beta e^{-i\theta};q\right)_{\infty}}\frac{d\theta}{2\pi}
=\frac{\left(ab\alpha\beta,c,q/c,c\alpha/\beta,q\beta/c\alpha;q\right)_{\infty}}{\left(a\alpha,a\beta,b\alpha,b\beta,q;q\right)_{\infty}}. 
\end{eqnarray}

\section{First $q$- Analogue}

The first $q$-analogue of $\{H_{m,n}(z_1, z_2)\}$ is defined by 
\begin{eqnarray}
 \label{eqHmnq}
H_{m,n}(z_1, z_2|q) = \sum_{k=0}^{m\wedge n}\gauss{m}{k} \gauss{n}{k} (-1)^kq^{\binom{k}{2}}(q;q)_k
z_1^{m-k}z_2^{n-k}.
\end{eqnarray}
Clearly,
\begin{equation}
H_{m,n}\left(z_{2},z_{1}\big|q\right)=H_{n,m}\left(z_{1},z_{2}\big|q\right).\label{eq:Hpsymmetry}
\end{equation}

\begin{thm} \label{thmH}
The polynomials $\{H_{m,n}(z_1, z_2|q)\}$ satisfy the relations 
\begin{eqnarray}
\sum_{m, n=0}^\infty H_{m,n}(z_1, z_2|q)  \frac{u^m\; v^n}{(q;q)_m(q;q)_n} 
= \frac{(uv;q)_\infty}{(uz_1, vz_2;q)_\infty}
\label{eqGFHq}
\end{eqnarray}
\begin{equation}
H_{m,n}\left(qz_{1},z_{2}\vert q\right)=H_{m,n}(z_1, z_2|q)-z_{1}\left(1-q^{m}\right)H_{m-1,n}(z_1, z_2|q),\label{eq:Hp18}
\end{equation}
 
\begin{equation}
H_{m,n}(z_1, qz_2|q)=H_{m,n}(z_1, z_2|q)-z_{2}\left(1-q^{n}\right)H_{m,n-1}(z_1, z_2|q),\label{eq:Hp19}
\end{equation}
\begin{equation}
H_{m,n}(qz_1, z_2|q)q^{-m}=H_{m,n}(z_1, z_2|q)-q^{-1}\left(1-q^{m}\right)\left(1-q^{n}\right)H_{m-1,n-1}(z_1, z_2|q)
\label{eq:Hp20}
\end{equation}

\begin{equation}
H_{m,n}(z_1, qz_2|q)q^{-n}=H_{m,n}(z_1, z_2|q)-q^{-1}\left(1-q^{m}\right)\left(1-q^{n}\right)H_{m-1,n-1}(z_1, z_2|q),\label{eq:Hp21}
\end{equation}
\begin{eqnarray}
\label{eqHmn3trr}
\bg
\qquad z_1 H_{m,n}(z_1, z_2|q) =   q^m(1-q^n )H_{m,n-1}(z_1, z_2|q)+ H_{m+1,n}(z_1, z_2|q) \\
\qquad z_2 H_{m,n}(z_1, z_2|q) =   q^n(1-q^m) H_{m-1,n}(z_1, z_2|q)+ H_{m,n+1}(z_1, z_2|q)
\eg
\end{eqnarray}
Moreover they have the operational representation 
\begin{eqnarray}
\label{eqoprepH}
 H_{m,n}(z_1, z_2|q) = ((1-q)^2D_{q,z_1}D_{q,z_2};q)_\infty z_1^m z_2^n 
\end{eqnarray}
\end{thm}
Before proving Theorem \ref{thmH} we consider some of its implications.  
We note that \eqref{eq:Hp18} and \eqref{eq:Hp19} are  the 
lowering operator relations 
\begin{eqnarray}
D_{q,z_1}H_{m,n}(z_1, z_2|q) &=& \frac{1-q^m}{1-q} H_{m-1,n}(z_1, z_2|q), \\
D_{q,z_2}H_{m,n}(z_1, z_2|q) &=& \frac{1-q^n}{1-q} H_{m,n-1}(z_1, z_2|q), 
\end{eqnarray}
respectively.  Moreover we observe  that \eqref{eq:Hp20}--\eqref{eq:Hp21} imply the symmetry relation 
\begin{eqnarray}
\label{eqSym}
H_{m,n}(qz_1, z_2|q)q^{-m} = H_{m,n}(z_1, qz_2|q)q^{-n}. 
\end{eqnarray}
Indeed \eqref{eqSym} can be proved directly from the generating function \eqref{eqGFHq}. Finally we record a possible connection between the generating function  \eqref{eqGFHq} 
and partitions. Let $M(m,n)$ denotes the number of partitions of a positive integer $n$ with crank $= m$. Andrews and Garvan  \cite{And:Gar} established the generating function 
\begin{eqnarray}
\label{eqGFcranks}
\sum_{n=0}^\infty \sum_{m=-\infty}^\infty  M(m,n) z^m q^n = \frac{(q;q)_\infty}{(qz, q/z;q)_\infty}. 
\end{eqnarray}
It is clear that \eqref{eqGFcranks} is a special case of our generating function  \eqref{eqGFHq}. This suggests that there may be a more refined  statistic defined on partitions which will have
the generating function  \eqref{eqGFHq}.

\begin{proof}[Proof of Theorem \ref{thmH}]
The generating function 
follows from \eqref{eqHmnq} and the Euler sums \eqref{eqEuler1}-\eqref{eqEuler2}. 
\eqref{eq:Hp18} and \eqref{eq:Hp19} follow from
\[
\frac{\left(uv;q\right)_{\infty}}{\left(uz_{1}q,vz_{2};q\right)_{\infty}}=
\left(1-uz_{1}\right)\frac{\left(uv;q\right)_{\infty}}{\left(uz_{1},vz_{2};q\right)_{\infty}}
\]
and
\[
\frac{\left(uv;q\right)_{\infty}}{\left(uz_{1},vz_{2}q;q\right)_{\infty}}=\left(1-vz_{2}\right)\frac{\left(uv;q\right)_{\infty}}{\left(uz_{1},vz_{2};q\right)_{\infty}},
\]
\eqref{eq:Hp20} and \eqref{eq:Hp21} follow from 
\[
\frac{\left(uq^{-1}v;q\right)_{\infty}}{\left(uq^{-1}z_{1}q,vz_{2};q\right)_{\infty}}=\left(1-q^{-1}uv\right)\frac{\left(uv;q\right)_{\infty}}{\left(uz_{1},vz_{2};q\right)_{\infty}}
\]
 and
\[
\frac{\left(uvq^{-1};q\right)_{\infty}}{\left(uz_{1},vq^{-1}z_{2}q;q\right)_{\infty}}=\left(1-q^{-1}uv\right)\frac{\left(uv;q\right)_{\infty}}{\left(uz_{1},vz_{2};q\right)_{\infty}}.
\]
The first 3-term recurrence follows from \eqref{eq:Hp18} and \eqref{eq:Hp20}, similarly, the second one can be obtained from \eqref{eq:Hp19} and \eqref{eq:Hp21}. It can be proved directly. It is clear that $z_1 H_{m,n}(z_1, z_2|q)  - H_{m+1,n}(z_1, z_2|q)$ is 
\begin{eqnarray}
\notag
\bg
 \sum_{k=0}^{(m+1)\wedge n}\left\{\gauss{m}{k} - \gauss{m+1}{k}\right\} \gauss{n}{k} 
 (-1)^kq^{\binom{k}{2}}  (q;q)_k z_1^{m+1-k}z_2^{n-k}  \\
 = - 
  \sum_{k=1}^{m\wedge n}  q^{m+1-k}
   \gauss{m}{k-1}   \gauss{n}{k}  (-1)^kq^{\binom{k}{2}} (q;q)_k z_1^{m+1-k}z_2^{n-k}
\eg
\end{eqnarray} 
which gives the first recurrence relation after replacing $k$ by $k+1$. The proof of the second recurrence relation is similar. The representation  of \eqref{eqoprepH} follows by expanding 
$((1-q)^2D_{q,z_1}D_{q,z_2};q)_\infty)$ using \eqref{eqEuler2}. 
\end{proof}

 \begin{thm}
\label{thmRodHmn}
The polynomials satisfy the Rodrigues type formula 
 \begin{eqnarray}
 \label{eqRodHmnq}
 H_{m,n}(z_1, z_2|q)  = \frac{(1-1/q)^{m+n}q^{mn}}{(qz_1z_2;q)_\infty}
  D_{q^{-1}, z_2}^m  
 D_{q^{-1}, z_1}^n \left((qz_1z_2;q)_\infty\right)
\end{eqnarray}
and the raising relations
\begin{eqnarray}
 H_{m+1,n}(z_1, z_2|q)   &=&q^n \frac{1-1/q}{(qz_1z_2;q)_\infty}
  D_{q^{-1}, z_2} \left((qz_1z_2;q)_\infty H_{m,n}(z_1, z_2|q) \right), \label{eqraisHm}\\
  H_{m,n+1}(z_1, z_2|q)   &=&q^m \frac{1-1/q}{(qz_1z_2;q)_\infty}
  D_{q^{-1}, z_1} \left((qz_1z_2;q)_\infty H_{m,n}(z_1, z_2|q) \right).  \label{eqraisHn}
\end{eqnarray}
Moreover the polynomials $\{H_{m+1,n}(z_1, z_2|q)\}$ have  the multiplication formula 
\begin{eqnarray}
H_{m,n}(az_{1},bz_{2}|q) = \sum_{j=0}^{m\wedge n}
 \gauss{m}{j}\gauss{n}{j} \frac{H_{m-j,n-j}(z_{1},z_{2}|q)}{a^{j-m}b^{j-n} }  (q, 1/ab;q)_j (q;q)_j.  
\label{eqMF1}
\end{eqnarray}
\end{thm}
 \begin{proof}
 It is clear $(1-1/q)D_{q^{-1}, z_1} (qz_1z_2;q)_\infty = z_2(qz_1z_2;q)_\infty.$ Therefore the right-hand side of \eqref{eqRodHmnq} is 
 \begin{eqnarray}
 \notag
 \bg
 q^{nm} (1-1/q)^m D_{q^{-1}, z_2}^m\left[  z_2^n(qz_1z_2;q)_\infty \right] \\
 = q^{nm} (1-1/q)^m \sum_{k=0}^m  \genfrac{[}{]}{0pt}{}{m}{k}_{q^{-1}} \frac{z_1^k}{(1-1/q)^k}
(q^{-k} z_2)^{n-m+k} \frac{(q^{-n};q)_{m-k}}{(1-1/q)^{m-k}}\\
= q^{nm}   \sum_{k=0}^m  \genfrac{[}{]}{0pt}{}{m}{k}_{q^{-1}} z_1^{m-k} 
z_2^{n-k} q^{-(m-k)(n-k)} (q^{-n};q)_k 
=  H_{m,n}(z_1, z_2|q),
\eg
 \end{eqnarray}
 and we have proved \eqref{eqRodHmnq}.  Formulas  \eqref{eqraisHm} and 
  \eqref{eqraisHn} follow directly from  \eqref{eqRodHmnq}. The generating function \eqref{eqGFHq} implies 
\begin{eqnarray}
\sum_{n=0}^\infty H_{m,n}(z_1, z_2|q)  \frac{u^m\; v^n}{(q;q)_m(q;q)_n} 
= \frac{(uv;q)_\infty}{(abuv;q)_\infty} \frac{(abuv;q)_\infty}{(uz_1, vz_2;q)_\infty}
\end{eqnarray}
and \eqref{eqMF1}   follow from the $q$-binomial theorem \eqref{eqqBT}. 
 \end{proof}

In the next section we shall introduce the polynomials $\{h_{m,n}\left(z_{1},z_{2}| q\right)\}$, see 
\eqref{eq:hp1} and \eqref{eqhvsH}.  We also note \eqref{eq:h2l} which indicates  their relation to the $q$-Laguerre polynomials, \cite{Koe:Swa}.  We now show a connection between the polynomials $\{H_{m,n}(z_1, z_2|q)\}$ 
and the little $q$-Jacobi polynomials, \cite{Koe:Swa}. 
\begin{eqnarray*}
H_{m,n}\left(z_{1},z_{2}\big|q\right) & = & q^{mn}i^{m+n}h_{m,n}\left(z_{1}/i,z_{2}/i|q^{-1}\right)\\
 & = & q^{mn}\left(q^{-1};q^{-1}\right)_{n}z_{1}^{m-n}L_{n}^{\left(m-n\right)}\left(-z_{1}z_{2};q^{-1}\right)\\
 & = & q^{mn}\left(q^{n-m-1};q^{-1}\right)_{n}z_{1}^{m-n}p_{n}\left(z_{1}z_{2},q^{m-n}\bigg|q\right)\\
 & = & \left(-1\right)^{n}\frac{\left(q;q\right)_{m}q^{\binom{n}{2}}}{\left(q;q\right)_{m-n}}z_{1}^{m-n}p_{n}\left(z_{1}z_{2},q^{m-n}\bigg|q\right),
\end{eqnarray*}
or
\begin{equation}
H_{m,n}\left(z_{1},z_{2}\big|q\right)=\left(-1\right)^{n}\frac{\left(q;q\right)_{m}q^{\binom{n}{2}}}{\left(q;q\right)_{m-n}}z_{1}^{m-n}p_{n}\left(z_{1}z_{2},q^{m-n}\bigg|q\right),\label{eq:H2w}
\end{equation}
where $p_{n}\left(x;q^{\alpha}\vert q\right)$ is the little $q$-Laguerre
or Wall's polynomials, \cite{Koe:Swa}
\begin{eqnarray*}
p_{n}\left(x;a\vert q\right) & = & {}_{2}\phi_{1}\left(\begin{array}{cc}
\begin{array}{c}
q^{-n},0\\
aq
\end{array} & \bigg|q;qx\end{array}\right). 
\end{eqnarray*}
They satisfy the discrete orthogonality relation
\[
\begin{aligned}   \sum_{k=0}^{\infty}a^{k}q^{k}\left(q^{k+1};q\right)_{\infty}p_{m}\left(q^{k};a\vert q\right)p_{n}\left(q^{k};a\vert q\right)
   =\frac{\left(q;q\right)_{\infty}}{\left(aq;q\right)_{\infty}}\frac{\left(aq\right)^{n}\left(q;q\right)_{n}}{\left(aq;q\right)_{\infty}}\delta_{m,n},
\end{aligned}
\]
where $m,n\in\mathbb{N}_{0}$ and $0<a<q^{-1}$.  
\begin{thm}
\label{thm:Hporthogonality}The polynomials $\left\{ H_{m,n}\left(z,\overline{z}\big|q\right)\right\} $
satisfy the following orthogonality

\begin{equation}
\begin{aligned}\int_{\mathbb{C}}H_{m,n}\left(z,\overline{z}\big|q\right)\overline{H_{s,t}\left(z,\overline{z}\big|q\right)}d\mu\left(z,\overline{z}\right) & =\frac{q^{mn}\left(q;q\right)_{m}\left(q;q\right)_{n}}{\left(q;q\right)_{\infty}}\delta_{m,s}\delta_{n,t},\end{aligned}
\label{eq:Hporthogonality}
\end{equation}
where
\[
d\mu\left(z,\overline{z}\right)=\frac{d\theta}{2\pi}\otimes\sum_{k=0}^{\infty}\frac{q^{k}}{\left(q;q\right)_{k}}\delta\left(r-q^{k/2}\right),
\]
and $z=re^{i\theta},r\in\mathbb{R}^{+},\theta\in[0,2\pi],\ m,n,s,t\in\mathbb{N}_{0}$.
\end{thm}
\begin{proof}
We may assume that $m\ge n$ because of the symmetry \eqref{eq:Hpsymmetry}.
Then apply \eqref{eq:H2w} and change into polar coordinates  to get

\[
\begin{aligned} & \int_{0}^{2\pi}H_{m,n}\left(z,\overline{z}\big|q\right)\overline{H_{s,t}\left(z,\overline{z}\big|q\right)}d\mu\left(z,\overline{z}\right)\\
 & =(-1)^{n+t}\frac{\left(q;q\right)_{m}q^{\binom{n}{2}}}{\left(q;q\right)_{m-n}}\frac{\left(q;q\right)_{s}q^{\binom{t}{2}}}{\left(q;q\right)_{s-t}}\int_{0}^{2\pi}e^{i\theta\left(m-n+t-s\right)}\frac{d\theta}{2\pi}\\
 & \times\sum_{k=0}^{\infty}\frac{q^{k}}{\left(q;q\right)_{k}}p_{n}\left(r^{2},q^{m-n}\bigg|q\right)p_{t}\left(r^{2},q^{s-t}\bigg|q\right)r^{m-n+s-t}\delta\left(r-q^{k/2}\right)\\
 & =(-1)^{n+t}\frac{\left(q;q\right)_{m}q^{\binom{n}{2}}}{\left(q;q\right)_{m-n}}\frac{\left(q;q\right)_{s}q^{\binom{t}{2}}}{\left(q;q\right)_{s-t}}\delta_{m-n+t-s,0}\\
 & \times\sum_{k=0}^{\infty}\frac{q^{k(1+m-n)}}{\left(q;q\right)_{k}}p_{n}\left(q^{k},q^{m-n}\bigg|q\right)p_{t}\left(q^{k},q^{m-n}\bigg|q\right)\\
 & =\frac{q^{mn}\left(q;q\right)_{m}\left(q;q\right)_{n}}{\left(q;q\right)_{\infty}}\delta_{m,s}\delta_{n,t}.
\end{aligned}
\]
This completes the proof of the orthogonality relation. 
\end{proof}

It is clear that the orthogonality relation \eqref{eq:Hporthogonality} and the generating function 
 \eqref{eqGFHq} imply the $q$-beta integral
 \begin{eqnarray}
 \int_{\mathbb{C}} \frac{d\mu(z, \bar z)}{(u_1z, v_1\bar z, v_2z, u_2 \bar z;q)_\infty} 
 = \frac{(u_1u_2v_1 v_2;q)_\infty}{(q, u_1 u_2, v_1 v_2, u_1v_1, u_2v_2;q)_\infty}. 
 \end{eqnarray}

The large degree asymptotics of $H_{m,n}\left(z,\overline{z}\big|q\right)$ are straightforward. 
Indeed \eqref{eqHmnq} and Tannery's theorem  show  that 
\begin{eqnarray}
\label{Hmtoinf}
\lim_{m\to \infty} z_1^{-m}H_{m,n}\left(z_1, z_2\big|q\right) = z_2^n \sum_{k=0}^n 
\gauss{n}{k} q^{\binom{k}{2}} (- z_1z_2)^{-k}=  z_2^n(1/z_1z_2;q)_n, 
\end{eqnarray}
where we used the $q$ binomial theorem \eqref{eqqbt} in the last step.  Similarly 
\begin{eqnarray}
\label{Hntoinf}
\lim_{n\to \infty} z_2^{-n}H_{m,n}\left(z_1, z_2\big|q\right) = z_1^m  (1/z_1z_2;q)_m.  
\end{eqnarray}
One can similarly show that 
\begin{eqnarray}
\lim_{m,n \to \infty}z_1^{-m} z_2^{-n}H_{m,n}\left(z_1, z_2\big|q\right) = (1/z_1z_2;q)_\infty. 
\label{eqHmntoinf}
\end{eqnarray}
 It must be noted that the convergence in \eqref{Hmtoinf}--\eqref{eqHmntoinf} is uniform on compact subsets of the $z_1$ and $z_2$ planes. 
 
 \begin{thm}
 The polynomials $\{H_{m,n}\left(z_1, z_2\big|q\right) \}$ have the generating function 
 \begin{eqnarray}
 \label{eqHmnu+v}
 \bg
 \sum_{m,n=0}^\infty H_{m,n}\left(z_1, z_2\big|q\right) \frac{u^m(a/u;q)_n v^n(b/v;q)_n}{(q;q)_m(q;q)_n} \\ = \frac{(az_1. bz_2;q)_\infty}{(uz_1vz_2;q)_\infty} \; 
 {}_{2}\phi_{2}\left(\begin{array}{c}
a/u,  b/v\\
 az_1, bz_2
\end{array},\bigg|q; uv \right) 
 \eg
 \end{eqnarray}
 \end{thm}
 \begin{proof}
 From the explicit representation \eqref{eqHmnq}  it follows that the right-hand side of \eqref{eqHmnu+v} is equal to 
 \begin{eqnarray}
 \bg
 \sum_{m\ge k \ge0, n\ge k \ge 0}^\infty (-1)^k q^{\binom{k}{2}} 
  \frac{u^m(a/u;q)_m v^n(b/v;q)_n}{(q;q)_k(q;q)_{m-k}(q;q)_{n-k}} z_1^{m-k} z_2^{n-k} \\
  = \sum_{k= 0}^\infty \frac{(-1)^k q^{\binom{k}{2}}}{(q;q)_k}  u^k(a/u;q)_k v^k(b/v;q)_k
  \sum_{m=0}^\infty \frac{(uz_1)^m(aq^k/u;q)_m}{(q;q)_m}  
  \sum_{m=0}^\infty \frac{(vz_2)^n(bq^k/v;q)_n}{(q;q)_n} \\
  =  \sum_{k= 0}^\infty \frac{(-1)^k q^{\binom{k}{2}}}{(q;q)_k}  u^k(a/u;q)_k v^k(b/v;q)_k
  \frac{(az_1q^k, bz_2q^k;q)_\infty}{(uz_1, v_2;q)_\infty}, 
 \eg
 \notag
 \end{eqnarray}
 and the theorem follows.
 \end{proof}

The  polynomials $\{H_{m,n}(z_1,z_2|q)\}$ have an additional orthogonality relation, which we now record.
\begin{thm}
We have the orthogonality relation 
\begin{eqnarray}
\bg
\sum_{j=0}^p\sum_{k=0}^s 
 \int_0^\pi  (q, e^{2i\theta}, e^{-2i\theta};q)_\infty  \frac{H_{j,k}(re^{i\theta},re^{-i\theta}|q)
H_{s-k,p-j}(re^{i\theta},re^{-i\theta}|q)}
{(q;q)j(q;q)_k(q;q)_{s-k}(q;q)_{p-j}}  \; \frac{d\theta}{\pi} \\
=\frac{r^{2p}(1/r^2;q)_p}{(q;q)_p}   {}_1\phi_1\left(\left.  \begin{array}{c}
 q^{-s} \\
 q^{1-p}r^2
     \end{array} \right| q,  q  \right) \delta_{s,p}. 
\eg
\end{eqnarray}
\end{thm} 
\begin{proof}
A special case of the Askey--Wilson integral is  \cite{Ism}, \cite{Gas:Rah}
\begin{eqnarray}
\int_0^\pi \frac{(e^{2i\theta}, e^{-2i\theta};q)_\infty}{(ae^{i\theta}, ae^{-i\theta}, be^{i\theta}, 
be^{-i\theta};q)_\infty}  \; d\theta = \frac{\pi}{(q, ab;q)_\infty}. 
\end{eqnarray}
Therefore 
\begin{eqnarray}
\notag
\bg
\frac{(q;q)_\infty}{\pi} \int_0^\pi (e^{2i\theta}, e^{-2i\theta};q)_\infty
\left[ \sum_{j,k, m,n=0}^\infty  \frac{H_{j,k}(re^{i\theta},re^{-i\theta}|q)
H_{m,n}(re^{i\theta},re^{-i\theta}|q)}
{(q;q)j(q;q)_k(q;q)_m(q;q)_n} u^jv^k v^m u^n \right] \; d\theta \\
= \frac{(q;q)_\infty}{\pi} \int_0^\pi \frac{(uv, uv;q)_\infty (e^{2i\theta}, 
e^{-2i\theta};q)_\infty}{(ure^{i\theta}, vr e^{-i\theta}, vre^{i\theta}, 
ur e^{-i\theta};q)_\infty}  d\theta = \frac{(uv, uv;q)_\infty}{(uvr^2;q)_\infty} \\
=  \sum_{s,t =0}^\infty \frac{(-1)^sq^{\binom{s}{2}}(1/r^2;q)_t}{(q;q)_s(q;q)_t} 
r^{2t}(uv)^{s+t}.
\eg
\end{eqnarray}
Therefore $j+n$ must $= k+m = p$, say. The coefficient of $(uv)^p$ 
in the above expression is 
\begin{eqnarray}
\notag
\bg
r^{2p}  \sum_{s =0}^p \frac{(-1)^sq^{\binom{s}{2}}(1/r^2;q)_{p-s}}{(q;q)_s(q;q)_{p-s}} r^{-2s}
 = \frac{r^{2p}(1/r^2;q)_p}{(q;q)_p}   {}_1\phi_1\left(\left.  \begin{array}{c}
 q^{-s} \\
 q^{1-p}r^2
     \end{array} \right| q,  q  \right),
\eg
\end{eqnarray}
and the theorem follows. 
\end{proof}
 
 The next theorem gives sharp bounds on the zeros of $H_{m,n}(z_1,z_2|q)$. 
 \begin{thm}
\label{thm:H-asymptotics 2} Let $a,b,c,d\ge0$, $\tau=\tau\left(m,n;a,b,c,d\right)=\left\lfloor \left(a+c\right)m+\left(b+d\right)n\right\rfloor $
and $\chi=\chi\left(m,n;a,b,c,d\right)=\left\{ \left(a+c\right)m+\left(b+d\right)n\right\} $,
for $0<\tau\left(m,n;a,b,c,d\right)<m\wedge n$, then
\begin{eqnarray*}
 \lim_{m,n\to \infty}  \frac{\left(q;q\right)_{\infty}H_{m,n}\left(z_{1}q^{am+bn-\frac{1}{4}},z_{2}q^{cm+dn-\frac{1}{4}}|q\right)\left(-z_{1}z_{2}\right)^{\tau}}{z_{1}^{m}z_{2}^{n}q^{am^{2}+\left(b+c\right)mn+dn^{2}-\left(m+n\right)/4-\tau^{2}/2-\tau\chi}} =
   \theta_{4}\left(z_{1}z_{2}q^{\chi};q^{1/2}\right) 
\end{eqnarray*}
holds uniformly on compact subsets of the $z_1$ and $z_2$ planes. 
\end{thm}
\begin{proof}
The proof follows from 
\begin{eqnarray*}
 &  & \frac{H_{m,n}(z_{1}q^{-1/4},z_{2}q^{-1/4}|q)}{z_{1}^{m}z_{2}^{n}}\left(q;q\right)_{\infty}q^{\left(m+n\right)/4}\\
 & = & \frac{1}{\left(q^{m+1},q^{n+1};q\right)_{\infty}}\sum_{k=0}^{m\wedge n}q^{k^{2}/2}\left(-z_{1}z_{2}\right)^{-k}\left(q^{k+1},q^{m-k+1},q^{n-k+1};q\right)_{\infty}. 
\end{eqnarray*}
\end{proof}
The application of this theorem to derive sharp bounds for the zeros of $p_{m,n}$ will done in Section 8. 

 \section{A second $q$-Analogue}
 Our second $q$-analogue is defined by the explicit representation  
\begin{eqnarray}
\bg
h_{m,n}\left(z_{1},z_{2}|q\right)  \\
\qquad :=  \left(-1\right)^{n}\left(q^{m-n+1};q\right)_{n}z_{1}^{m-n}{}_{1}\phi_{1}\left(\begin{array}{c}
q^{-n}\\
q^{m-n+1}
\end{array},\bigg|q;-q^{m+1}z_{1}z_{2}\right) \\
   =    q^{mn}z_{1}^{m}z_{2}^{n}\sum_{j=0}^{\infty}\frac{\left(q^{-m},q^{-n};q\right)_{j}}{\left(q;q\right)_{j}}\left(\frac{-q}{z_{1}z_{2}}\right)^{j}.
 \eg
 \label{eq:hp1}
\end{eqnarray}
Equivalently 
\begin{equation}
h_{m,n}\left(z_{1},z_{2}|q\right)= \sum_{j=0}^{m\wedge n}\gauss{m}{j} \gauss{n}{j} q^{\left(m-j\right)\left(n-j\right)}\left(-1\right)^{j}\left(q;q\right)_{j}z_{1}^{m-j}z_{2}^{n-j}.
\label{eq:hp2}
\end{equation}
Note the $(m,n)-(z_1, z_2)$ symmetry
\begin{eqnarray}
h_{m,n}\left(z_{1},z_{2}|q\right) = h_{n,m}\left(z_{2},z_{1}|q\right) 
\label{eq:hpsymmetry}
\end{eqnarray}
It is easy to see that  
\begin{eqnarray}
 \genfrac{[}{]}{0pt}{}{m}{k}_{q^{-1}} =q^{k(k-m)} \gauss{m}{n}.  
\end{eqnarray}
Therefore 
\begin{eqnarray}
h_{m,n}\left(z_{1},z_{2}|1/q\right) = q^{-mn}i^{-m-n}H_{m,n}(iz_1, iz_2|q) 
\label{eqhvsH}
\end{eqnarray}

\begin{thm}\label{thm9}
The polynomials $\left\{ h_{m,n}\left(z_{1},z_{2}|q\right)\right\}$
have the following properties

\begin{equation}
\sum_{m,n=0}^{\infty}\frac{h_{m,n}\left(z_{1},z_{2}|q\right)}{\left(q;q\right)_{m}\left(q;q\right)_{n}} 
q^{(m-n)^2/2} u^{m}v^{n}  =  \frac{\left(-q^{1/2}uz_{1},-q^{1/2}vz_{2};q\right)_{\infty}}{\left(-uv;q\right)_{\infty}},
\label{eq:hp3}
\end{equation}
 
\begin{equation}
h_{m,n}\left(z_{1}q^{-1},z_{2}|q\right) = h_{m,n}\left(z_{1},z_{2}|q\right)+z_{1}\left(1-q^{m}\right)q^{ -m}h_{m-1,n}\left(z_{1},z_{2}|q\right),
\label{eq:hp4}
\end{equation}
 
\begin{equation}
h_{m,n}\left(z_{1},z_{2}q^{-1}|q\right)=h_{m,n}\left(z_{1},z_{2}|q\right)+z_{2}\left(1-q^{n}\right)
q^{-n}h_{m,n-1}\left(z_{1},z_{2}|q\right),
\label{eq:hp5}
\end{equation}
 
\begin{equation}
q^{m}h_{m,n}\left(z_{1}/q,z_{2}|q\right)=h_{m,n}\left(z_{1},z_{2}|q\right)
+ \left(1-q^{m}\right)\left(1-q^{n}\right) q^{1-m-n} h_{m-1,n-1}\left(z_{1},z_{2}|q\right),
\label{eq:hp6}
\end{equation}

\begin{equation}
q^{n}h_{m,n}\left(z_{1},z_{2}/q|q\right)=h_{m,n}\left(z_{1},z_{2}|q\right)
+ \left(1-q^{m}\right)\left(1-q^{n}\right)q^{1-m-n} h_{m-1,n-1}\left(z_{1},z_{2}|q\right),
\label{eq:hp7}
\end{equation}
\begin{equation}
q^n z_{1} h_{m,n}\left(z_{1},z_{2}|q\right) = h_{m+1,n}\left(z_{1},z_{2}|q\right)
+ \left(1-q^{n}\right) h_{m,n-1}\left(z_{1},z_{2}|q\right), 
\label{eq:hp22}
\end{equation}
and
\begin{equation}
q^mz_{2} h_{m,n}\left(z_{1},z_{2}|q\right)=h_{m,n+1}\left(z_{1},z_{2}|q\right)+ 
\left(1-q^{m}\right)h_{m-1,n}\left(z_{1},z_{2}|q\right).\label{eq:hp23}
\end{equation}
Moreover they have the Rodrigues type formula
\begin{eqnarray}
\quad  h_{m,n}(z_1, z_2|q) = (q-1)^{m+n} (-z_1z_2;q)_\infty  
 D_{q, z_2}^m  D_{q, z_1}^n\frac{1}{(-z_1z_2;q)_\infty} \label{eqRodhmn}
 \end{eqnarray} 
 Furthermore we also have the operational formula 
\begin{eqnarray}
\label{eqhmnqop}
  h_{m,n}(z_1, z_2|q)= \frac{q^{mn}}{ (-q^{-1}(1-q)^2D_{q^{-1},z_1}D_{q^{-1},z_2};q)_\infty}
  \;  z_1^m z_2^n.  
\end{eqnarray}
\end{thm}
Before proving Theorem \ref{thm9} we explore some of its consequences. First 
note that \eqref{eq:hp4}--\eqref{eq:hp5} describe lowering operators. Indeed they can 
be written as 
\begin{eqnarray}
D_{q^{-1}, z_1} h_{m,n}(z_1, z_2|q) = \frac{q^{-m}-1}{q^{-1}-1} h_{m-1,n}(z_1, z_2|q), 
\label{eqlowerm}
\end{eqnarray}
and 
\begin{eqnarray}
D_{q^{-1}, z_2} h_{m,n}(z_1, z_2|q) =\frac{q^{-n}-1}{q^{-1}-1} h_{m,n-1}(z_1, z_2|q)
\label{eqlowern}, 
\end{eqnarray}
respectively.  The raising operators come from the Rodrigues type formula \eqref{eqRodhmn}. We have 
\begin{eqnarray}
h_{m+1,n}(z_1, z_2|q) &=& (q-1)(-z_1z_2;q)_\infty D_{q, z_1}\left(\frac{h_{m,n}(z_1, z_2|q)}
{(-z_1z_2;q)_\infty}\right), \label{eqraisem}\\
h_{m,n+1}(z_1, z_2|q) &=& (q-1)(-z_1z_2;q)_\infty D_{q, z_2}\left(\frac{h_{m,n}(z_1, z_2|q)}
{(-z_1z_2;q)_\infty}\right).  \label{eqraisen}
\end{eqnarray}

\begin{proof}[Proof of Theorem \ref{thm9}]
The generating function \eqref{eq:hp3} follows from
\begin{eqnarray*}
 &  & \sum_{m,n=0}^{\infty}\frac{h_{m,n}\left(z_{1},z_{2}|q\right)}
 {\left(q;q\right)_{m}\left(q;q\right)_{n}}   q^{(m-n)^2/2} u^{m}v^{n}\\
 & = & \sum_{j=0}^{\infty}\frac{\left(-uv\right)^{j}}{\left(q;q\right)_{j}}\sum_{m=j}^{\infty}\frac{\left(uz_{1}\right)^{m-j}q^{\left(m-j\right)^{2}/2}}{\left(q;q\right)_{m-j}}\sum_{n=j}^{\infty}\frac{\left(vz_{2}\right)^{n-j}q^{\left(n-j\right)^{2}/2}}{\left(q;q\right)_{n-j}}\\
 & = & \sum_{j=0}^{\infty}\frac{\left(-uv\right)^{j}}{\left(q;q\right)_{j}}\left(-q^{1/2}uz_{1},-q^{1/2}vz_{2};q\right)_{\infty}\\
 & = & \frac{\left(-q^{1/2}uz_{1},-q^{1/2}vz_{2};q\right)_{\infty}}{\left(-uv;q\right)_{\infty}}.
\end{eqnarray*}
 Let $z_{1}\to z_{1}/q$ in \eqref{eq:hp3}, then
\[
\frac{\left(-q^{-1/2}uz_{1},-q^{1/2}vz_{2};q\right)_{\infty}}{\left(-quv;q\right)_{\infty}}=\left(1+uz_{1}q^{-1/2}\right)\frac{\left(-q^{1/2}uz_{1},-q^{1/2}vz_{2};q\right)_{\infty}}{\left(-quv;q\right)_{\infty}}
\]
implies \eqref{eq:hp4}, Let $z_{2}\to z_{2}/q$ in \eqref{eq:hp3} we
get \eqref{eq:hp5}.

Let $u\to uq,z_{1}\to z_{1}/q$ in \eqref{eq:hp3}, from 
\[
\frac{\left(-q^{1/2}uz_{1},-q^{1/2}vz_{2};q\right)_{\infty}}{\left(-q^{2}uv;q\right)_{\infty}}=\left(1+quv\right)\frac{\left(-q^{1/2}uz_{1},-q^{1/2}vz_{2};q\right)_{\infty}}{\left(-quv;q\right)_{\infty}}
\]
we get \eqref{eq:hp6}, similarly let $v\to vq,z_{2}\to z_{2}/q$ in
\eqref{eq:hp3} to get \eqref{eq:hp7}. To prove \eqref{eqRodhmn} we first note that 
\begin{eqnarray}
\notag
D_{q,z_1} \frac{1}{(-z_1z_2;q)_\infty} = -\frac{z_2}{1-q} \frac{1}{(-z_1z_2;q)_\infty}.
\end{eqnarray}
Therefore the right-hand side of \eqref{eqRodhmn} is 
\begin{eqnarray}
\notag
\bg
(q-1)^m(-z_1z_2;q)_\infty  D_{q, z_2}^m \left(\frac{1}{(-z_1z_2;q)_\infty} \, z_2^n\right) 
\qquad    \qquad \qquad         \\
\quad  = (-1)^m \sum_{k=0}^m \gauss{m}{k}   (-z_1)^k \eta_{q,z_2}^k  
\frac{(q;q)_n}{(q;q)_{n-m+k}} z_2^{n-m+k} \\
\qquad \qquad  \;\; =  \sum_{k=0}^{m} \gauss{m}{k}   z_1^k   (-1)^{m-k} 
\frac{(q;q)_n}{(q;q)_{n-m+k}} z_2^{n-m+k} q^{k(n-m+k)}.
\eg
\end{eqnarray}
Replace $k$ by $m-k$ and we get the left-hand side of \eqref{eqRodhmn}.  We now come to 
\eqref{eqhmnqop}. It is clear that 
\begin{eqnarray}
\notag
D_{q^{-1}}^k z^r = \frac{q^{-r}-1}{q^{-1}-1} \cdots \frac{q^{k-r-1}-1}{q^{-1}-1} z^{r-k} 
=\frac{ (q;q)_r}{(q;q)_{r-k}} \frac{q^{\binom{k+1}{2} - kr}}{(1-q)^k}  z^{r-k} 
\end{eqnarray}
The  
right-hand side of \eqref{eqhmnqop} is 
\begin{eqnarray}
\notag
q^{mn} \sum_{k=0}^\infty \gauss{m}{k}\gauss{n}{k} 
(-1)^k (q;q)_k q^{k^2-km-kn} z_1^{m-k}z_2^{n-k}.
\end{eqnarray}
and the proof is complete. 
\end{proof}
 
  The next theorem gives multiplication formulas for the polynomials   
  $\{h_{m,n}(z_{1},z_{2}; q)\}$. 
\begin{thm}
We have 
\begin{eqnarray}
h_{m,n}(az_{1},bz_{2};q) =  \sum_{j=0}^{m\wedge n} \gauss{m}{j}\gauss{n}{j} 
(q, ab;q)_j a^{m-j}b^{n-j} 
h_{m-j,n-j}(z_{1},z_{2};q). 
\label{eqMF2}
\end{eqnarray}
\end{thm}
\begin{proof}
 The desired result  
follows from the generating function \eqref{eq:hp3} and the $q$-binomial theorem 
\eqref{eqqBT}. 
\end{proof}

We now identify the $q$-Sturm--Liouville problems whose eigenvalues are  
$\{h_{m,n}\left(z_{1},z_{2}|q\right)\}$. Combine \eqref{eqlowerm} and \eqref{eqraisem} to obtain
\begin{eqnarray}
\label{eqqSLz1}
\qquad - \frac{1}{w(z_1z_2)}D_{q,z_1}\left(w(z_1z_2)D_{q^{-1},z_1} h_{m,n}\left(z_{1},z_{2}|q\right)\right)
= \frac{1-q^m}{(1-q)^2} q^{1-m}h_{m,n}\left(z_{1},z_{2}|q\right),
\end{eqnarray}
where 
\begin{eqnarray}
w(z) := 1/(-z;q)_\infty. 
\end{eqnarray}
Similarly 
\begin{eqnarray}
\label{eqqSLz2}
\qquad - \frac{1}{w(z_1z_2)}D_{q,z_2}\left(w(z_1z_2)D_{q^{-1},z_2} h_{m,n}\left(z_{1},z_{2}|q\right)\right)
= \frac{1-q^n}{(1-q)^2} q^{1-n}h_{m,n}\left(z_{1},z_{2}|q\right).
\end{eqnarray}

We note the relation
\begin{equation}
h_{m,n}\left(z_{1},z_{2}|q\right)=\left(-1\right)^{n}\left(q;q\right)_{n}z_{1}^{m-n}L_{n}^{\left(m-n\right)}\left(z_{1}z_{2};q\right)\label{eq:h2l}
\end{equation}

\begin{thm}
\label{thm:hporthogonality} The polynomials $\left\{ h_{m,n}\left(z,\overline{z}|q\right)\right\} $
satisfy the following orthogonality

\begin{equation}
\begin{aligned}\int_{\mathbb{R}^{2}}h_{m,n}\left(z,\overline{z}|q\right)\overline{h_{s,t}\left(z,\overline{z}|q\right)}\frac{dxdy}{\left(-z\overline{z};q\right)_{\infty}} & =\frac{\pi\log q^{-1}\left(q;q\right)_{m}\left(q;q\right)_{n}}{q^{(m-n)^{2}/2+(m+n)/2}}\delta_{m,s}\delta_{n,t},\end{aligned}
\label{eq:hporthogonality}
\end{equation}
where $m,n,s,t\in\mathbb{N}_{0}$.\end{thm}
\begin{proof}

There is no loss of  generality in assuming  that $m\ge n$ because
of the symmetry property \eqref{eq:hpsymmetry}. Apply \eqref{eq:h2l}
and change into polar coordinates  to get
\[
\begin{aligned} & \int_{\mathbb{R}^{2}}h_{m,n}\left(z,\overline{z}|q\right)\overline{h_{s,t}\left(z,\overline{z}|q\right)}\frac{dxdy}{\left(-z\overline{z};q\right)_{\infty}}\\
 & =(-1)^{n+t}\left(q;q\right)_{n}\left(q;q\right)_{t}\int_{0}^{2\pi}e^{i\theta\left(m-n+t-s\right)}d\theta\\
 & \times\int_{0}^{\infty}L_{n}^{\left(m-n\right)}\left(r^{2};q\right)L_{t}^{\left(s-t\right)}\left(r^{2};q\right)\frac{r^{m-n+s-t+1}dr}{\left(-r^{2};q\right)_{\infty}}\\
 & =(-1)^{n+t}\pi\left(q;q\right)_{n}\left(q;q\right)_{t}\delta_{m-n+t-s,0}\\
 & \times\int_{0}^{\infty}L_{n}^{\left(m-n\right)}\left(x;q\right)L_{t}^{\left(m-n\right)}\left(x;q\right)\frac{x^{(m-n)}dx}{\left(-x;q\right)_{\infty}}\\
 & =(-1)^{n+t}\pi\left(q;q\right)_{n}\left(q;q\right)_{t}\delta_{m-n+t-s,0}\\
 & \times(\log q^{-1}) q^{-(m-n)^{2}/2-(m+n)/2}\left(q^{n+1};q\right)_{m-n}\delta_{n,t}\\
 & =\frac{\pi (\log q^{-1}) \left(q;q\right)_{m}\left(q;q\right)_{n}}{q^{(m-n)^{2}/2+(m+n)/2}}\delta_{m,s}\delta_{n,t},
\end{aligned}
\]
and the proof is complete. 
\end{proof}
 Note that the orthogonality relation \eqref{eq:hporthogonality} and the generating function  
 \eqref{eq:hp3} give rise to the integral
 \begin{eqnarray}
 \bg
 \int_{\mathbb{R}^2} \frac{(-q^{1/2}u_1z, -q^{1/2}v_1\bar z, -q^{1/2}v_2 z, -q^{1/2}u_2 \bar z;q)_\infty} {(-z\bar z ;q)_\infty} dx dy \\
 = \pi \ln q^{-1} \frac{(-u_1v_1, -u_2v_2, -u_1u_2, -v_1v_2;q)_\infty}{(u_1u_2v_1v_2;q)_\infty}. 
\eg
 \label{eqqbqta}
 \end{eqnarray} 
 It is clear that \eqref{eqqbqta} is a $q$-beta integral.

Since the associated moment problem for $L_{n}^{(\alpha)}\left(x;q\right)$
is indeterminate, they have infinitely many orthogonal measures. Let
$x^{\alpha}d\mu(x)$ be such a measure, for example,
\[
d\mu\left(x\right)=x^{-\alpha}w_{QL}\left(x;\alpha,c,\lambda\right)dx,\quad\alpha,\lambda,c>0,
\]
\[
d\mu\left(x\right)=\frac{x}{\left(-x;q\right)_{\infty}}\delta\left(x-cq^{y}\right),\quad y\in\mathbb{Z},c,\alpha+1>0,
\]
etc. it is clear that our proof shows that 
\[
d\sigma\left(re^{i\theta},re^{-i\theta}\right)=\frac{1}{2}d\theta d\mu(r^{2}),\quad r\in\mathbb{R}^{+},\theta\in[0,2\pi]
\]
 is also an orthogonal measure for $h_{m,n}\left(re^{i\theta},re^{-i\theta}|q\right)$
where $r\in\mathbb{R}^{+},\theta\in[0,2\pi]$.

\begin{thm}
\label{thm:ram1} Assume that $q=e^{-2k^{2}}$ and $\left|q\right|<1$, then we
have
\begin{equation}
\begin{aligned}\left(-aqe^{2imk},-bqe^{-2imk};q\right)_{\infty} & =\sum_{s,t=0}^{\infty}\frac{H_{s,t}\left(a,b\vert q\right)}{\left(q;q\right)_{s}\left(q;q\right)_{t}}q^{\frac{\left(s-t\right)^{2}}{2}}\left(q^{\frac{1}{2}}e^{2imk}\right)^{s}\left(q^{\frac{1}{2}}e^{-2imk}\right)^{t},\end{aligned}
\label{eq:ramHgen1}
\end{equation}

\begin{equation}
\begin{aligned}\frac{\left(ab;q\right)_{\infty}}{\left(-ab;ae^{2mk},be^{-2mk};q\right)_{\infty}} & =\sum_{s,t=0}^{\infty}\frac{h_{s,t}\left(a,b|q\right)}{\left(q;q\right)_{s}\left(q;q\right)_{t}}\left(q^{\frac{1}{2}}e^{mk}\right)^{s}\left(q^{\frac{1}{2}}e^{-mk}\right)^{t}\end{aligned}
\label{eq:ramhgen1}
\end{equation}
and for any $0<q<1$

\begin{equation}
\frac{\left(q^{a+b};q\right)_{\infty}}{\left(-q^{a+b},q^{a+c},q^{b-c};q\right)_{\infty}}=\sum_{m,n=0}^{\infty}\frac{h_{m,n}\left(q^{a},q^{b}|q\right)}{\left(q;q\right)_{m}\left(q;q\right)_{n}}q^{c\left(m-n\right)},\label{eq:ramhgen2}
\end{equation}
 where $\Re\left(a+c\right)>0$ and $\Re\left(b-c\right)>0$. \end{thm}
\begin{proof}
Equation \eqref{eq:ramhgen2} could be proved either from equations
\eqref{eq:rambeta1}, \eqref{eq:rambeta3} or equations \eqref{eq:rambeta2},\eqref{eq:rambeta4},
\[
\begin{aligned} & \frac{\left(q,-q^{c},-q^{1-c},q^{a+b};q\right)_{\infty}}{\left(-1,-q,-q^{a+b},q^{a+c},q^{b-c};q\right)_{\infty}}
=\int_{0}^{\infty}\frac{\left(-q^{a+1}/t,-tq^{b};q\right)_{\infty}t^{c-1}d_{q}t}{\left(-q^{a+b}-t,-q/t;q\right)_{\infty}\left(1-q\right)}\\
 & =\sum_{m,n=0}^{\infty}\frac{h_{m,n}\left(q^{a+\frac{1}{2}},q^{b-\frac{1}{2}}|q\right)}{\left(q;q\right)_{m}\left(q;q\right)_{n}}q^{\frac{\left(m-n\right)^{2}}{2}}\int_{0}^{\infty}\frac{t^{c+n-m-1}d_{q}t}{\left(-t,-q/t;q\right)_{\infty}\left(1-q\right)}\\
 & =\sum_{m,n=0}^{\infty}\frac{h_{m,n}\left(q^{a+\frac{1}{2}},q^{b-\frac{1}{2}}|q\right)}{\left(q;q\right)_{m}\left(q;q\right)_{n}}q^{\frac{\left(m-n\right)^{2}}{2}}\frac{\left(q,-q^{c+n-m},-q^{m-n-c+1};q\right)_{\infty}}{\left(-1,-q;q\right)_{\infty}}.
\end{aligned}
\]
 Without loss of generality we may assume that $m\ge n$
\[
\begin{aligned} & \frac{\left(q^{a+b};q\right)_{\infty}}{\left(-q^{a+b},q^{a+c},q^{b-c};q\right)_{\infty}}\\
 & =\sum_{m,n=0}^{\infty}\frac{h_{m,n}\left(q^{a},q^{b}|q\right)}{\left(q;q\right)_{m}\left(q;q\right)_{n}}q^{\frac{\left(m-n\right)^{2}}{2}}\frac{\left(-q^{c+n-m+\frac{1}{2}},-q^{m-n-c+\frac{1}{2}};q\right)_{\infty}}{\left(-q^{c+\frac{1}{2}},-q^{\frac{1}{2}-c};q\right)_{\infty}}
  =\sum_{m,n=0}^{\infty}\frac{h_{m,n}\left(q^{a},q^{b}|q\right)}{\left(q;q\right)_{m}\left(q;q\right)_{n}}q^{c\left(m-n\right)}.
\end{aligned}
\]
For $q=e^{-2k^{2}}$and $\left|q\right|<1$, observe that
\[
\begin{aligned} & \sqrt{\pi}e^{m^{2}}\left(-aqe^{2imk},-bqe^{-2imk};q\right)_{\infty}
=\int_{-\infty}^{\infty}\frac{\left(abq;q\right)_{\infty}e^{-x^{2}+2mx}dx}{\left(ae^{2ikx}\sqrt{q},be^{-2ikx}\sqrt{q};q\right)_{\infty}}\\
 & =\sum_{s,t=0}^{\infty}\frac{H_{s,t}\left(a,b\vert q\right)}{\left(q;q\right)_{s}\left(q;q\right)_{t}}q^{(s+t)/2} \int_{-\infty}^{\infty}e^{-x^{2}+2mx+2iksx-2iktx}dx\\
 & =\sqrt{\pi}e^{m^{2}}\sum_{s,t=0}^{\infty}\frac{H_{s,t}\left(a,b\vert q\right)}{\left(q;q\right)_{s}\left(q;q\right)_{t}}q^{(s-t)^{2}/2}
  \left(q^{1/2}e^{2imk}\right)^{s}\left(q^{1/2}e^{-2imk}\right)^{t}.
\end{aligned}
\]
Equations \eqref{eq:ramhgen1} is proved in a similar fashion. 
 \end{proof}

From \eqref{eq:h2l} and \eqref{eq:H2w} we get
\begin{cor}
\label{thm:ram2} $q=e^{-2k^{2}}$and $\left|q\right|<1$, then we
have
\begin{equation}
\begin{aligned}\left(-aqe^{2imk},-bqe^{-2imk};q\right)_{\infty} & =\sum_{s,t=0}^{\infty}\frac{p_{t}\left(ab,q^{s-t}\bigg|q\right)}{\left(q;q\right)_{s}\left(q;q\right)_{t}}q^{\frac{\left(s-t\right)^{2}}{2}}\left(aq^{\frac{1}{2}}e^{2imk}\right)^{s}\left(a^{-1}q^{\frac{1}{2}}e^{-2imk}\right)^{t},\end{aligned}
\label{eq:ramHgen2}
\end{equation}

\begin{equation}
\begin{aligned}\frac{\left(ab;q\right)_{\infty}}{\left(-ab;ae^{2mk},be^{-2mk};q\right)_{\infty}} & =\sum_{s,t=0}^{\infty}\frac{L_{t}^{\left(s-t\right)}\left(ab;q\right)}{\left(q;q\right)_{s}}\left(aq^{\frac{1}{2}}e^{mk}\right)^{s}\left(-a^{-1}q^{\frac{1}{2}}e^{-mk}\right)^{t}\end{aligned}
\label{eq:ramhgen1-1}
\end{equation}
and for any $0<q<1$

\begin{equation}
\frac{\left(q^{a+b};q\right)_{\infty}}{\left(-q^{a+b},q^{a+c},q^{b-c};q\right)_{\infty}}=\sum_{m,n=0}^{\infty}\frac{L_{n}^{\left(m-n\right)}\left(q^{a+b};q\right)}{\left(q;q\right)_{m}}\left(-1\right)^{n}q^{\left(a+c\right)\left(m-n\right)},\label{eq:ramhgen2-1}
\end{equation}
where $\Re\left(a+c\right)>0$ and $\Re\left(b-c\right)>0$, where
$L_{n}^{(\alpha)}\left(x;q\right)$ and $p_{n}\left(x,a\vert q\right)$
are the $q$-Laguerre and Little $q$-Laguerre polynomials respectively. 
\end{cor}

The next theorem gives the Plancherel-Rotach asymptotics of the polynomials $\{h_{m,n})z_1, z_2)\}$. 
\begin{thm}
\label{thm:h-asymptotics} For $a,b\in\mathbb{C}$ and $0<\epsilon<1$
the following  asymptotic result holds uniformly when $w_1, w_2$ are in compact subsets of the complex plane
\begin{eqnarray}
\begin{gathered}
\lim_{m,n \to \infty} 
\frac{h_{m,n}\left(w_{1}q^{-am-bn},w_{2}q^{-\left(1-a\right)m-\left(1-b\right)n}|q\right)\left(q;q\right)_{\infty}^{2}}{w_{1}^{m}w_{2}^{n}q^{\left(a-b\right)mn-am^{2}+\left(b-1\right)n^{2}}} 
=A_{q}\left(\frac{1}{w_{1}w_{2}}\right) 
\label{eqPRhmn}
\end{gathered}
\end{eqnarray}
\end{thm}
\begin{proof}
From the definition  \eqref{eq:hp2}  it follows that  
\begin{eqnarray*}
\begin{gathered}
 \frac{h_{m,n}\left(w_{1}q^{-am-bn},w_{2}q^{-\left(1-a\right)m-\left(1-b\right)n}|q\right)\left(q;q\right)_{\infty}^{2}}{w_{1}^{m}w_{2}^{n}q^{\left(a-b\right)mn-am^{2}+\left(b-1\right)n^{2}}}  \\
  =  \sum_{j=0}^{m\wedge n}\frac{\left(-w_{1}w_{2}\right)^{-j}q^{j^{2}}}{\left(q;q\right)_{j}}\left(q^{m-j+1},q^{n-j+1};q\right)_{\infty} \to 
 A_{q}\left(\frac{1}{w_{1}w_{2}}\right),
\end{gathered}
\end{eqnarray*}
as $m,n  \to \infty$. 
\end{proof}

\section{$2D$ $q$-Ultraspherical Polynomials}
The $2D$-ultraspherical  polynomials are 
 \begin{eqnarray}
 C_{m,n}^\nu(z_1, z_2) = \sum_{k=0}^{m \wedge n} \binom{m}{k}\binom{n}{k} k!(-1)^k 
 (\nu)_{m+n-k}z_1^{m-k} z_2^{n-k}, \quad \nu > -1.
 \label{eqde2dU}
 \end{eqnarray}
They are also known as the disk polynomials or the Zernike polynomials, \cite{Dun:Xu}.

 It is clear that 
 \begin{eqnarray}
  C_{m,n}^\nu(z_1, z_2) = (\nu)_{m+n}z_1^m z_2^n\, {}_{2}F_{1}\left(\begin{array}{cc}
\begin{array}{c}
-m, -n\\
-m-n-\nu +1
\end{array} \bigg| \frac{1}{z_1z_2}\end{array}\right)
 \end{eqnarray}
 which a constant multiple of the disk polynomials of \S 2.3 in \cite{Dun:Xu}. 
 
 They have the generating function 
 \begin{eqnarray}
 \label{eqGF2DU}
 \sum_{m,n=0}^\infty C_{m,n}^\nu(z_1, z_2) \frac{u^m\; v^n}{(m!\; n!} 
 = (1-uz_1-vz_2+uv)^{-\nu}, 
 \end{eqnarray}
 whose proof is an exercise in the application of the binomial theorem. Next we solve the connection relation between $C_{m,n}^\nu(z_1, z_2)$ and $H_{m,n}(z_1, z_2)$.  
 We claim that 
 \begin{eqnarray}
 \label{eqCmntoHmn}
 \bg 
 C_{m,n}^\nu(z_1, z_2)\\
  = \sum_{p=o}^{m \wedge n} \frac{(\nu)_{m+n}\; m! \; n!}{p!\; (m-p)!\; (n-p)!}
 H_{m-p, n-p}(z_1z_2)\,{}_1F_1(-p; -\nu -m-n +1; -1)
 \eg
 \end{eqnarray}
 
 Write the right-hand side of \eqref{eqGF2DU} as 
 \begin{eqnarray}
 \bg
 \int_0^\infty \frac{t^{\nu-1}}{\Gamma(\nu)}  e^{-t+ tuz_1+tvz_2-tuv} dt = 
  \int_0^\infty \frac{t^{\nu-1}}{\Gamma(\nu)}  e^{-t +t(t-1)uv}\sum_{r, s=0}^\infty H_{r,s}(z_1z_2) 
  \frac{u^r\;v^s}{r!\; s!} \; t^{r+s} dt \\
  = \sum_{r,s=0}^\infty H_{r,s}(z_1z_2) 
  \frac{u^r\;v^s}{r!\; s!} \sum_{j,k=0}^\infty  \frac{(-1)^j(uv)^{j+k}}{j!\; k!} 
\frac{ \Gamma(\nu+r+s+j + 2k)}{\Gamma(\nu)}. 
 \eg
 \notag
 \end{eqnarray}
 Equating coefficients of $u^mv^n$ we see that $m=r+j+k, n = s+j+k$. Let $p= j+k$. 
 Now \eqref{eqCmntoHmn} follows after some manipulations.    
 
 It is clear from the generating function \eqref{eqGF2DU}  that the disk polynomials have the convolution property 
 \begin{eqnarray}
 C_{m,n}^{\nu+\mu}(z_1, z_2) = \sum_{j=0}^m\sum_{k=0}^n 
 C_{j,k}^{\mu}(z_1, z_2)C_{m-j,n-k}^{\nu}(z_1, z_2)
 \end{eqnarray}
 
Floris \cite{Flo2}, see also \cite{Flo} and  \cite{Flo:Koe}, introduced the following $q$-analogue of the disk polynomials. For $\alpha,\beta>-1$ and $l,m\in\mathbb{Z}_{+}$, the Floris $q$-disk
polynomials $R_{l,m}^{\left(\alpha\right)}\left(z,z*;q^{2}\right)$
are defined by  
\[
R_{l,m}^{\left(\alpha\right)}\left(z,z*;q^{2}\right)=\begin{cases}
z^{l-m}P_{m}^{\left(\alpha,l-m\right)}\left(1-zz^{*}:q^{2}\right) & l\ge m\\
P_{m}^{\left(\alpha,m-l\right)}\left(1-zz^{*}:q^{2}\right)\left(z^{*}\right)^{m-l} & l\le m,
\end{cases}
\]
 where

\[
z^{*}z=q^{2}zz^{*}+1-q^{2},
\]
 and
\[
P_{m}^{\left(\alpha,\beta\right)}\left(x:q\right)=p_{m}\left(x;q^{\alpha},q^{\beta};q\right)
\]
is the little $q$-Jacobi polynomials. The $q$-disk polynomials $R_{l,m}^{\left(\alpha\right)}\left(z,z*;q^{2}\right)$
satisfy
\[
R_{l,m}^{\left(\alpha\right)}\left(z,z*;q^{2}\right)^{*}=R_{m,l}^{\left(\alpha\right)}\left(z,z*;q^{2}\right)
\]
 and the orthogonality relation 
\[
\begin{aligned} & \frac{1}{2\pi}\int_{0}^{1}\int_{0}^{2\pi}R_{l,m}^{\left(\alpha\right)}\left(e^{i\theta}z,e^{-i\theta}z*;q^{2}\right)^{*}R_{l',m'}^{\left(\alpha\right)}\left(e^{i\theta}z,e^{-i\theta}z*;q^{2}\right)d\theta\left(1-zz^{*}\right)^{\alpha}d_{q^{2}}\left(1-zz^{*}\right)\\
= & \frac{\left(1-q^{2}\right)\left(q^{2};q^{2}\right)_{l}\left(q^{2};q^{2}\right)_{m}q^{2m\left(\alpha+1\right)}\delta_{ll'}\delta_{mm'}}{\left(1-q^{2\left(\alpha+l+m+1\right)}\right)\left(q^{2\left(\alpha+1\right)};q^{2}\right)_{l}\left(q^{2\left(\alpha+1\right)};q^{2}\right)_{m}},
\end{aligned}
\]
for $\alpha>-1$, $l,l',m,m'\in\mathbb{Z}_{+}$.

We now introduce our  $q$-analogue of these polynomials.  
For $0< q <1$ and $b < q^{-1}$, let us define
\begin{equation}
\begin{aligned}p_{m,n}\left(z_{1},z_{2};b\vert q\right) & =\sum_{k=0}^{\infty}\gauss{m}{k}\gauss{n}{k}\frac{q^{\binom{k}{2}}\left(q;q\right)_{k}\left(bq;q\right)_{m+n-k}}{\left(-1\right)^{k}z_{1}^{k-m}z_{2}^{k-n}}\end{aligned}
,\label{eq:jpdefinition}
\end{equation}
it is clear that
\begin{equation}
p_{m,n}\left(z_{2},z_{1};b\vert q\right)=p_{n,m}\left(z_{1},z_{2};b\vert q\right)\label{eq:jpsymmetry}
\end{equation}
then for $m\ge n$ we have
\begin{equation}
\begin{aligned}p_{m,n}\left(z_{1},z_{2};b\vert q\right) & =\left(-1\right)^{n}q^{\binom{n}{2}}\left(bq;q\right)_{m}\left(q^{m-n+1};q\right)_{n}z_{1}^{m-n}p_{n}\left(z_{1}z_{2};q^{m-n},b\vert q\right),\end{aligned}
\label{eq:p2j}
\end{equation}
where $p_{n}\left(x;a,b\vert q\right)$ is the little Jacobi polynomials.
\begin{thm}
\label{thm:jporthogonality} For $0<q<1$ and $b<q^{-1}$, the polynomials
$\left\{ p_{m,n}\left(z,\overline{z};b\vert q\right)\right\} $ satisfy
the following orthogonality

\begin{equation}
\begin{aligned}\int_{\mathbb{C}}p_{m,n}\left(z,\overline{z};b\vert q\right)\overline{p_{s,t}\left(z,\overline{z};b\vert q\right)}d\mu\left(z,\overline{z}\right) & =\frac{\left(bq;q\right)_{\infty}}{\left(q;q\right)_{\infty}}\frac{q^{mn}\left(q,bq;q\right)_{m}\left(q,bq;q\right)_{n}}{1-bq^{m+n+1}}\delta_{m,s}\delta_{n,t},\end{aligned}
\label{eq:jporthogonality}
\end{equation}
where
\begin{equation}
d\mu\left(z,\overline{z}\right)=\frac{d\theta}{2\pi}\otimes\sum_{k=0}^{\infty}\frac{\left(bq;q\right)_{k}q^{k}}{\left(q;q\right)_{k}}\delta\left(r-q^{k/2}\right),\label{eq:jpmeasure}
\end{equation}
$z=re^{i\theta},r\in\mathbb{R}^{+},\theta\in[0,2\pi]$ and $m,n,s,t\in\mathbb{N}_{0}$.\end{thm}
\begin{proof}
Because of the symmetry \eqref{eq:jpsymmetry}, we may assume that
$m\ge n$. We first use polar coordinates, then apply \eqref{eq:p2j}
and the orthogonality relation of the little Jacobi polynomials to
get

\[
\begin{aligned} & \int_{\mathbb{C}}p_{m,n}\left(z,\overline{z};b\vert q\right)\overline{p_{s,t}\left(z,\overline{z};b\vert q\right)}d\mu\left(z,\overline{z}\right)\\
 & =\left(-1\right)^{n+t}q^{\binom{n}{2}+\binom{t}{2}}\left(bq;q\right)_{m}\left(bq;q\right)_{s}\left(q^{m-n+1};q\right)_{n}\left(q^{s-t+1};q\right)_{t}\\
 & \times\int_{0}^{2\pi}e^{i\theta\left(m-n+t-s\right)}\frac{d\theta}{2\pi}\sum_{k=0}^{\infty}\frac{\left(bq;q\right)_{k}q^{k}r^{m-n+s-t}}{\left(q;q\right)_{k}}\\
 & \times\delta\left(r-q^{k/2}\right)p_{n}\left(r^{2};q^{m-n},b\big|q\right)p_{t}\left(r^{2};q^{s-t},b\big|q\right)\\
 & =\left(-1\right)^{n+t}q^{\binom{n}{2}+\binom{t}{2}}\left(bq;q\right)_{m}\left(bq;q\right)_{s}\left(q^{m-n+1};q\right)_{n}\left(q^{m-n+1};q\right)_{t}\\
 & \times\sum_{k=0}^{\infty}\frac{\left(bq;q\right)_{k}q^{k\left(m-n+1\right)}}{\left(q;q\right)_{k}}p_{n}\left(q^{k};q^{m-n},b\big|q\right)p_{t}\left(q^{k};q^{m-n},b\big|q\right)\delta_{m-n+t-s,0}\\
 & =\frac{\left(bq;q\right)_{\infty}}{\left(q;q\right)_{\infty}}\frac{q^{mn}\left(q,bq;q\right)_{m}\left(q,bq;q\right)_{n}}{1-bq^{m+n+1}}\delta_{m,s}\delta_{n,t}.
\end{aligned}
\]
This completes the proof. 
\end{proof}
\begin{thm}
\label{thm:jp generating function} The polynomials $\left\{ p_{m,n}\left(z_{1},z_{2};b\vert q\right)\right\} $
have  the generating function 
\begin{equation}
\begin{aligned}\sum_{m,n=0}^{\infty}\frac{p_{m,n}\left(z_{1},z_{2};b\vert q\right)}{\left(q;q\right)_{m}\left(q;q\right)_{n}}u^{m}v^{n} & =\frac{\left(bq,uv;q\right)_{\infty}}{\left(uz_{1},z_{2}v;q\right)_{\infty}}{}_{2}\phi_{1}\left(\begin{array}{cc}
\begin{array}{c}
uz_{1},vz_{2}\\
uv
\end{array} & \bigg|q;bq\end{array}\right). 
\end{aligned}
.\label{eq:jp generating function}
\end{equation}
For $bq,cq<1$ and $b\neq0$, the connection relation between the $q-2D$ ultraspherical polynomials is given by
\begin{equation}
\frac{p_{m,n}\left(z_{1},z_{2};b\vert q\right)}{\left(bq;q\right)_{\infty}}=\sum_{j=0}^{\infty}\frac{\left(\frac{c}{b};q\right)_{j}}{\left(q;q\right)_{j}}\left(bq^{\frac{m+n}{2}+1}\right)^{j}\frac{p_{m,n}\left(z_{1}q^{\frac{j}{2}},z_{2}q^{\frac{j}{2}};c\big|q\right)}{\left(cq;q\right)_{\infty}}.
\end{equation}
The connection relation between the $q-2D$ ultraspherical and $q-2D$ Hermite is given by 
\begin{equation}
\begin{aligned}\frac{p_{m,n}\left(z_{1},z_{2};b\vert q\right)}{\left(bq;q\right)_{\infty}} & =\sum_{j=0}^{\infty}\frac{\left(bq^{(m+n)/2+1}\right)^{j}}{\left(q;q\right)_{j}}H_{m,n}\left(z_{1}q^{j/2}, 
z_{2}q^{j/2}\vert q\right).\end{aligned}
\label{eq:jp p2H}
\end{equation}
Moreover we have the inverse relation
\begin{eqnarray}
H_{m,n}\left(z_{1},z_{2}\vert q\right)=\sum_{k=0}^{\infty}\frac{\left(-bq^{(m+n)/2+1}\right)^{k}}{\left(bq;q\right)_{\infty}\left(q;q\right)_{k}}q^{\binom{k}{2}}p_{m,n}\left(z_{1}q^{k/2},z_{2}q^{k/2};b\vert q\right).
\label{eqcinnhtop}
\end{eqnarray}
\end{thm}
\begin{proof}
Using the explicit definition  \eqref{eq:p2j}  we see that 
\[
\begin{aligned} & \sum_{m,n=0}^{\infty}\frac{p_{m,n}\left(z_{1},z_{2};b\vert q\right)}{\left(q;q\right)_{m}\left(q;q\right)_{n}}u^{m}v^{n}\\
 & =\sum_{k=0}^{\infty}\frac{q^{\binom{k}{2}}\left(-uv\right)^{k}}{\left(q;q\right)_{k}}\sum_{m,n\ge k}\frac{\left(bq;q\right)_{m+n-k}\left(z_{1}u\right)^{m-k}\left(z_{2}v\right)^{n-k}}{\left(q;q\right)_{m-k}\left(q;q\right)_{n-k}}\\
 & =\sum_{k=0}^{\infty}\frac{q^{\binom{k}{2}}\left(-uv\right)^{k}}{\left(q;q\right)_{k}}\sum_{m,n=0}^{\infty}\frac{\left(bq;q\right)_{m+n+k}\left(z_{1}u\right)^{m}\left(z_{2}v\right)^{n}}{\left(q;q\right)_{m}\left(q;q\right)_{n}}\\
 & =\sum_{k=0}^{\infty}\frac{q^{\binom{k}{2}}\left(-uv\right)^{k}}{\left(q;q\right)_{k}}\sum_{m,n=0}^{\infty}\frac{\left(z_{1}u\right)^{m}\left(z_{2}v\right)^{n}}{\left(q;q\right)_{m}\left(q;q\right)_{n}}\frac{\left(bq;q\right)_{\infty}}{\left(bq^{m+n+k+1};q\right)_{\infty}}.
 \end{aligned}
\]
 We then expand $1/(bq^{m+n+k+1};q)_{\infty}$ using \eqref{eqEuler2} and write  the 
 above expression as  
\[
\begin{aligned}
 & =\left(bq;q\right)_{\infty}\sum_{j=0}^{\infty}\frac{\left(bq\right)^{j}}{\left(q;q\right)_{j}}\sum_{m,n=0}^{\infty}\frac{\left(z_{1}uq^{j}\right)^{m}\left(z_{2}vq^{j}\right)^{n}}{\left(q;q\right)_{m}\left(q;q\right)_{n}}\sum_{k=0}^{\infty}\frac{q^{\binom{k}{2}}\left(-uvq^{j}\right)^{k}}{\left(q;q\right)_{k}}\\
 & =\left(bq;q\right)_{\infty}\sum_{j=0}^{\infty}\frac{\left(bq\right)^{j}}{\left(q;q\right)_{j}}\sum_{m,n=0}^{\infty}\frac{\left(z_{1}uq^{j}\right)^{m}\left(z_{2}vq^{j}\right)^{n}}{\left(q;q\right)_{m}\left(q;q\right)_{n}}\left(uvq^{j};q\right)_{\infty}\\
 & =\left(bq,uv;q\right)_{\infty}\sum_{j=0}^{\infty}\frac{\left(bq\right)^{j}}{\left(q,uv;q\right)_{j}}\frac{1}{\left(uz_{1}q^{j},z_{2}vq^{j};q\right)_{\infty}}\\
 & =\frac{\left(bq,uv;q\right)_{\infty}}{\left(uz_{1},z_{2}v;q\right)_{\infty}}\sum_{j=0}^{\infty}\frac{\left(uz_{1},z_{2}v;q\right)_{j}\left(bq\right)^{j}}{\left(q,uv;q\right)_{j}}
 =\frac{\left(bq,uv;q\right)_{\infty}}{\left(uz_{1},z_{2}v;q\right)_{\infty}}{}_{2}\phi_{1}\left(\begin{array}{cc}
\begin{array}{c}
uz_{1},vz_{2}\\
uv
\end{array} & \bigg|q;bq\end{array}\right).
\end{aligned}
\]
 From above calculations we see that  the generating function could be also written as
\[
\begin{aligned}\sum_{m,n=0}^{\infty}\frac{p_{m,n}\left(z_{1},z_{2};b\vert q\right)}{\left(q;q\right)_{m}\left(q;q\right)_{n}}u^{m}v^{n} & =\left(bq;q\right)_{\infty}\sum_{j=0}^{\infty}\frac{\left(bq\right)^{j}}{\left(q;q\right)_{j}}\frac{\left(uvq^{j};q\right)_{\infty}}{\left(uz_{1}q^{j},z_{2}vq^{j};q\right)_{\infty}}\end{aligned}
\]
and we find that 
\[
\begin{aligned} & \sum_{m,n=0}^{\infty}\frac{p_{m,n}\left(z_{1},z_{2};b\vert q\right)}
{\left(q;q\right)_{m}\left(q;q\right)_{n}}u^{m}v^{n}\\
 & =\left(bq;q\right)_{\infty}\sum_{m,n=0}^{\infty}\frac{u^{m}v^{n}}{\left(q;q\right)_{m} 
 \left(q;q\right)_{n}} 
 \sum_{j=0}^{\infty}\frac{\left(bq^{(m+n)/2+1}\right)^{j}}{\left(q;q\right)_{j}}H_{m,n}\left(z_{1}
 q^{j/2},z_{2}q^{j/2}\vert q\right)
\end{aligned}
\]
which gives \eqref{eq:jp p2H}. We now prove \eqref{eqcinnhtop}. From 
\[
\sum_{k=0}^{n}\gauss{n}{k}\left(-1\right)^{k}q^{\binom{k}{2}}=\delta_{n,0}
\]
to get
\begin{eqnarray*}
 &  & \frac{1}{\left(bq;q\right)_{\infty}} 
 \sum_{k=0}^{\infty}\frac{\left(-bq^{(m+n)/2+1}\right)^{k}}{\left(q;q\right)_{k}}q^{\binom{k}{2}}
 p_{m,n}\left(z_{1}q^{k/2},z_{2}q^{k/2};b\vert q\right)\\
 & = & \sum_{j,k=0}^{\infty}\frac{\left(-bq^{(m+n)/2+1}\right)^{k+j}}{\left(q;q\right)_{k} 
 \left(q;q\right)_{j}}q^{\binom{k}{2}}H_{m,n}\left(z_{1}q^{\left(k+j\right)/2},z_{2}
 q^{\left(k+j\right)/2}\vert q\right)\\
 & = & \sum_{\ell=0}^{\infty}\frac{\left(-bq^{(m+n)/2+1}\right)^{\ell}}{\left(q;q\right)_{\ell}} 
 H_{m,n}\left(z_{1}q^{\ell/2},z_{2}q^{\ell/2}\vert q\right)\sum_{k=0}^{\ell}\gauss{\ell}{k}
 \left(-1\right)^{k}q^{\binom{k}{2}}\\
 & = & H_{m,n}\left(z_{1},z_{2}\vert q\right).
\end{eqnarray*}
The connection formula between $p_{m,n}(z_1,z_2;b|q)$ polynomials can be proved directly 
by observing  that
\begin{equation*}
\begin{aligned}p_{m,n}\left(z_{1},z_{2};b\vert q\right) & =\sum_{k=0}^{\infty}\gauss{m}{k}
\gauss{n}{k}q^{\binom{k}{2}}\left(q;q\right)_{k}\left(-1\right)^{k}z_{1}^{m-k}z_{2}^{n-k}\\
\times & \left(cq;q\right)_{m+n-k}\frac{\left(bq;q\right)_{\infty}}{\left(cq;q\right)_{\infty}}
\frac{\left(cq^{m+n-k+1};q\right)_{\infty}}{\left(bq^{m+n-k+1};q\right)_{\infty}}\\
= & \sum_{k=0}^{\infty}\gauss{m}{k}\gauss{n}{k}q^{\binom{k}{2}}\left(q;q\right)_{k} 
\left(-1\right)^{k}z_{1}^{m-k}z_{2}^{n-k}\\
\times & \left(cq;q\right)_{m+n-k}\frac{\left(bq;q\right)_{\infty}}{\left(cq;q\right)_{\infty}}
\sum_{j=0}^{\infty}\frac{\left(\frac{c}{b};q\right)_{j}}{\left(q;q\right)_{j}}\left(bq^{m+n-k+1}\right)^{j}\\
= & \frac{\left(bq;q\right)_{\infty}}{\left(cq;q\right)_{\infty}}\sum_{j=0}^{\infty} 
\frac{\left(\frac{c}{b};q\right)_{j}}{\left(q;q\right)_{j}}\left(bq^{\frac{m+n}{2}+1}\right)^{j} 
\sum_{k=0}^{\infty}\gauss{m}{k}\gauss{n}{k}\\
\times & q^{\binom{k}{2}}\left(q;q\right)_{k}\left(-1\right)^{k}\left(cq;q\right)_{m+n-k}
\left(z_{1}q^{\frac{j}{2}}\right)^{m-k}\left(z_{2}q^{\frac{j}{2}}\right)^{n-k}\\
= & \frac{\left(bq;q\right)_{\infty}}{\left(cq;q\right)_{\infty}}\sum_{j=0}^{\infty}
\frac{\left(\frac{c}{b};q\right)_{j}}{\left(q;q\right)_{j}}\left(bq^{\frac{m+n}{2}+1}\right)^{j} 
p_{m,n}\left(z_{1}q^{\frac{j}{2}},z_{2}q^{\frac{j}{2}};c\big|q\right).
\end{aligned}
\end{equation*}
 This completes the proof of our theorem. 
\end{proof}

Let us rewrite  \eqref{eq:jp generating function} in the form 
\begin{equation}
_{2}\phi_{1}\left(\begin{array}{cc}
\begin{array}{c}
uz_{1},vz_{2}\\
uv
\end{array} & \bigg|q;bq\end{array}\right)=\frac{\left(uz_{1},z_{2}v;q\right)_{\infty}}{\left(bq,uv;q\right)_{\infty}}\sum_{m,n=0}^{\infty}\frac{p_{m,n}\left(z_{1},z_{2};b\vert q\right)}{\left(q;q\right)_{m}\left(q;q\right)_{n}}u^{m}v^{n}.\label{eq:jp generating function d}.
\end{equation}

\begin{thm}
\label{thm:jp properties} The polynomials $\left\{ p_{m,n}\left(z_{1},z_{2};b\vert q\right)\right\} $
satisfy the following properties, 

\begin{equation}
\begin{aligned}D_{q,z_{1}}p_{m,n}\left(z_{1},z_{2};b\vert q\right) & 
=\frac{\left(1-bq\right)}{1-q}\left(1-q^{m}\right)p_{m-1,n}\left(z_{1},z_{2};bq\vert q\right),\end{aligned}
\label{eq:jp forward 1}
\end{equation}
\begin{equation}
\begin{aligned}D_{q,z_{2}}p_{m,n}\left(z_{1},z_{2};b\vert q\right) 
& =\frac{\left(1-bq\right)}{1-q}\left(1-q^{n}\right)p_{m,n-1}\left(z_{1},z_{2};bq\vert q\right),
\end{aligned}
\label{eq:jp forward 2}
\end{equation}
 
\begin{equation}
\begin{aligned}D_{q^{-1},z_{1}}\left(\frac{\left(qz_{1}z_{2};q\right)_{\infty}}
{\left(bqz_{1}z_{2};q\right)_{\infty}}p_{m,n}\left(z_{1},z_{2};b\vert q\right)\right) &  
 =\frac{\left(qz_{1}z_{2};q\right)_{\infty}p_{m,n+1}\left(z_{1},z_{2};bq^{-1}\vert q\right)}
 {q^{m-1}\left(q-1\right)\left(bz_{1}z_{2};q\right)_{\infty}},\end{aligned}
\label{eq:jp backward 1}
\end{equation}
 
\begin{equation}
\begin{aligned}D_{q^{-1},z_{2}}\left(\frac{\left(qz_{1}z_{2};q\right)_{\infty}}
{\left(bqz_{1}z_{2};q\right)_{\infty}}p_{m,n}\left(z_{1},z_{2};b\vert q\right)\right) 
& =\frac{\left(qz_{1}z_{2};q\right)_{\infty}p_{m+1,n}\left(z_{1},z_{2};bq^{-1}\vert q\right)}
{q^{n-1}\left(q-1\right)\left(bz_{1}z_{2};q\right)_{\infty}},\end{aligned}
\label{eq:jp backward 2}
\end{equation}

\begin{equation}
\begin{aligned} & p_{m,n}\left(z_{1}q,z_{2};b\vert q\right)-bq^{n+m-1}
\left(1-q^{m}\right)\left(1-q^{n}\right)p_{m-1,n-1}\left(z_{1}q,z_{2};b\vert q\right)\\
& =p_{m,n}\left(z_{1},z_{2};b\vert q\right)q^{m}-q^{m-1}\left(1-q^{m}\right)\left(1-q^{n}\right)p_{m-1,n-1}\left(z_{1},z_{2};b\vert q\right),
\end{aligned}
\label{eq:jp properties 17 a}
\end{equation}
 
\begin{equation}
\begin{aligned} & p_{m,n}\left(z_{1}q,z_{2};b\vert q\right)-bq^{2m-1}\left(1-q^{m}\right)\left(1-q^{n}\right)p_{m-1,n-1}\left(z_{1},z_{2}q;b\vert q\right)\\
& = p_{m,n}\left(z_{1},z_{2};b\vert q\right)-q^{m-1}\left(1-q^{m}\right)\left(1-q^{n}\right)p_{m-1,n-1}\left(z_{1},z_{2};b\vert q\right),
\end{aligned}
\label{eq:jp properties 17 b}
\end{equation}

\begin{equation}
\begin{aligned} & p_{m,n}\left(z_{1},z_{2}q;b\vert q\right)-bq^{2n-1}\left(1-q^{m}\right)\left(1-q^{n}\right)p_{m-1,n-1}\left(z_{1}q,z_{2};b\vert q\right)\\
 &= q^{n}p_{m,n}\left(z_{1},z_{2};b\vert q\right)-q^{n-1}\left(1-q^{m}\right)\left(1-q^{n}\right)p_{m-1,n-1}\left(z_{1},z_{2};b\vert q\right),
\end{aligned}
\label{eq:jp properties 17 c}
\end{equation}

\begin{equation}
\begin{aligned} & p_{m,n}\left(z_{1},z_{2}q;b\vert q\right)-bq^{m+n}\left(1-q^{m}\right)\left(1-q^{n}\right)p_{m-1,n-1}\left(z_{1},z_{2}q;b\vert q\right)\\
& = q^{n}p_{m,n}\left(z_{1},z_{2};b\vert q\right)-q^{n-1}\left(1-q^{m}\right)\left(1-q^{n}\right)p_{m-1,n-1}\left(z_{1},z_{2};b\vert q\right).
\end{aligned}
\label{eq:jp properties 17 d}
\end{equation}
 
\begin{equation}
\begin{aligned} & p_{m,n}\left(z_{1}q,z_{2};b\vert q\right)-bq^{n+1}z_{1}\left(1-q^{m}\right)p_{m-1,n}\left(z_{1}q,z_{2};b\vert q\right)\\
 &= p_{m,n}\left(z_{1},z_{2};b\vert q\right)-z_{1}\left(1-q^{m}\right)p_{m-1,n}\left(z_{1},z_{2};b\vert q\right),
\end{aligned}
\label{eq:jp properties 18 a}
\end{equation}
 
\begin{equation}
\begin{aligned} & p_{m,n}\left(z_{1}q,z_{2};b\vert q\right)-bq^{m}z_{1}\left(1-q^{m}\right)p_{m-1,n}\left(z_{1},z_{2}q;b\vert q\right)\\
 & =p_{m,n}\left(z_{1},z_{2};b\vert q\right)-z_{1}\left(1-q^{m}\right)p_{m-1,n}\left(z_{1},z_{2};b\vert q\right),
\end{aligned}
\label{eq:jp properties 18 b}
\end{equation}
\begin{equation}
\begin{aligned} & \left(1+bq^{m+n}z_{1}\right)p_{m,n}\left(z_{1},z_{2};b\vert q\right)\\
 & =z_{1}p_{m-1,n}\left(z_{1},z_{2};b\vert q\right)-q^{m-1}\left(1-q^{n}\right)\left(1-bq^{n}\right)p_{m-1,n-1}\left(z_{1},z_{2};b\vert q\right),
\end{aligned}
\label{eq:jp properties 19}
\end{equation}

\begin{equation}
\begin{aligned} & \left(1-q^{m-n}\right)p_{m+1,n+1}\left(z_{1},z_{2};b\vert q\right)\\
 &= z_{1}\left(1-bq^{m+1}\right)\left(1-q^{m+1}\right)p_{m,n+1}\left(z_{1},z_{2};b\vert q\right)\\
 &- z_{2}q^{m-n}\left(1-bq^{n+1}\right)\left(1-q^{n+1}\right)p_{m+1,n}\left(z_{1},z_{2};b\vert q\right),
\end{aligned}
\label{eq:jp properties 20 a}
\end{equation}
 
\begin{equation}
\begin{aligned} & p_{m+1,n+1}\left(z_{1}q,z_{2};b\vert q\right)-p_{m+1,n+1}\left(z_{1},z_{2}q;b\vert q\right)\\
 & =z_{2}\left(1-bq^{m+2}\right)\left(1-q^{n+1}\right)p_{m+1,n}\left(z_{1}q,z_{2};b\vert q\right)\\
 &- z_{1}\left(1-bq^{m+1}\right)\left(1-q^{m+1}\right)p_{m,n+1}\left(z_{1},z_{2}q;b\vert q\right),
\end{aligned}
\label{eq:jp properties 20 b}
\end{equation}
\begin{equation}
\begin{aligned} & z_{2}\left(1-q^{n}\right)p_{m,n-1}\left(z_{1}q,z_{2};b\vert q\right)-z_{1}\left(1-q^{m}\right)p_{m-1,n}\left(z_{1},z_{2}q;b\vert q\right)\\
 &= z_{2}\left(1-q^{n}\right)p_{m,n-1}\left(z_{1},z_{2};b\vert q\right)-z_{1}\left(1-q^{m}\right)p_{m-1,n}\left(z_{1},z_{2};b\vert q\right),
\end{aligned}
\label{eq:jp properties 21}
\end{equation}
\begin{equation}
\begin{aligned} & qz_{1}z_{2}\left(p_{m,n}\left(z_{1}q,z_{2};b\vert q\right)-p_{m,n}\left(z_{1},z_{2}q;b\vert q\right)\right)\\
 &+ z_{1}z_{2}\left(1-q^{m}\right)\left(1-q^{n}\right)\left(p_{m-1,n-1}\left(z_{1}q,z_{2};b\vert q\right)-p_{m-1,n-1}\left(z_{1},z_{2}q;b\vert q\right)\right)\\
 &= qz_{1}z_{2}^{2}\left(1-q^{n}\right)\left(p_{m,n-1}\left(z_{1}q,z_{2};b\vert q\right)-bp_{m,n-1}\left(z_{1}q,z_{2}q;b\vert q\right)\right)\\
 &- qz_{1}^{2}z_{2}\left(1-q^{m}\right)\left(p_{m-1,n}\left(z_{1},z_{2}q;b\vert q\right)-bp_{m-1,n}\left(z_{1}q,z_{2}q;b\vert q\right)\right)\\
 &+ z_{1}\left(1-q^{m}\right)\left(p_{m-1,n}\left(z_{1}q,z_{2};b\vert q\right)-p_{m-1,n}\left(z_{1}q,z_{2}q;b\vert q\right)\right)\\
 &- z_{2}\left(1-q^{n}\right)\left(p_{m,n-1}\left(z_{1},z_{2}q;b\vert q\right)-p_{m,n-1}\left(z_{1}q,z_{2}q;b\vert q\right)\right).
\end{aligned}
\label{eq:jp properties 22}
\end{equation}
 \end{thm}
\begin{proof}
Applying  the difference operator $D_{q,z_{1}}$ to the form \eqref{eq:jp generating function d} of the generating function   we find that 

\begin{eqnarray*}
&{}& \sum_{m,n=0}^{\infty}\frac{D_{q,z_{1}}p_{m,n}\left(z_{1},z_{2};b\vert q\right)}{\left(q;q\right)_{m}\left(q;q\right)_{n}\left(bq;q\right)_{\infty}}u^{m}v^{n}  =  \sum_{j=0}^{\infty}\frac{\left(bq\right)^{j}}{\left(q;q\right)_{j}}\frac{\left(uvq^{j};q\right)_{\infty}}{\left(z_{2}vq^{j};q\right)_{\infty}}D_{q,z_{1}}\frac{1}{\left(z_{1}uq^{j};q\right)_{\infty}}\\
 & = & \frac{u}{1-q}\sum_{j=0}^{\infty}\frac{\left(bq^{2}\right)^{j}}{\left(q;q\right)_{j}}\frac{\left(uvq^{j};q\right)_{\infty}}{\left(uz_{1}q^{j},z_{2}vq^{j};q\right)_{\infty}} 
  =  \frac{u}{1-q}\sum_{m,n=0}^{\infty}\frac{p_{m,n}\left(z_{1},z_{2};bq\vert q\right)}{\left(bq^{2};q\right)_{\infty}\left(q;q\right)_{m}\left(q;q\right)_{n}}u^{m}v^{n}
\end{eqnarray*}
to get \eqref{eq:jp forward 1}. \eqref{eq:jp forward 2} is obtained
similarly. On the other hand, from \eqref{eq:p2j} and the backward
shift operator we get \eqref{eq:jp backward 1} and by the symmetry
\eqref{eq:jpsymmetry} we have \eqref{eq:jp backward 2}.From the
Heine's contiguous relation (17.6.17) in \cite{andrews-nist} we get
\begin{equation}
\begin{aligned} & _{2}\phi_{1}\left(\begin{array}{cc}
\begin{array}{c}
uq^{-1}z_{1}q,vz_{2}\\
uq^{-1}v
\end{array} & \bigg|q;bq\end{array}\right)-{}_{2}\phi_{1}\left(\begin{array}{cc}
\begin{array}{c}
uz_{1},vz_{2}\\
uv
\end{array} & \bigg|q;bq\end{array}\right)\\
 & =uvb\frac{\left(1-uz_{1}\right)\left(1-vz_{2}\right)}{\left(1-uvq^{-1}\right)\left(1-uv\right)}{}_{2}\phi_{1}\left(\begin{array}{cc}
\begin{array}{c}
uz_{1}q,vqz_{2}\\
uvq
\end{array} & \bigg|q;bq\end{array}\right),
\end{aligned}
\label{eq:jp 17 a}
\end{equation}
 
\begin{equation}
\begin{aligned} & _{2}\phi_{1}\left(\begin{array}{cc}
\begin{array}{c}
uq^{-1}z_{1}q,vz_{2}\\
uq^{-1}v
\end{array} & \bigg|q;bq\end{array}\right)-{}_{2}\phi_{1}\left(\begin{array}{cc}
\begin{array}{c}
uz_{1},vz_{2}\\
uv
\end{array} & \bigg|q;bq\end{array}\right)\\
 & =uvb\frac{\left(1-uz_{1}\right)\left(1-vz_{2}\right)}{\left(1-uvq^{-1}\right)\left(1-uv\right)}{}_{2}\phi_{1}\left(\begin{array}{cc}
\begin{array}{c}
uqz_{1},vz_{2}q\\
uqv
\end{array} & \bigg|q;bq\end{array}\right),
\end{aligned}
\label{eq:jp 17 b}
\end{equation}
 
\begin{equation}
\begin{aligned} & _{2}\phi_{1}\left(\begin{array}{cc}
\begin{array}{c}
uz_{1},vq^{-1}z_{2}q\\
uvq^{-1}
\end{array} & \bigg|q;bq\end{array}\right)-{}_{2}\phi_{1}\left(\begin{array}{cc}
\begin{array}{c}
uz_{1},vz_{2}\\
uv
\end{array} & \bigg|q;bq\end{array}\right)\\
 & =uvb\frac{\left(1-uz_{1}\right)\left(1-vz_{2}\right)}{\left(1-uvq^{-1}\right)\left(1-uv\right)}{}_{2}\phi_{1}\left(\begin{array}{cc}
\begin{array}{c}
uz_{1}q,vqz_{2}\\
uvq
\end{array} & \bigg|q;bq\end{array}\right)
\end{aligned}
\label{eq:jp 17 c}
\end{equation}
and
\begin{equation}
\begin{aligned} & _{2}\phi_{1}\left(\begin{array}{cc}
\begin{array}{c}
uz_{1},vq^{-1}z_{2}q\\
uvq^{-1}
\end{array} & \bigg|q;bq\end{array}\right)-{}_{2}\phi_{1}\left(\begin{array}{cc}
\begin{array}{c}
uz_{1},vz_{2}\\
uv
\end{array} & \bigg|q;bq\end{array}\right)\\
 & =uvb\frac{\left(1-uz_{1}\right)\left(1-vz_{2}\right)}{\left(1-uvq^{-1}\right)\left(1-uv\right)}{}_{2}\phi_{1}\left(\begin{array}{cc}
\begin{array}{c}
uqz_{1},vqz_{2}\\
uqv
\end{array} & \bigg|q;bq\end{array}\right).
\end{aligned}
\label{eq:jp 17 d}
\end{equation}
 Applying \eqref{eq:jp generating function d} we get \eqref{eq:jp properties 17 a}
from \eqref{eq:jp 17 a}, \eqref{eq:jp properties 17 b} from \eqref{eq:jp 17 b}, 
\eqref{eq:jp properties 17 c} from \eqref{eq:jp 17 c} and \eqref{eq:jp properties 17 d}
from \eqref{eq:jp 17 d}. From Heine's contiguous relation \cite[(17.6.18)]{andrews-nist}

From the contiguous relation (17.6.18) in \cite{andrews-nist} we get
\begin{eqnarray}
\bg
{}_{2}\phi_{1}\left(\begin{array}{cc}
\begin{array}{c}
uz_{1}q,vz_{2}\\
uv
\end{array}   \bigg|q;bq\end{array}\right)-  _{2}\phi_{1}\left(\begin{array}{cc}
\begin{array}{c}
uz_{1},vz_{2}\\
uv
\end{array}   \bigg|q;bq\end{array}\right)   \qquad \qquad  \\
 \qquad \qquad =uz_{1}bq\frac{1-vz_{2}}{1-uv}{}_{2}\phi_{1}\left(\begin{array}{cc}
\begin{array}{c}
uz_{1}q,vqz_{2}\\
uvq
\end{array}  \bigg|q;bq\end{array}\right) 
\eg\label{eq:jp 18 a}
\end{eqnarray}
and
\begin{eqnarray}
\bg
{}_{2}\phi_{1}\left(\begin{array}{cc}
\begin{array}{c}
uz_{1}q,vz_{2}\\
uv
\end{array}  \bigg|q;bq\end{array}\right)-  {}_{2}\phi_{1}\left(\begin{array}{cc}
\begin{array}{c}
uz_{1},vz_{2}\\
uv
\end{array}   \bigg|q;bq\end{array}\right)\qquad \qquad  \\
\qquad  \qquad =uz_{1}bq\frac{1-vz_{2}}{1-uv}{}_{2}\phi_{1}\left(\begin{array}{cc}
\begin{array}{c}
uqz_{1},vz_{2}q\\
uqv
\end{array}   \bigg|q;bq\end{array}\right).
\eg
\label{eq:jp 18 b}
\end{eqnarray}
 Applying \eqref{eq:jp generating function d} to \eqref{eq:jp properties 20 a}
to get \eqref{eq:jp properties 18 a}, applying \eqref{eq:jp generating function d}
to \eqref{eq:jp 18 b} to get \eqref{eq:jp properties 18 b}. 

From the fourth contiguous relation \cite[(17.6.19)]{andrews-nist}  we get
\begin{eqnarray}
\bg
{}_{2}\phi_{1}\left(\begin{array}{cc}
\begin{array}{c}
uqz_{1},vz_{2}\\
uqv
\end{array} & \bigg|q;bq\end{array}\right)-{}_{2}\phi_{1}\left(\begin{array}{cc}
\begin{array}{c}
uz_{1},vz_{2}\\
uv
\end{array}   \bigg|q;bq\end{array}\right) \\
=bq\frac{\left(1-vz_{2}\right)\left(uz_{1}-uv\right)}{\left(1-uv\right)\left(1-uvq\right)}{}_{2}\phi_{1}\left(\begin{array}{cc}
\begin{array}{c}
uqz_{1},vqz_{2}\\
uqvq
\end{array} & \bigg|q;bq\end{array}\right) 
,\label{eq:jp 19}
\eg
\end{eqnarray}
 we apply \eqref{eq:jp generating function d} to \eqref{eq:jp 19}
to get \eqref{eq:jp properties 19}.

From the fourth contiguous relation \cite[(17.6.20)]{andrews-nist} we get 
\begin{eqnarray}
\bg
{}_{2}\phi_{1}\left(\begin{array}{cc}
\begin{array}{c}
uqz_{1},vq^{-1}z_{2}\\
uv
\end{array} & \bigg|q;bq\end{array}\right)-{}_{2}\phi_{1}\left(\begin{array}{cc}
\begin{array}{c}
uz_{1},vz_{2}\\
uv
\end{array}   \bigg|q;bq\end{array}\right)  \\
=b\frac{\left(uz_{1}q-vz_{2}\right)}{\left(1-uv\right)}{}_{2}\phi_{1}\left(\begin{array}{cc}
\begin{array}{c}
uqz_{1},vz_{2}\\
uqv
\end{array}   \bigg|q;bq\end{array}\right) 
\label{eq:jp 20 a}
\eg
\end{eqnarray}
and
\begin{eqnarray}
\bg
{}_{2}\phi_{1}\left(\begin{array}{cc}
\begin{array}{c}
uz_{1}q,vz_{2}q^{-1}\\
uv
\end{array}   \bigg|q;bq\end{array}\right)-{}_{2}\phi_{1}\left(\begin{array}{cc}
\begin{array}{c}
uz_{1},vz_{2}\\
uv
\end{array}   \bigg|q;bq\end{array}\right) \\
 =b\frac{\left(uz_{1}q-vz_{2}\right)}{\left(1-uv\right)}{}_{2}\phi_{1}\left(\begin{array}{cc}
\begin{array}{c}
uqz_{1},vz_{2}\\
uqv
\end{array}   \bigg|q;bq\end{array}\right). 
\label{eq:jp 20 b}
\eg
\end{eqnarray}

From the contiguous relation (17.6.21) we get
\begin{equation}
\begin{aligned} & vz_{2}\left(1-uz_{1}\right){}_{2}\phi_{1}\left(\begin{array}{cc}
\begin{array}{c}
uz_{1}q,vz_{2}\\
uv
\end{array} & \bigg|q;bq\end{array}\right)\\
- & uz_{1}\left(1-vz_{2}\right){}_{2}\phi_{1}\left(\begin{array}{cc}
\begin{array}{c}
uz_{1},vz_{2}q\\
uv
\end{array} & \bigg|q;bq\end{array}\right)\\
= & \left(vz_{2}-uz_{1}\right){}_{2}\phi_{1}\left(\begin{array}{cc}
\begin{array}{c}
uz_{1},vz_{2}\\
uv
\end{array} & \bigg|q;bq\end{array}\right),
\end{aligned}
,\label{eq:jp 21}
\end{equation}
then apply \eqref{eq:jp generating function d} we get \eqref{eq:jp properties 21}.
\[
\begin{aligned} & vz_{2}\sum_{m,n=0}^{\infty}\frac{p_{m,n}\left(z_{1}q,z_{2};b\vert q\right)}{\left(q;q\right)_{m}\left(q;q\right)_{n}}u^{m}v^{n}\\
- & uz_{1}\sum_{m,n=0}^{\infty}\frac{p_{m,n}\left(z_{1},z_{2}q;b\vert q\right)}{\left(q;q\right)_{m}\left(q;q\right)_{n}}u^{m}v^{n}\\
= & \left(vz_{2}-uz_{1}\right)\sum_{m,n=0}^{\infty}\frac{p_{m,n}\left(z_{1},z_{2};b\vert q\right)}{\left(q;q\right)_{m}\left(q;q\right)_{n}}u^{m}v^{n}
\end{aligned}
\]
From \eqref{eq:jp 20 a}, \eqref{eq:jp 20 b} we get  \eqref{eq:jp properties 20 a}
and \eqref{eq:jp properties 20 b} respectively. From the contiguous
relation (17.6.22) in \cite{andrews-nist} we obtain  
\begin{equation}
\begin{aligned} & \left(uz_{1}-z_{1}z_{2}q\right){}_{2}\phi_{1}\left(\begin{array}{cc}
\begin{array}{c}
uz_{1}q,vz_{2}\\
uv
\end{array} & \bigg|q;bq\end{array}\right)\\
 & -\left(vz_{2}-z_{1}z_{2}q\right){}_{2}\phi_{1}\left(\begin{array}{cc}
\begin{array}{c}
uz_{1},vz_{2}q\\
uv
\end{array} & \bigg|q;bq\end{array}\right)\\
 & =\left(uz_{1}-vz_{2}\right)\left(1-bz_{1}z_{2}q\right){}_{2}\phi_{1}\left(\begin{array}{cc}
\begin{array}{c}
uz_{1}q,vz_{2}q\\
uv
\end{array} & \bigg|q;bq\end{array}\right),
\end{aligned}
\label{eq:jp 22}
\end{equation}
 which gives \eqref{eq:jp properties 22}. \end{proof}

The inversion transformation of quanta $q\to q^{-1}$ in \eqref{eqhvsH}
relates the properties of one family of polynomials for $q>1$ to
the properties of another family of polynomials with $0<q<1$. The
polynomials $p_{m,n}\left(z_{1},z_{2};b\vert q\right)$ are essentially
invariant under the quanta inversion transformation, 

\begin{eqnarray*}
&{}& p_{m,n}\left(z_{1},z_{2};b\vert q^{-1}\right)  \\
&{}& =  \left(-b\right)^{m+n}\sum_{k=0}^{\infty}\gauss{m}{k}\gauss{n}{k}\left(-b\right)^{-k}\left(q;q\right)_{k}\left(\frac{q}{b};q\right)_{m+n-k}\\
&{}  &\qquad  \times  q^{k\left(k-m\right)+k\left(k-n\right)-\binom{k}{2}-\binom{k+1}{2}-\binom{m+n-k+1}{2}}z_{1}^{m-k}z_{2}^{n-k}\\
&{}  & =  \left(-b\right)^{m+n}q^{-\binom{m+n+1}{2}}
\sum_{k=0}^{\infty}\gauss{m}{k}\gauss{n}{k}
 q^{\binom{k}{2}}\left(-\frac{q}{b}\right)^{k}\left(q;q\right)_{k}\left(\frac{q}{b};q\right)_{m+n-k}z_{1}^{m-k}z_{2}^{n-k}\\
 &{} & =  \left(-1\right)^{m+n}\left(\frac{b}{q}\right)^{\left(1-\alpha\right)m+\alpha n}q^{-\binom{m+n}{2}}p_{m,n}\left(\left(\frac{b}{q}\right)^{\alpha}z_{1},\left(\frac{b}{q}\right)^{1-\alpha}z_{2};\frac{1}{b}\vert q\right).
\end{eqnarray*}
Therefore, we have established the symmetry 
\begin{eqnarray}
\label{eqpmnq1/qsymm}
p_{m,n}\left(z_{1},z_{2};b\vert q^{-1}\right)=\frac{\left(bq^{-1}\right)^{\left(1-\alpha\right)m+\alpha n}}{\left(-1\right)^{m+n}q^{\binom{m+n}{2}}}p_{m,n}\left(\left(b/q\right)^{\alpha}z_{1},
\left(b/q\right)^{1-\alpha}z_{2}; 1/b \vert q\right),
\end{eqnarray}
 for  $\alpha\in\mathbb{C}$.

We now come the asymptotics of   $p_{m,n}\left(z_{1},z_{2};b\vert q\right)$. 
\begin{thm}
Let $z_{1},z_{2}\in\mathbb{C}$,  $bq<1$ and $z_{1}z_{2}\neq0$,
then we have
\begin{eqnarray} \label{eqasypmn}
 \lim_{m,n\to \infty} 
 \frac{p_{m,n}\left(z_{1},z_{2};b\vert q\right)}{\left(bq;q\right)_{\infty}z_{1}^{m}z_{2}^{n}}
 =\left(\frac{1}{z_{1}z_{2}};q\right)_{\infty},
\end{eqnarray}
 uniformly on compact subsets of the $z_1$ and $z_2$ planes. 
\end{thm}
The theorem follows from the definition \eqref{eq:jpdefinition}  and Tannery's theorem. 
 
\section{Applications}
\begin{thm}
Let $\left|t_{i}x_{i}\right|<\sqrt{q}$ for $i =1,2,3,4$, then we
have
\begin{equation}
\begin{aligned} & 
 \frac{\left(t_{1}x_{1}\sqrt{q},t_{2}x_{2}\sqrt{q},t_{3}x_{3}\sqrt{q},t_{4}
x_{4}\sqrt{q},x_{1}x_{2},x_{3}x_{4},t_{1}t_{2}t_{3}t_{4}x_{1}x_{2}x_{3}x_{4}q^{2};q\right)_{\infty}}
{\left(t_{1}t_{2}x_{1}x_{2},t_{1}t_{4}x_{1}x_{4},t_{2}t_{3}x_{2}x_{3},t_{3}t_{4}x_{3}x_{4},-x_{1}
x_{2}x_{3}x_{4};q\right)_{\infty}}\\
 & =\sum\frac{h_{m_{1},n_{1}}\left(t_{1}t_{3},t_{2}t_{4}|q\right)}{\left(q;q\right)_{m_{1}
 }\left(q;q\right)_{n_{1}}}q^{\left(\left(m_{1}-n_{1}\right)^{2}+m_{1}+n_{1} + 
 \left(m_{1}+m_{2}+m_{3}-n_{1}-n_{2}-n_{3}\right)^{2}\right)/2}\\
 & \times\frac{H_{m_{2},n_{2}}\left(t_{1},t_{2}\big|q\right)H_{m_{3},n_{3}}\left(t_{3}, 
 t_{4}\big|q\right)x_{1}^{m_{1}+m_{2}}x_{2}^{n_{1}+n_{2}}x_{3}^{m_{1}+m_{3}}
 x_{4}^{n_{1}+n_{3}}}{\left(-1\right)^{m_{2}+m_{3}-n_{2}-n_{3}} 
 \prod_{j=2}^{3}\left(q;q\right)_{m_{j}}\left(q;q\right)_{n_{j}}},
\end{aligned}
\label{eq:circle}
\end{equation}
where the summation is over all the nonnegative integers $m_{i},n_{i}\quad i=1,2,3$ 
such that $m_{1}+m_{2}+m_{3}=n_{1}+n_{2}+n_{3}$.\end{thm}
\begin{proof}
Observe that
\[
\begin{aligned} & \frac{\left(x_{1}x_{2}q^{-1};q\right)_{\infty}}{\left(t_{1}x_{1}q^{-1/2} 
z,t_{2}x_{2}q^{-1/2}/z;q\right)_{\infty}} 
 \frac{\left(x_{3}x_{4}q^{-1};q\right)_{\infty}}{\left(t_{3}x_{3}q^{-1/2}z,t_{4}x_{4}q^{-1/2}/z;q
 \right)_{\infty}}\\
 & \times\frac{\left(q^{1/2}t_{1}t_{3}x_{1}x_{3}zq^{-1},q^{1/2}t_{2}t_{4}x_{2}x_{4}q^{-1}/z;q\right)_{\infty}}{\left(-x_{1}x_{2}x_{3}x_{4}q^{-2};q\right)_{\infty}}\\
 & =\sum\frac{h_{m_{1},n_{1}}\left(t_{1}t_{3},t_{2}t_{4}|q\right)H_{m_{2},n_{2}} 
 \left(t_{1},t_{2}\big|q\right)H_{m_{3},n_{3}}\left(t_{3},t_{4}\big|q\right)}
 {\prod_{j=1}^{3}\left(q;q\right)_{m_{j}}\left(q;q\right)_{n_{j}}}\\
 & \times q^{\left(\left(m_{1}-n_{1}\right)^{2}-2m_{1}-2n_{1}-m_{2}-n_{2}-m_{3}-n_{3}\right)/2}\\
 & \left(-1\right)^{\left(m_{1}-n_{1}\right)}x_{1}^{m_{1}+m_{2}}x_{2}^{n_{1} + 
 n_{2}}x_{3}^{m_{1}+m_{3}}x_{4}^{n_{1}+n_{3}}\\
 & \times z^{\left(m_{1}+m_{2}+m_{3}\right)-\left(n_{1}+n_{2}+n_{3}\right)},
\end{aligned}
\]
 by the $q$-beta integral we have
\[
\begin{aligned} & \frac{\left(x_{1}x_{2}q^{-1},x_{3}x_{4}q^{-1},t_{1}t_{2}t_{3}t_{4}x_{1}x_{2}x_{3}x_{4};q\right)_{\infty}}{\left(q,-x_{1}x_{2}x_{3}x_{4}q^{-2};q\right)_{\infty}}\\
 & \times\frac{\left(t_{1}x_{1},t_{2}x_{2},t_{3}x_{3},t_{4}x_{4};q\right)_{\infty}}{\left(t_{1}t_{2}x_{1}x_{2}q^{-1},t_{1}t_{4}x_{1}x_{4}q^{-1},t_{2}t_{3}x_{2}x_{3}q^{-1},t_{3}t_{4}x_{3}x_{4}q^{-1};q\right)_{\infty}}\\
 & =\sum\frac{h_{m_{1},n_{1}}\left(t_{1}t_{3},t_{2}t_{4}|q\right)H_{m_{2},n_{2}}\left(t_{1},t_{2}\big|q\right)H_{m_{3},n_{3}}\left(t_{3},t_{4}\big|q\right)}{\left(q;q\right)_{\infty}\prod_{j=1}^{3}\left(q;q\right)_{m_{j}}\left(q;q\right)_{n_{j}}}\\
 & \times q^{\left(\left(m_{1}-n_{1}\right)^{2}-2m_{1}-2n_{1}-m_{2}-n_{2}-m_{3}-n_{3}\right)/2}/\left(2\pi i\right)\\
 & \times\left(-1\right)^{\left(m_{1}-n_{1}\right)}x_{1}^{m_{1}+m_{2}}x_{2}^{n_{1}+n_{2}}x_{3}^{m_{1}+m_{3}}x_{4}^{n_{1}+n_{3}}\\
 & \times\int_{\left|z\right|=1}z^{\left(m_{1}+m_{2}+m_{3}\right)-\left(n_{1}+n_{2}+n_{3}\right)}\left(q,q^{1/2}z,q^{1/2}/z;q\right)_{\infty}\frac{dz}{z}\\
 & =\sum\frac{h_{m_{1},n_{1}}\left(t_{1}t_{3},t_{2}t_{4}|q\right)H_{m_{2},n_{2}}\left(t_{1},t_{2}\big|q\right)H_{m_{3},n_{3}}\left(t_{3},t_{4}\big|q\right)}{\left(q;q\right)_{\infty}\prod_{j=1}^{3}\left(q;q\right)_{m_{j}}\left(q;q\right)_{n_{j}}}\\
 & \times q^{\left(\left(m_{1}-n_{1}\right)^{2}-2m_{1}-2n_{1}-m_{2}-n_{2}-m_{3}-n_{3}+\left(m_{1}+m_{2}+m_{3}-n_{1}-n_{2}-n_{3}\right)^{2}\right)/2}\\
 & \times\left(-1\right)^{m_{2}+m_{3}-n_{2}-n_{3}}x_{1}^{m_{1}+m_{2}}x_{2}^{n_{1}+n_{2}}x_{3}^{m_{1}+m_{3}}x_{4}^{n_{1}+n_{3}},
\end{aligned}
\]
where the summation is over all the nonnegative integers  $m_{i},n_{i}\quad i=1,2,3$ such that $m_{1}+m_{2}+m_{3}=n_{1}+n_{2}+n_{3}$,
which is \eqref{eq:circle}.\end{proof} 
From \eqref{eq:h2l} and \eqref{eq:H2w} we obtain the following equivalent
representation:
\begin{cor}
Let $\left|t_{i}x_{i}\right|<q$ for $j=1,2,3,4$, then we have
\begin{equation}
\begin{aligned} & \frac{\left(t_{1}x_{1},t_{2}x_{2},t_{3}x_{3},t_{4}x_{4},x_{1}x_{2},x_{3}x_{4},t_{1}t_{2}t_{3}t_{4}x_{1}x_{2}x_{3}x_{4};q\right)_{\infty}}{\left(t_{1}t_{2}x_{1}x_{2}q^{-1},t_{1}t_{4}x_{1}x_{4}q^{-1},t_{2}t_{3}x_{2}x_{3}q^{-1},t_{3}t_{4}x_{3}x_{4}q^{-1},-x_{1}x_{2}x_{3}x_{4};q\right)_{\infty}}\\
 & =\sum\frac{L_{n_{1}}^{\left(m_{1}-n_{1}\right)}\left(t_{1}t_{2}t_{3}t_{4}q^{-2};q\right)}{\left(q;q\right)_{m_{1}}\left(q;q\right)_{m_{2}}\left(q;q\right)_{m_{3}}}p_{n}\left(t_{1}t_{2}q^{-1},q^{m_{2}-n_{2}}\bigg|q\right)p_{n}\left(t_{3}t_{4}q^{-1},q^{m_{3}-n_{3}}\bigg|q\right)\\
 & \times\gauss{m_{2}}{n_{2}}\gauss{m_{3}}{n_{3}}q^{\left(\left(m_{1}-n_{1}\right)^{2}+n_{2}^{2}+n_{3}^{2}+\left(m_{1}+m_{2}+m_{3}-n_{1}-n_{2}-n_{3}\right)^{2}-\left(m_{1}+m_{2}+m_{3}-3n_{1}\right)\right)/2}\\
 & \times\left(-1\right)^{m_{2}+m_{3}+n_{1}}t_{1}^{m_{1}+m_{2}-n_{1}-n_{2}}t_{3}^{m_{1}+m_{3}-n_{1}-n_{3}}x_{1}^{m_{1}+m_{2}}x_{2}^{n_{1}+n_{2}}x_{3}^{m_{1}+m_{3}}x_{4}^{n_{1}+n_{3}},
\end{aligned}
\label{eq:circle2}
\end{equation}
where the summation is over all the nonnegative integers $m_{i},n_{i}\quad i=1,2,3$ such that $m_{1}+m_{2}+m_{3}=n_{1}+n_{2}+n_{3}$.
\end{cor}

From \eqref{eq:H2w} we find that 
\begin{eqnarray*}
&{}& \sum_{n=-\infty}^{\infty}t^{n}J_{n}^{(2)}\left(x;q\right)   =  \frac{\left(\frac{x^{2}}{4};q\right)_{\infty}}{\left(\frac{xt}{2},\frac{x}{2t};q\right)_{\infty}}\frac{\left(-\frac{x^{2}}{4};q\right)_{\infty}}{\left(\frac{x^{2}}{4};q\right)_{\infty}}
  =  \frac{\left(-\frac{x^{2}}{4};q\right)_{\infty}}{\left(\frac{x^{2}}{4};q\right)_{\infty}}\sum_{j,k=0}^{\infty}\frac{H_{j,k}\left(\frac{x}{2},\frac{x}{2}\big|q\right)}{\left(q;q\right)_{j}\left(q;q\right)_{k}}t^{j-k}\\
 & = & \frac{\left(-\frac{x^{2}}{4};q\right)_{\infty}}{\left(\frac{x^{2}}{4};q\right)_{\infty}}\sum_{j,k=0}^{\infty}\frac{p_{k}\left(\frac{x^{2}}{4},q^{j-k}\bigg|q\right)}{\left(q;q\right)_{j}\left(q;q\right)_{k}}\left(\frac{x}{2}\right)^{j-k}t^{j-k}\\
 & = & \frac{\left(-\frac{x^{2}}{4};q\right)_{\infty}}{\left(\frac{x^{2}}{4};q\right)_{\infty}}\sum_{n=-\infty}^{\infty}t^{n}\frac{\left(\frac{x}{2}\right)^{n}}{\left(q;q\right)_{n}}\sum_{k=0}^{\infty}\frac{p_{k}\left(\frac{x^{2}}{4},q^{n}\bigg|q\right)}{\left(q,q^{n+1};q\right)_{k}},
\end{eqnarray*}
then,
\[
J_{n}^{(2)}\left(x;q\right)=\left(\frac{x}{2}\right)^{n}\frac{\left(q^{n+1},-\frac{x^{2}}{4};q\right)_{\infty}}{\left(q,\frac{x^{2}}{4};q\right)_{\infty}}\sum_{k=0}^{\infty}\frac{p_{k}\left(\frac{x^{2}}{4},q^{n}\bigg|q\right)}{\left(q,q^{n+1};q\right)_{k}}
\]
for all $n\in\mathbb{Z}$, and 
\begin{equation}
\frac{J_{\alpha}^{(2)}\left(2x;q\right)}{x^{\alpha}}=\frac{\left(q^{\alpha+1},-x^{2};q\right)_{\infty}}{\left(q,x^{2};q\right)_{\infty}}\sum_{k=0}^{\infty}\frac{p_{k}\left(x^{2},q^{\alpha}\bigg|q\right)}{\left(q,q^{\alpha+1};q\right)_{k}}\label{eq:bessel2wall}
\end{equation}
 for $\alpha>0$ by analytic continuation. 

We now use the  Askey--Roy integral  \eqref{eqAsk:Roy}  to derive 
\[
\begin{aligned} & \frac{\left(ab\alpha\beta,c,q/c,c\alpha/\beta,q\beta/c\alpha;q\right)_{\infty}}{\left(a\beta,b\alpha,q,-c^{2}\alpha/q\beta,-q\beta/c^{2}\alpha;q\right)_{\infty}}\\
 & =\int_{-\pi}^{\pi}\frac{\left(ce^{i\theta}/\beta,c\alpha e^{-i\theta},qe^{i\theta}/c\alpha,q\beta e^{-i\theta}/c;q\right)_{\infty}}{\left(-c^{2}\alpha/q\beta,-q\beta/c^{2}\alpha;q\right)_{\infty}}\\
 & \times\frac{\left(a\alpha,b\beta;q\right)_{\infty}}{\left(ae^{i\theta},\alpha e^{-i\theta},be^{i\theta},\beta e^{-i\theta};q\right)_{\infty}}\frac{d\theta}{2\pi}\\
 & =\sum\frac{h_{m_{1},n_{1}}\left(c/\beta,c\alpha|q\right)}{\left(q;q\right)_{m_{1}}\left(q;q\right)_{n_{1}}}\frac{h_{m_{2},n_{2}}\left(1/c\alpha,\beta/c|q\right)}{\left(q;q\right)_{m_{2}}\left(q;q\right)_{n_{2}}}\\
 & \times\frac{H_{m_{3},n_{3}}\left(a,\alpha\vert q\right)}{\left(q;q\right)_{m_{3}}\left(q;q\right)_{n_{3}}}\frac{H_{m_{4},n_{4}}\left(b,\beta\vert q\right)}{\left(q;q\right)_{m_{4}}\left(q;q\right)_{n_{4}}}\left(-1\right)^{m_{1}+m_{2}+n_{1}+n_{2}}\\
 & \times q^{-(m_{1}+m_{2}+n_{1}+n_{2})/2+(m_{1}-n_{1})^{2}/2+(m_{2}-n_{2})^{2}/2}\\
 & \times\int_{-\pi}^{\pi}e^{i\theta\left(m_{1}+m_{2}+m_{3}+m_{4}-n_{1}-n_{2}-n_{3}-n_{4}\right)}\frac{d\theta}{2\pi},
\end{aligned}
\]
that is,
\begin{equation}
\begin{aligned} & \frac{\left(ab\alpha\beta,c,q/c,c\alpha/\beta,q\beta/c\alpha;q\right)_{\infty}}{\left(a\beta,b\alpha,q,-c^{2}\alpha/q\beta,-q\beta/c^{2}\alpha;q\right)_{\infty}}\\
 & =\sum\frac{h_{m_{1},n_{1}}\left(c/\beta,c\alpha|q\right)}{\left(q;q\right)_{m_{1}}\left(q;q\right)_{n_{1}}}\frac{h_{m_{2},n_{2}}\left(1/c\alpha,\beta/c|q\right)}{\left(q;q\right)_{m_{2}}\left(q;q\right)_{n_{2}}}\\
 & \times\frac{H_{m_{3},n_{3}}\left(a,\alpha\vert q\right)}{\left(q;q\right)_{m_{3}}\left(q;q\right)_{n_{3}}}\frac{H_{m_{4},n_{4}}\left(b,\beta\vert q\right)}{\left(q;q\right)_{m_{4}}\left(q;q\right)_{n_{4}}}\left(-1\right)^{m_{1}+m_{2}+n_{1}+n_{2}}\\
 & \times q^{-(m_{1}+m_{2}+n_{1}+n_{2})/2+(m_{1}-n_{1})^{2}/2+(m_{2}-n_{2})^{2}/2},
\end{aligned}
\label{eq:askeyroy}
\end{equation}
 where $\left|q\right|,\left|\alpha\right|,\left|\beta\right|,\left|a\right|,\left|b\right|<1$,
$c\alpha\beta\neq0$ and the summation is over all the nonnegative
integers such that $m_{1}+m_{2}+m_{3}+m_{4}-n_{1}-n_{2}-n_{3}-n_{4}=0$.

From \eqref{eq:h2l} and \eqref{eq:H2w} we obtain the following equivalent
representation,

\begin{equation}
\begin{aligned} & \frac{\left(ab\alpha\beta,c,q/c,c\alpha/\beta,q\beta/c\alpha;q\right)_{\infty}}{\left(a\beta,b\alpha,q,-c^{2}\alpha/q\beta,-q\beta/c^{2}\alpha;q\right)_{\infty}}\\
 & =\sum\frac{L_{n_{1}}^{\left(m_{1}-n_{1}\right)}\left(c^{2}\alpha/\beta;q\right)q^{(m_{1}-n_{1})^{2}/2}}{\left(-1\right)^{m_{1}}\left(q;q\right)_{m_{1}}\left(\beta/c\right)^{m_{1}-n_{1}}q^{(m_{1}+n_{1})/2}}\\
 & \times\frac{L_{n_{2}}^{\left(m_{2}-n_{2}\right)}\left(\beta/c^{2}\alpha;q\right)q^{(m_{2}-n_{2})^{2}/2}}{\left(-1\right)^{m_{2}}\left(q;q\right)_{m_{2}}\left(c\alpha\right)^{m_{2}-n_{2}}q^{(m_{2}+n_{2})/2}}\\
 & \times\frac{p_{n_{3}}\left(a\alpha,q^{m_{3}-n_{3}}\bigg|q\right)p_{n_{4}}\left(b\beta,q^{m_{4}-n_{4}}\bigg|q\right)}{a^{n_{3}-m_{3}}b^{n_{4}-m_{4}}\left(q;q\right)_{m_{3}}\left(q;q\right)_{n_{3}}\left(q;q\right)_{m_{4}}\left(q;q\right)_{n_{4}}},
\end{aligned}
\label{eq:askeyroy2}
\end{equation}
where $\left|q\right|,\left|\alpha\right|,\left|\beta\right|,\left|a\right|,\left|b\right|<1$,
$c\alpha\beta\neq0$, the summation is over all the nonnegative integers
such that $m_{1}+m_{2}+m_{3}+m_{4}-n_{1}-n_{2}-n_{3}-n_{4}=0$ and
$L_{n}^{(\alpha)}\left(x;q\right)$ and $p_{n}\left(x,a\vert q\right)$
are the $q$-Laguerre and Little $q$-Laguerre polynomials respectively.

Let $a=ue^{i\phi},b=ue^{-i\phi},\alpha=ve^{i\psi},\beta=ve^{-i\psi},c=q^{1/2}$
in Askey and Roy integral to get 
\[
\begin{aligned} & \int_{-\pi}^{\pi}\frac{\left(q^{1/2}e^{i\theta}e^{i\psi}/v,q^{1/2}e^{i\theta}e^{-i\psi}/v,q^{1/2}e^{-i\theta}ve^{i\psi},q^{1/2}e^{-i\theta}ve^{-i\psi};q\right)_{\infty}}{\left(e^{i\phi}ue^{i\theta},e^{-i\phi}ue^{i\theta},ve^{-i\theta}e^{i\psi},ve^{-i\theta}e^{-i\psi};q\right)_{\infty}}\frac{d\theta}{2\pi}\\
 & =\frac{\left(u^{2}v^{2},q^{1/2},q^{1/2},q^{1/2}e^{2i\psi},q^{1/2}e^{-2i\psi};q\right)_{\infty}}{\left(uve^{i(\phi+\psi)},uve^{i(\phi-\psi)},uve^{-i(\phi-\psi)},uve^{-i(\phi+\psi)},q;q\right)_{\infty}}
\end{aligned}
\]
and
\[
\begin{aligned} & \frac{\left(u^{2}v^{2},q^{1/2},q^{1/2},q^{1/2}e^{2i\psi},q^{1/2}e^{-2i\psi};q\right)_{\infty}}{\left(uve^{i(\phi+\psi)},uve^{i(\phi-\psi)},uve^{-i(\phi-\psi)},uve^{-i(\phi+\psi)},q;q\right)_{\infty}}\\
 & =\sum\frac{h_{m_{1}}\left(\sinh\left(i\psi+\frac{\pi}{2}i\right)\vert q\right)h_{m_{1}}\left(\sinh\left(i\psi-\frac{\pi}{2}i\right)\vert q\right)}{\left(q;q\right)_{m_{1}}\left(q;q\right)_{m_{2}}\left(q;q\right)_{m_{3}}\left(q;q\right)_{m_{4}}\left(vi\right)^{m_{1}+m_{2}}}\\
 & \times H_{m_{3}}\left(\cos\phi\vert q\right)H_{m_{4}}\left(\cos\psi\vert q\right)u^{m_{3}}v^{m_{4}}\\
 & \times q^{(m_{1}^{2}+m_{2}^{2})/2}\int_{-\pi}^{\pi}e^{i\theta(m_{1}+m_{3}-m_{2}-m_{4})}\frac{d\theta}{2\pi}\\
 & =\sum\frac{h_{m_{1}}\left(\sinh\left(i\psi+\frac{\pi}{2}i\right)\vert q\right)h_{m_{2}}\left(i\sinh\left(\psi-\frac{\pi}{2}i\right)\vert q\right)}{\left(q;q\right)_{m_{1}}\left(q;q\right)_{m_{2}}\left(q;q\right)_{m_{3}}\left(q;q\right)_{m_{4}}\left(vi\right)^{m_{1}+m_{2}}}\\
 & \times q^{(m_{1}^{2}+m_{2}^{2})/2}H_{m_{3}}\left(\cos\phi\vert q\right)H_{m_{4}}\left(\cos\psi\vert q\right)u^{m_{3}}v^{m_{4}},
\end{aligned}
\]
or
\[
\begin{aligned} & \frac{\left(u^{2}v^{2},q^{1/2},q^{1/2},q^{1/2}e^{\pi+2i\psi},q^{1/2}e^{-\pi-2i\psi};q\right)_{\infty}}{\left(uve^{i(\phi+\psi)+\pi/2},uve^{i(\phi-\psi)-\pi/2},uve^{-i(\phi-\psi)+\pi/2},uve^{-i(\phi+\psi)-\pi/2},q;q\right)_{\infty}}\\
 & =\sum\frac{h_{m_{1}}\left(i\sin\psi\vert q\right)h_{m_{2}}\left(i\sin\psi\vert q\right)\left(-1\right)^{m_{2}}}{\left(q;q\right)_{m_{1}}\left(q;q\right)_{m_{2}}\left(q;q\right)_{m_{3}}\left(q;q\right)_{m_{4}}\left(vi\right)^{m_{1}+m_{2}}}\\
 & \times q^{(m_{1}^{2}+m_{2}^{2})/2}H_{m_{3}}\left(\cos\phi\vert q\right)H_{m_{4}}\left(\cos\psi\vert q\right)u^{m_{3}}v^{m_{4}}\\
 & =\sum\frac{\left(-1\right)^{m_{1}}q^{(m_{1}^{2}+m_{2}^{2})/2}u^{m_{3}}v^{m_{4}-m_{1}-m_{2}}}{\left(q;q\right)_{m_{1}}\left(q;q\right)_{m_{2}}\left(q;q\right)_{m_{3}}\left(q;q\right)_{m_{4}}}\\
 & \times H_{m_{1}}\left(\sin\psi\vert q^{-1}\right)H_{m_{2}}\left(\sin\psi\vert q^{-1}\right)H_{m_{3}}\left(\cos\phi\vert q\right)H_{m_{4}}\left(\cos\psi\vert q\right)
\end{aligned}
\]
or 
\begin{equation}
\begin{aligned} & \frac{\left(u^{2}v^{2},q^{1/2},q^{1/2},q^{1/2}e^{\pi+2i\psi},q^{1/2}e^{-\pi-2i\psi};q\right)_{\infty}}{\left(uve^{i(\phi+\psi)+\pi/2},uve^{i(\phi-\psi)-\pi/2},uve^{-i(\phi-\psi)+\pi/2},uve^{-i(\phi+\psi)-\pi/2},q;q\right)_{\infty}}\\
 & =\sum\frac{H_{m_{1}}\left(\sin\psi\vert q^{-1}\right)H_{m_{2}}\left(\sin\psi\vert q^{-1}\right)H_{m_{3}}\left(\cos\phi\vert q\right)H_{m_{4}}\left(\cos\psi\vert q\right)}{\left(-1\right)^{m_{1}}q^{-(m_{1}^{2}+m_{2}^{2})/2}\left(q;q\right)_{m_{1}}\left(q;q\right)_{m_{2}}\left(q;q\right)_{m_{3}}\left(q;q\right)_{m_{4}}v^{m_{1}+m_{2}-m_{4}}},
\end{aligned}
\label{eq:qks1}
\end{equation}
where the summation is over all the nonnegative integers $m_{i},i=1,2,3,4$
such that $m_{1}+m_{3}= m_{2}+m_{4}$.

 \section{Additional Results}  
 In this section we first derive moment integral representations for $\{H_{m,n}(\zeta, \bar \zeta|q)\}$ and 
  $\{h_{m,n}(\zeta, \bar \zeta|q)\}$. We then derive additional generating functions and expansions.
  We shall use the terminating $q$-binomial theorem \eqref{eqqbt} in the form
  \begin{eqnarray}
  \label{eqfqbinom}
  \prod_{j=0}^{n-1}(a + bq^j) = \sum_{j=0}^n \gauss{n}{j} q^{\binom{j}{2}} a^{n-j}b^j. 
  \end{eqnarray}
 \begin{thm}
Let $\mu\left(\zeta,\bar{\zeta}\right)$, be a normalized orthogonal
measure for $H_{m,n}\left(z,\bar{z}\vert q\right)$ and $\nu\left(\zeta,\bar{\zeta}\right)$
be a normalized measure for $h_{m,n}\left(z,\bar{z}\vert\right)$
respectively, then we have the integral representations 

\begin{equation}
H_{m,n}\left(z_{1},z_{2}\vert q\right)=\int_{\mathbb{R}^{2}}\prod_{j=0}^{m-1}\left(z_{1}+i\zeta q^{\frac{1}{2}+j}\right)\prod_{k=0}^{n-1}\left(z_{2}+i\bar{\zeta}q^{\frac{1}{2}+k}\right)d\nu\left(\zeta,\bar{\zeta}\right),
\label{eq1stmoment}
\end{equation}
and
\begin{equation}
q^{\frac{\left(m-n\right)^{2}}{2}}i^{m+n}h_{m,n}\left(z_{1},z_{2}\vert q\right)=\int_{\mathbb{R}^{2}}\prod_{j=0}^{m-1}\left(\zeta+iz_{1}q^{\frac{1}{2}+j}\right)\prod_{k=0}^{n-1}\left(\bar{\zeta}+iz_{2}q^{\frac{1}{2}+k}\right)d\mu\left(\zeta,\bar{\zeta}\right),\label{eq2ndmoment}
\end{equation}
where $z_{1},z_{2}\in\mathbb{C}$ and $m,n\in\mathbb{N}_{0}$. \end{thm}
\begin{proof}
Let
\[
a_{m,n}\left(z_{1},z_{2}\vert q\right)=\int_{\mathbb{R}^{2}}\prod_{j=0}^{m-1}\left(\zeta+iz_{1}q^{\frac{1}{2}+j}\right)\prod_{k=0}^{n-1}\left(\bar{\zeta}+iz_{2}q^{\frac{1}{2}+k}\right)d\mu\left(\zeta,\bar{\zeta}\right). 
\]
The form \eqref{eqfqbinom} of the $q$-binomial theorem implies
\begin{eqnarray*}
   \prod_{j=0}^{m-1}\left(\zeta+iz_{1}q^{\frac{1}{2}+j}\right)\prod_{k=0}^{n-1}\left(\bar{\zeta}+iz_{2}q^{\frac{1}{2}+k}\right) 
  =   \sum_{j=0}^{m}\sum_{k=0}^{n}\gauss{m}{j}\gauss{n}{k}q^{\frac{j^{2}+k^{2}}{2}}i^{j+k}z_{1}^{j}z_{2}^{k}\zeta^{m-j}\bar{\zeta}^{n-k},
\end{eqnarray*}
and we find that
\begin{eqnarray*}
&{}& \sum_{m,n=0}^{\infty}\frac{a_{m,n}\left(z_{1},z_{2}\vert q\right)u^{m}v^{n}}{\left(q;q\right)_{m}\left(q;q\right)_{n}}  \\
&{}&   =  \int_{\mathbb{R}^{2}}\sum_{j,k=0}^{\infty}\frac{q^{\binom{j}{2}+\binom{k}{2}}
\left(iq^{\frac{1}{2}}uz_{1}\right)^{j}\left(iq^{\frac{1}{2}}vz_{2}\right)^{k}}{\left(q;q\right)_{j}\left(q;q\right)_{k}}
 \sum_{m\ge j,n\ge k}\frac{\left(u\zeta\right)^{m-j}}{\left(q;q\right)_{m-j}}\frac{\left(v\bar{\zeta}\right)^{n-j}}{\left(q;q\right)_{n-j}}d\mu\left(\zeta,\bar{\zeta}\right)\\
 & = & \frac{\left(-iq^{\frac{1}{2}}uz_{1},-iq^{\frac{1}{2}}vz_{2};q\right)_{\infty}}
  {\left(uv;q\right)_{\infty}}\int_{\mathbb{R}^{2}}\frac{\left(uv;q\right)_{\infty}\; 
   d\mu\left(\zeta,\bar{\zeta}\right)}{\left(u\zeta,v\bar{\zeta};q\right)_{\infty}}\\
 & = & \frac{\left(-iq^{\frac{1}{2}}uz_{1},-iq^{\frac{1}{2}}vz_{2};q\right)_{\infty}}
  {\left(uv;q\right)_{\infty}}
 =\sum_{m,n=0}^{\infty}\frac{h_{m,n}\left(z_{1},z_{2}\vert q\right)}{\left(q;q\right)_{m}
 \left(q;q\right)_{n}}q^{\frac{\left(m-n\right)^{2}}{2}}i^{m+n}u^{m}v^{n}.
\end{eqnarray*}
 Similarly, let 
\begin{eqnarray*}
b_{m,n}\left(z_{1},z_{2}\vert q\right) & = & \int_{\mathbb{R}^{2}}
  \prod_{j=0}^{m-1}\left(z_{1}+i\zeta q^{\frac{1}{2}+j}\right)
    \prod_{k=0}^{n-1}\left(z_{2}+i\bar{\zeta}q^{\frac{1}{2}+k}\right)\; 
     d\nu\left(\zeta,\bar{\zeta}\right)\\
 & = & \int_{\mathbb{R}^{2}}\sum_{j=0}^{m}\sum_{k=0}^{n}\gauss{m}{j}\gauss{n}{k}q^{\frac{j^{2}+k^{2}}{2}}i^{j+k}\zeta^{j}\bar{\zeta}^{k}z_{1}^{m-j}z_{2}^{n-k}\; 
  d\nu\left(\zeta,\bar{\zeta}\right),
\end{eqnarray*}
then,
\begin{eqnarray*}
&{}& \sum_{m,n=0}^{\infty}\frac{b_{m,n}\left(z_{1},z_{2}\vert q\right)u^{m}v^{n}}{\left(q;q\right)_{m}\left(q;q\right)_{n}} \\
&{} & =  \int_{\mathbb{R}^{2}}\sum_{j,k=0}^{\infty} 
 \frac{q^{\binom{j}{2}+\binom{k}{2}}\left(iq^{\frac{1}{2}}\zeta u\right)^{j}
  \left(iq^{\frac{1}{2}}\bar{\zeta}v\right)^{k}}{\left(q;q\right)_{m}\left(q;q\right)_{n}}
   \sum_{j\ge m,k\ge n}\frac{\left(uz_{1}\right)^{m-j}\left(vz_{2}\right)^{n-k}d\nu
    \left(\zeta,\bar{\zeta}\right)}{\left(q;q\right)_{m-j}\left(q;q\right)_{n-k}}\\
 & {} &= \frac{\left(uv;q\right)_{\infty}}{\left(uz_{1},vz_{2};q\right)_{\infty}}
  \int_{\mathbb{R}^{2}}\frac{\left(-iq^{\frac{1}{2}}\zeta u,-iq^{\frac{1}{2}}
   \bar{\zeta}v;q\right)_{\infty}d\nu\left(\zeta,\bar{\zeta}\right)}{\left(uv;q\right)_{\infty}}
   \\
 & {} &=  \frac{\left(uv;q\right)_{\infty}}{\left(uz_{1},vz_{2};q\right)_{\infty}}.
\end{eqnarray*}
This completes the proof. 
\end{proof}

\begin{rem}
 Observe that for any fixed $z_{1},z_{2}\neq0$, \eqref{eq1stmoment}
and \eqref{eq2ndmoment} can be re-casted into 

\begin{equation}
H_{m,n}\left(z_{1},z_{2}\vert q\right)=z_{1}^{m}z_{2}^{n}\int_{\mathbb{R}^{2}}\left(-\frac{i\zeta q^{\frac{1}{2}}}{z_{1}};q\right)_{m}\left(-\frac{i\bar{\zeta}q^{\frac{1}{2}}}{z_{2}};q\right)_{n}d\nu\left(\zeta,\bar{\zeta}\right)\label{eq1stmoment-1}
\end{equation}
 and
\begin{equation}
q^{\frac{\left(m-n\right)^{2}}{2}}i^{m+n}h_{m,n}\left(z_{1},z_{2}\vert q\right)
 =\int_{\mathbb{R}^{2}}\zeta^{m}\bar{\zeta}^{n}
  \left(-\frac{iz_{1}q^{\frac{1}{2}}}{\zeta};q\right)_{m}
   \left(-\frac{iz_{2}q^{\frac{1}{2}}}{\bar{\zeta}};q\right)_{n}d\mu\left(\zeta,\bar{\zeta}\right).\label{eq2ndmoment-1}
\end{equation}
Using the relation $\left(a;q\right)_{-n}= \left(-qa^{-1}\right)^{n}q^{\binom{n}{2}}/
\left(qa^{-1};q\right)_{n}$
for $n=0,1,\dots$ and above equations, we can extend the definitions
of $H_{m,n}\left(z_{1},z_{2}\vert q\right)$ and $h_{m,n}\left(z_{1},z_{2}\vert q\right)$
to all $m,n\in\mathbb{Z}$. Of course, we can use these equations
and $\left(a;q\right)_{n}= (a;q )_{\infty}/(aq^{n};q)_{\infty}$
to extend the definitions of $H_{m,n}\left(z_{1},z_{2}\vert q\right)$
and $h_{m,n}\left(z_{1},z_{2}\vert q\right)$ to all $m,n\in\mathbb{C}$
where the integrals are convergent. 
\end{rem}

\begin{cor} \label{cor19}
Let $a,b,z_{1},z_{2}\neq0$ such that $\left|\frac{cdq}{abz_{1}z_{2}}\right|<1$
and $cq^{m},dq^{m}\neq1,\ m\in\mathbb{N}$, then
\begin{equation}
\label{eq2phi1gf}
\begin{aligned} & \sum_{m,n=0}^{\infty}\frac{\left(a;q\right)_{m}\left(b;q\right)_{n}
 q^{\frac{\left(m-n\right)^{2}}{2}}h_{m,n}\left(z_{1},z_{2}\vert q\right)}
 {\left(q,c;q\right)_{m}\left(q,d;q\right)_{n}}
  \left(\frac{-c}{\sqrt{q}az_{1}}\right)^{m}\left(\frac{-d}{\sqrt{q} bz_{2}}\right)^{n}\\
& \quad  =  \frac{\left(c/a, d/b;q\right)_{\infty}}{\left(c,d ;q\right)_{\infty}}\; 
{}_2\phi_1\left(\left. \begin{matrix}
a,  b   \\
0
\end{matrix}\, \right|  q,    -\frac{cd}{qabz_{1}z_{2}}\right) \\
& \quad  =  \frac{\left(c/a, d/b, cdq^{-1}/bz_1z_2;q\right)_{\infty}}
 {\left(c,d, -cdq^{-1}/abz_1z_2 ;q\right)_{\infty}}\; 
{}_1\phi_1\left(\left. \begin{matrix}
a   \\
cdq^{-1}/bz_1z_2
\end{matrix}\, \right|  q,    \frac{cd}{q a z_{1}z_{2}}\right).  
\end{aligned}
\end{equation}
The equality between the left-hand side and the extreme right-hand side holds 
when $z_1z_2 \ne 0$ without the assumption  $\left|\frac{cdq}{abz_{1}z_{2}}\right|<1$. 
On the other hand if   $z_{1},z_{2}\neq0$ and $cq^{m},dq^{m}\neq1,\ m\in\mathbb{N}$ then 
\begin{equation}
\label{eqAqgf}
\begin{aligned}\sum_{m,n=0}^{\infty}\frac{q^{m^{2}-mn+n^{2}}h_{m,n}\left(z_{1},z_{2}\vert q\right)}{\left(q,cq;q\right)_{m}\left(q,dq;q\right)_{n}}\left(\frac{c}{z_{1}}\right)^{m}\left(\frac{d}{z_{2}}\right)^{n} & =\frac{A_{q}\left(\frac{cd}{z_{1}z_{2}}\right)}{\left(cq,dq;q\right)_{\infty}},
\end{aligned}
\end{equation}
where $A_q$ is the Ramanujan function defined in \eqref{eqAq}.  Alternately the generating 
function \eqref{eqAqgf} may be written as  
\begin{equation}
\label{eqAqgf2}
\sum_{m,n=0}^{\infty}\frac{q^{m^{2}-mn+n^{2}}h_{m,n}\left(z_{1},z_{2}\vert q\right)}{\left(q,cz_{1}q;q\right)_{m}\left(q,dz_{2}q;q\right)_{n}}c^{m}d^{n}=\frac{A_{q}\left(cd\right)}{\left(cz_{1}q,dz_{2}q;q\right)_{\infty}}.
\end{equation}
which   holds 
for any $c,d,z_{1},z_{2}\in\mathbb{C}$ such that 
 $cz_{1}q^{m},dz_{2}q^{m}\neq1,\ m\in\mathbb{N}$. \end{cor}
\begin{proof}
The integral representation \eqref{eq2ndmoment-1} and Fubini's theorem imply  
\begin{eqnarray*}
 &  & \sum_{m,n=0}^{\infty}\frac{\left(a;q\right)_{m}\left(b;q\right)_{n}
  q^{\frac{\left(m-n\right)^{2}}{2}}i^{m+n}h_{m,n}\left(z_{1},z_{2}\vert q\right)}
   {\left(q,c;q\right)_{m}
  \left(q,d;q\right)_{n}}\left(\frac{ic}{az_{1}\sqrt{q}}\right)^{m}
   \left(\frac{id}{bz_{2}\sqrt{q}}\right)^{n}\\
 & = & \int_{\mathbb{R}^{2}}\sum_{m=0}^{\infty}
  \frac{\left(a,- iz_{1}q^{\frac{1}{2}}/\zeta;q\right)_{m}}{\left(q,c;q\right)_{m}}
   \left(\frac{ic\zeta}{az_{1}\sqrt{q}}\right)^{m}
   \sum_{n=0}^{\infty}\frac{\left(b,- iz_{2}q^{1/2}/ \bar{\zeta};q\right)_{n}}
    {\left(q,d;q\right)_{n}}\left(\frac{id\bar{\zeta}}{bz_{2}\sqrt{q}}\right)^{n}\; 
     d\mu\left(\zeta,\bar{\zeta}\right), 
  \end{eqnarray*}   
  where we used the $q$-Gauss sum in the last line, \cite[(II.8)]{Gas:Rah}.  
  Thus the above quantity equals 
 \begin{eqnarray*}
 &{}&  \frac{\left(c/a, d/b; q\right)_{\infty}}{\left(c,d;q\right)_{\infty}}
  \int_{\mathbb{R}^{2}}\frac{\left(\frac{ic\zeta}{z_{1}\sqrt{q}},\frac{id\bar{\zeta}}
   {z_{2}\sqrt{q}};q\right)_{\infty}d\mu\left(\zeta,\bar{\zeta}\right)}
    {\left(\frac{ic\zeta}{az_{1}\sqrt{q}},\frac{id\bar{\zeta}}{bz_{2}\sqrt{q}};q\right)_{\infty}}\\
 & = & \frac{\left(c/a, d/b;q\right)_{\infty}}{\left(c,d;q\right)_{\infty}}
  \sum_{m,n=0}^{\infty}\frac{\left(a;q\right)_{m}}{\left(q;q\right)_{m}}
   \frac{\left(b;q\right)_{n}}{\left(q;q\right)_{n}}
    \left(\frac{ic}{az_{1}\sqrt{q}}\right)^{m}\left(\frac{id}{bz_{2}\sqrt{q}}\right)^{n}\; 
     \int_{\mathbb{R}^{2}}\zeta^{m}\bar{\zeta}^{n}d\mu\left(\zeta,\bar{\zeta}\right).
\end{eqnarray*}
It is straight forward to find that  
\[
\int_{\mathbb{R}^{2}}\zeta^{m}\bar{\zeta}^{n}d\mu\left(\zeta,\bar{\zeta}\right)
 =\left(q;q\right)_{n}\delta_{m,n},
\]
whence,
\begin{eqnarray*}
 &  & \sum_{m,n=0}^{\infty}\frac{\left(a;q\right)_{m}\left(b;q\right)_{n}
 q^{\frac{\left(m-n\right)^{2}}{2}}i^{m+n}h_{m,n}\left(z_{1},z_{2}\vert q\right)}
  {\left(q,c;q\right)_{m}\left(q,d;q\right)_{n}}
   \left(\frac{ic}{az_{1}\sqrt{q}}\right)^{m}\left(\frac{id}{bz_{2}\sqrt{q}}\right)^{n}\\
 & = & \frac{\left( c/a, d/b;q\right)_{\infty}}{\left(c,d;q\right)_{\infty}}
  \sum_{n=0}^{\infty}\frac{\left(a,b;q\right)_{n}}{\left(q;q\right)_{n}}
   \left(-\frac{cd}{abz_{1}z_{2}q}\right)^{n}, 
\end{eqnarray*}
which is the first equality in \eqref{eq2phi1gf}. The equality between the ${}_2\phi_1$ 
and the ${}_1\phi_1$  is a special case of the ${}_2\phi_1-{}_2\phi_2$ transformation 
\cite[(III.4]{Gas:Rah}. 

Similarly  the integral representation 
  \eqref{eq2ndmoment-1} and Fubini's theorem  give    
\begin{eqnarray*}
\sum_{m,n=0}^{\infty}\frac{q^{m^{2}-mn+n^{2}}h_{m,n}\left(z_{1},z_{2}\vert q\right)}{\left(q,c;q\right)_{m}\left(q,d;q\right)_{n}}\left(\frac{c}{z_{1}q}\right)^{m}\left(\frac{d}{z_{2}q}\right)^{n} & = & \frac{1}{\left(c,d;q\right)_{\infty}}\sum_{k=0}^{\infty}\frac{q^{k^{2}}}{\left(q;q\right)_{k}}\left(\frac{-cd}{z_{1}z_{2}q^{2}}\right)^{k}.
\end{eqnarray*}
This establishes \eqref{eqAqgf} and the proof is complete. 
\end{proof}

Note that the  ${}_1\phi_1$ 
representation of the generating function in Corollary \ref{cor19} is an analytic continuation of the 
${}_2\phi_1$ function. 

\begin{cor}\label{cor20}
We have the generating function 
\begin{eqnarray} 
\label{eqJqgf}
\bg
  \sum_{m,n=0}^{\infty}\frac{q^{\binom{m}{2}}(b;q)_{n}
 q^{\frac{\left(m-n\right)^{2}}{2}}h_{m,n}\left(z_{1},z_{2}\vert q\right)}
 {\left(q,cz_1;q\right)_{m}\left(q,dz_2;q\right)_{n}}
  \left(\frac{c}{\sqrt{q}}\right)^{m}\left(\frac{-d}{\sqrt{q} b}\right)^{n}\\
   =  \frac{\left(dz_2/b, q;q\right)_{\infty}}
 {\left(cz_1,dz_2;q\right)_{\infty}}\; b^{-\nu/2}\, I_\nu^{(2)}(2\sqrt{b};q), 
\eg
\end{eqnarray}
where $I_\nu^{(2)}$ is the modified $q$-Bessel function defined in  \eqref{eqInu2} and 
$q^\nu := cdq^{-2}/ b$. 
\end{cor}
\begin{proof}
Replace $c$ and $d$ by $cz_1$ and $dz_2$, respectively,  in \eqref{eq2phi1gf}
  then let  $a \to \infty$. 
\end{proof}

A very interesting special case of Corollary \ref{cor19} is the following theorem.
\begin{thm}
Let $cd = - q^s, s =0, 1, \cdots$. Then the generating function 
\begin{eqnarray}
\label{eqhasPGF}
\bg
\sum_{m,n=0}^{\infty}\frac{q^{m^{2}-mn+n^{2}}h_{m,n}\left(z_{1},z_{2}\vert q\right)}
 {\left(q,cz_{1}q;q\right)_{m}\left(q,dz_{2}q;q\right)_{n}}c^{m}d^{n} \qquad \qquad  \qquad \qquad  
 \\
\qquad \qquad  
 = \frac{1}{(cz_{1}q,dz_{2}q;q)_{\infty} } 
 \left[\frac{(-1)^s q^{-\binom{s}2} a_s(q)}{(q,q^4;q^5)_\infty}
+ \frac{(-1)^{s+1} q^{-\binom{s}2}b_s(q)}{(q^2,q^3;q^5)_\infty}\right].
\eg
\end{eqnarray}
holds when neither $cz_1$ nor $dz_2$ $= q^{-r}$ for $r=0,1, \cdots$. 
\end{thm}
\begin{proof} Apply 
\eqref{eqAqgf2} and the Garret--Ismail--Stanton result \eqref{eqGIS}. 
 \end{proof}

 \begin{thm}
\label{lem:LinearCoefficientH1} For any $m,n\in\mathbb{N}_{0}$ and
$z_{1},z_{2}\in\mathbb{C}$ we have
\begin{eqnarray}
z_{1}^{m}z_{2}^{n}&=&\sum_{k=0}^{\infty}\gauss{m}{k}\gauss{n}{k}
\left(q;q\right)_{k}H_{m-k,n-k}\left(z_{1},z_{2}\big|q\right),\label{eqzzinH} \\
z_{1}^{m}z_{2}^{n} &=& q^{-mn}\sum_{k=0}^{\infty}\gauss{m}{k}\gauss{n}{k}\left(q;q\right)_{k}q^{\binom{k}{2}}h_{m-k,n-k}\left(z_{1},z_{2}|q\right).
\label{eqzzinh}
\end{eqnarray}
\end{thm}
\begin{proof}
From the generating function \eqref{eqGFHq}  we see that 
\begin{eqnarray*}
\frac{1}{\left(uz_{1},vz_{2};q\right)_{\infty}} & = & \frac{1}{\left(uv;q\right)_{\infty}}\sum_{m,n=0}^{\infty}\frac{H_{m,n}\left(z_{1},z_{2}\big|q\right)u^{m}v^{n}}{\left(q;q\right)_{m}\left(q;q\right)_{n}}\\
 & = & \sum_{k=0}^{\infty}\sum_{m,n=0}^{\infty}\frac{H_{m,n}\left(z_{1},z_{2}\big|q\right)u^{m+k}v^{n+k}}{\left(q;q\right)_{m}\left(q;q\right)_{n}\left(q;q\right)_{k}}. 
\end{eqnarray*}
The expansion \eqref{eqzzinH} follows from equating like powers of $u$ and $v$.  Similarly we apply the generating function \eqref{eq:hp3} and  find that 
\begin{eqnarray*}
&{}& \left(-q^{1/2}u z_{1},-q^{1/2}z_{2}v;q\right)_{\infty}  =  \left(-uv;q\right)_{\infty}\sum_{m,n=0}^{\infty}\frac{h_{m,n}\left(z_{1},z_{2}|q\right)q^{\left(m-n\right)^{2}/2}u^{m}v^{n}}{\left(q;q\right)_{m}\left(q;q\right)_{n}}\\
 &{} &=  \sum_{k=0}^{\infty}\frac{q^{\binom{k}{2}}}{\left(q;q\right)_{k}}\sum_{m,n=k}^{\infty}\frac{h_{m-k,n-k}\left(z_{1},z_{2}|q\right)q^{\left(m-n\right)^{2}/2}u^{m}v^{n}}{\left(q;q\right)_{m-k}\left(q;q\right)_{n-k}},
\end{eqnarray*}
and \eqref{eqzzinh} follows. 
 \end{proof}

\begin{thm}\label{ceonnHh}
The connection relations between $\{H_{m,n}(z_1,z_2|q)\}$ and $\{h_{m,n}(z_1,z_2|q)\}$ are 
\begin{eqnarray}
\label{fromhtoH}
&{}& H_{m,n}(z_1, z_2|q) \\
&{}& = q^{-mn}\sum_{s=0}^{m \wedge n} \gauss{m}{s} 
\gauss{n}{s} (q;q)_s q^{\binom{s}{2}} h_{m-s, n-s}(z_1, z_2|q) 
\sum_{k=0}^s \gauss{s}{k} (-1)^k q^{k(m+n-k)}, \notag\\
\label{fromHtoh}
&{}& h_{m,n}(z_1, z_2|q) \\
&{}& = \sum_{s=0}^{m \wedge n} \gauss{m}{s} 
\gauss{n}{s} (q;q)_s   H_{m-s, n-s}(z_1, z_2|q) 
\sum_{k=0}^s \gauss{s}{k} (-1)^k q^{(m-k)(n-k)}, \notag
\end{eqnarray}
\end{thm}
\begin{proof}
The theorem follows from the explicit formulas \eqref{eqHmnq}, \eqref{eq:hp1} and the connection relations 
\eqref{eqzzinH}--\eqref{eqzzinh} 
\end{proof}

\begin{thm}
We have the generating functions 
\begin{eqnarray}
\bg
 \sum_{m,n=0}^{\infty}H_{m+j,n+k}(z_{1},z_{2}|q)\frac{u^{m}\; v^{n}}{(q;q)_{m}(q;q)_{n}}
 \qquad     \qquad \qquad    \\
  =\frac{\left(uvq^{j+k};q\right)_{\infty}}{\left(uz_{1},vz_{2};q\right)_{\infty}}\sum_{\ell=0}^{j\wedge k}\gauss{j}{\ell}\gauss{k}{\ell}q^{\binom{\ell}{2}}\left(-1\right)^{\ell}\left(q;q\right)_{\ell}
  \\ 
  \qquad     \qquad \qquad  \times 
  z_{1}^{j-\ell}z_{2}^{k-\ell}\left(\frac{vq^{\ell}}{z_{1}};q\right)_{j-\ell}\left(\frac{uq^{j}}{z_{2}};q\right)_{k-\ell}\left(z_{2}v;q\right)_{\ell}       \\
  =\frac{\left(uvq^{j+k};q\right)_{\infty}}{\left(uz_{1},vz_{2};q\right)_{\infty}}\sum_{\ell=0}^{j\wedge k}\gauss{j}{\ell}\gauss{k}{\ell}q^{\binom{\ell}{2}}\left(-1\right)^{\ell}\left(q;q\right)_{\ell}
  \\ 
  \qquad     \qquad \qquad \times 
 z_{1}^{j-\ell}z_{2}^{k-\ell}\left(\frac{uq^{\ell}}{z_{2}};q\right)_{k-\ell}\left(\frac{vq^{k}}{z_{1}};q\right)_{j-\ell}\left(z_{1}u;q\right)_{\ell}
\eg
\label{eqGFHplus}
\end{eqnarray}
and

\begin{equation}
\begin{aligned} & \sum_{m,n=0}^{\infty}h_{m+j,n+k}\left(z_{1},z_{2}|q\right)q^{\left(m-n\right)^{2}/2}\frac{u^{m}}{\left(q;q\right)_{m}}\frac{v^{n}}{\left(q;q\right)_{n}}\\
 & =\frac{\left(-1\right)^{j+k}q^{\frac{\left(j-k\right)^{2}}{2}}\left(-uz_{1}q^{k+\frac{1}{2}}-vz_{2}q^{j+\frac{1}{2}},q\right)_{\infty}}{\left(-uv;q\right)_{\infty}\left(\frac{z_{1}q^{\frac{1}{2}}}{v};q\right)_{k-j}\left(\frac{z_{2}q^{\frac{1}{2}}}{u};q\right)_{j-k}}\sum_{\ell=0}^{j\wedge k}\gauss{j}{\ell}\gauss{k}{\ell}\\
 & \times\left(-1\right)^{\ell}\left(q;q\right)_{\ell}u^{k-\ell}v^{j-\ell}\left(\frac{z_{1}q^{\frac{1}{2}}}{v};q\right)_{k-\ell}\left(\frac{z_{2}q^{\frac{1}{2}}}{u};q\right)_{j-\ell}\left(-uv;q\right)_{\ell}.
\end{aligned}
\label{eqGFhplus}
\end{equation}
More generally our $q$-disk polynomials have the generating function,
\begin{equation}
\begin{aligned} & \sum_{m,n=0}^{\infty}\frac{p_{m+j,n+k}\left(z_{1},z_{2};b\vert q\right)}{\left(q;q\right)_{m}\left(q;q\right)_{n}}u^{m}v^{n}\\
 & =\frac{\left(bq,uvq^{j+k};q\right)_{\infty}}{\left(uz_{1},vz_{2};q\right)_{\infty}}\sum_{\ell=0}^{\infty}\frac{\left(bq^{1+j+k}\right)^{\ell}\left(uz_{1},vz_{2};q\right)_{\ell}}{\left(q,uvq^{j+k};q\right)_{\ell}}\sum_{i=0}^{j\wedge k}\gauss{j}{i}\gauss{k}{i}\\
 & \times\left(-q^{-\ell}\right)^{i}z_{1}^{j-i}z_{2}^{k-i}q^{\binom{i}{2}}\left(q;q\right)_{i}\left(\frac{v}{z_{1}}q^{i};q\right)_{j-i}\left(\frac{u}{z_{2}}q^{j};q\right)_{k-i}\left(z_{2}vq^{\ell};q\right)_{i}\\
 & =\frac{\left(bq,uvq^{j+k};q\right)_{\infty}}{\left(uz_{1},vz_{2};q\right)_{\infty}}\sum_{\ell=0}^{\infty}\frac{\left(uz_{1},vz_{2};q\right)_{\ell}}{\left(q,uvq^{j+k};q\right)_{\ell}}\left(bq^{1+j+k}\right)^{\ell}\sum_{i=0}^{j\wedge k}\gauss{j}{i}\gauss{k}{i}\\
 & \times\left(-q^{-\ell}\right)^{i}z_{1}^{j-i}z_{2}^{k-i}q^{\binom{i}{2}}\left(q;q\right)_{i}\left(\frac{uq^{i}}{z_{2}};q\right)_{k-i}\left(\frac{vq^{k}}{z_{1}};q\right)_{j-i}\left(uz_{1}q^{\ell};q\right)_{i}.
\end{aligned}
\label{eq:eqGFpplus}
\end{equation}
 \end{thm}
\begin{proof}
Observe that for $k\in\mathbb{N}$ we have 
\[
D_{q,u}^{k}\left(\frac{u^{m}}{\left(q;q\right)_{m}}\right)=\begin{cases}
0 & m<k\\
\frac{1}{\left(1-q\right)^{k}}\frac{u^{m-k}}{\left(q;q\right)_{m-k}} & m\ge k
\end{cases},
\]
\[
D_{q,u}^{k}\left(u^{m}\left(\frac{a}{u};q\right)_{m}\right)=\begin{cases}
0 & m<k\\
\gauss{m}{k}\frac{\left(q;q\right)_{k}}{\left(1-q\right)^{k}}u^{m-k}\left(\frac{a}{u};q\right)_{m-k} & m\ge k
\end{cases},
\]
\[
D_{q,u}^{k}\left(\left(au;q\right)_{m}\right)=\begin{cases}
0 & m<k\\
\gauss{m}{k}\left(\frac{a}{q-1}\right)^{k}q^{\binom{k}{2}}\left(q;q\right)_{k}
\left(auq^{k};q\right)_{m-k} & m\ge k
\end{cases}
\]
 and
\[
D_{q,u}^{k}\left(\frac{\left(au;q\right)_{\infty}}{\left(bu;q\right)_{\infty}}\right)
=\left(\frac{b}{1-q}\right)^{k}\left(\frac{a}{b};q\right)_{k}\frac{\left(auq^{k};q\right)_{\infty}}
{\left(bu;q\right)_{\infty}}.
\]
Then we apply the operator $\left(1-q\right)^{j+k}D_{q,u}^{j}D_{q,v}^{k}$
to the generating function for $h_{m,n}\left(z_{1},z_{2}|q\right)$
to get
\begin{eqnarray*}
 &  & \sum_{m\ge j,n\ge k}h_{m,n}\left(z_{1},z_{2}|q\right)q^{\left(m-n\right)^{2}/2} 
 \frac{u^{m-j}}{\left(q;q\right)_{m-j}}\frac{v^{n-k}}{\left(q;q\right)_{n-k}}\\
 & = & \left(1-q\right)^{j+k}D_{q,v}^{k}D_{q,u}^{j}
 \left(\frac{\left(-uz_{1}q^{1/2},-vz_{2}q^{1/2};q\right)_{\infty}}{\left(-uv;q\right)_{\infty}}\right)\\
 & = & \left(1-q\right)^{k}\left(-1\right)^{j}\left(-uz_{1}q^{j+1/2}; q\right)_{\infty}D_{q,v}^{k}\left(v^{j}
 \left(\frac{z_{1}}{v}q^{1/2};q\right)_{j}\frac{\left(-vz_{2}q^{1/2};q\right)_{\infty}}
 {\left(-uv;q\right)_{\infty}}\right)\\
 & = & \frac{\left(-uz_{1}q^{j+\frac{1}{2}}-vz_{2}q^{k+\frac{1}{2}};  q\right)_{\infty}}
 {\left(-uv; q\right)_{\infty}\left(-1\right)^{j+k}}\sum_{\ell=0}^{k}\gauss{j}{\ell}
 \gauss{k}{\ell}\left(-1\right)^{\ell}\left(q;q\right)_{\ell}\\
 & \times & v^{j-\ell}\left(\frac{z_{1}q^{\frac{1}{2}}}{v};q\right)_{j-\ell} 
 u^{k-\ell}\left(\frac{z_{2}q^{\frac{1}{2}}}{u};q\right)_{k-\ell}\left(-uv;q\right)_{\ell},
\end{eqnarray*}
and \eqref{eqGFhplus} follows. Similarly we find,
\begin{eqnarray*}
&{}& \sum_{m\ge j,n\ge k}^{\infty}H_{m,n}(z_{1},z_{2}|q)\frac{u^{m-j}\; v^{n-k}}{(q;q)_{m-j}
(q;q)_{n-k}}  =  \left(1-q\right)^{j+k}D_{q,v}^{k}D_{q,u}^{j}\left(\frac{(uv;q)_{\infty}}
{(uz_{1},vz_{2};q)_{\infty}}\right)\\
 & = & \frac{\left(uvq^{j+k};q\right)_{\infty}}{\left(uz_{1},vz_{2};q\right)_{\infty}}\sum_{\ell=0}^{j\wedge k}\gauss{j}{\ell}\gauss{k}{\ell}q^{\binom{\ell}{2}}\left(-1\right)^{\ell}\left(q;q\right)_{\ell}
 z_{1}^{j-\ell}z_{2}^{k-\ell}\left(\frac{vq^{\ell}}{z_{1}};q\right)_{j-\ell}\left(\frac{uq^{j}}{z_{2}};q\right)_{k-\ell}\left(z_{2}v;q\right)_{\ell}
\end{eqnarray*}
and
\begin{eqnarray*}
&{}& \sum_{m\ge j,n\ge k}^{\infty}H_{m,n}(z_{1},z_{2}|q)\frac{u^{m-j}\; v^{n-k}}{(q;q)_{m-j}(q;q)_{n-k}}  =  \left(1-q\right)^{j+k}D_{q,u}^{j}D_{q,v}^{k}\left(\frac{(uv;q)_{\infty}}{(uz_{1},vz_{2};q)_{\infty}}\right)\\
 & = & \frac{\left(uvq^{j+k};q\right)_{\infty}}{\left(uz_{1},vz_{2};q\right)_{\infty}}\sum_{\ell=0}^{j\wedge k}\gauss{j}{\ell}\gauss{k}{\ell}q^{\binom{\ell}{2}}\left(-1\right)^{\ell}\left(q;q\right)_{\ell}
 z_{1}^{j-\ell}z_{2}^{k-\ell}\left(\frac{uq^{\ell}}{z_{2}};q\right)_{k-\ell}\left(\frac{vq^{k}}{z_{1}};q\right)_{j-\ell}\left(z_{1}u;q\right)_{\ell},
\end{eqnarray*}
which gives \eqref{eqGFHplus}.

More generally, we have
\begin{eqnarray*}
 &  & \sum_{m\ge j,n\ge k}^{\infty}\frac{p_{m,n}\left(z_{1},z_{2};b\vert q\right)}{\left(q;q\right)_{m-j}\left(q;q\right)_{n-k}}u^{m-j}v^{n-k}\\
 & = & \left(1-q\right)^{j+k}\left(bq;q\right)_{\infty}\sum_{\ell=0}^{\infty}\frac{\left(bq\right)^{\ell}}{\left(q;q\right)_{\ell}}D_{q,v}^{k}D_{q,u}^{j}\frac{\left(uvq^{\ell};q\right)_{\infty}}{\left(uz_{1}q^{\ell},z_{2}vq^{\ell};q\right)_{\infty}}\\
 & = & \left(1-q\right)^{k}\left(bq;q\right)_{\infty}z_{1}^{j}\sum_{\ell=0}^{\infty}\frac{\left(bq^{1+j}\right)^{\ell}}{\left(q;q\right)_{\ell}\left(uz_{1}q^{\ell};q\right)_{\infty}}D_{q,v}^{k}\left(\left(\frac{v}{z_{1}};q\right)_{j}\frac{\left(uvq^{j+\ell};q\right)_{\infty}}{\left(z_{2}vq^{\ell};q\right)_{\infty}}\right)\\
 & = & \left(bq;q\right)_{\infty}\sum_{\ell=0}^{\infty}\frac{\left(bq^{1+j+k}\right)^{\ell}}{\left(q;q\right)_{\ell}\left(uz_{1}q^{\ell};q\right)_{\infty}}\sum_{i=0}^{\infty}\gauss{j}{i}\gauss{k}{i}\\
 & \times & \left(-1\right)^{i}z_{1}^{j-i}z_{2}^{k-i}q^{\binom{i}{2}-i\ell}\left(q;q\right)_{i}\left(\frac{v}{z_{1}}q^{i};q\right)_{j-i}\left(\frac{u}{z_{2}}q^{j};q\right)_{k-i}\frac{\left(uvq^{j+k+\ell};q\right)_{\infty}}{\left(z_{2}vq^{\ell+i};q\right)_{\infty}}\\
 & = & \frac{\left(bq,uvq^{j+k};q\right)_{\infty}}{\left(uz_{1},vz_{2};q\right)_{\infty}}\sum_{\ell=0}^{\infty}\frac{\left(uz_{1}vz_{2};q\right)_{\ell}}{\left(q,uvq^{j+k};q\right)_{\ell}}\left(bq^{1+j+k}\right)^{\ell}\\
 & \times & \sum_{i=0}^{\infty}\gauss{j}{i}\gauss{k}{i}\left(-q^{-\ell}\right)^{i}z_{1}^{j-i}z_{2}^{k-i}q^{\binom{i}{2}}\left(q;q\right)_{i}\left(\frac{v}{z_{1}}q^{i};q\right)_{j-i}\left(\frac{u}{z_{2}}q^{j};q\right)_{k-i}\left(z_{2}vq^{\ell};q\right)_{i},
\end{eqnarray*}
which gives
\[
\begin{aligned} & \sum_{m,n=0}^{\infty}\frac{p_{m+j,n+k}\left(z_{1},z_{2};b\vert q\right)}{\left(q;q\right)_{m}\left(q;q\right)_{n}}u^{m}v^{n}\\
 & =\frac{\left(bq,uvq^{j+k};q\right)_{\infty}}{\left(uz_{1},vz_{2};q\right)_{\infty}}\sum_{\ell=0}^{\infty}\frac{\left(bq^{1+j+k}\right)^{\ell}\left(uz_{1},vz_{2};q\right)_{\ell}}{\left(q,uvq^{j+k};q\right)_{\ell}}\sum_{i=0}^{\infty}\gauss{j}{i}\gauss{k}{i}\\
 & \times\left(-q^{-\ell}\right)^{i}z_{1}^{j-i}z_{2}^{k-i}q^{\binom{i}{2}}\left(q;q\right)_{i}\left(\frac{v}{z_{1}}q^{i};q\right)_{j-i}\left(\frac{u}{z_{2}}q^{j};q\right)_{k-i}\left(z_{2}vq^{\ell};q\right)_{i}.
\end{aligned}
\]
Similarly we have,
\begin{eqnarray*}
 &  & \sum_{m\ge j,n\ge k}^{\infty}\frac{p_{m,n}\left(z_{1},z_{2};b\vert q\right)}{\left(q;q\right)_{m-j}\left(q;q\right)_{n-k}}u^{m-j}v^{n-k}\\
 & = & \left(1-q\right)^{j+k}\left(bq;q\right)_{\infty}\sum_{\ell=0}^{\infty}\frac{\left(bq\right)^{\ell}}{\left(q;q\right)_{\ell}}D_{q,u}^{j}\frac{1}{\left(uz_{1}q^{\ell};q\right)_{\infty}}D_{q,v}^{k}\frac{\left(uvq^{\ell};q\right)_{\infty}}{\left(z_{2}vq^{\ell};q\right)_{\infty}}\\
 & = & \left(1-q\right)^{j}z_{2}^{k}\left(bq;q\right)_{\infty}\sum_{\ell=0}^{\infty}\frac{\left(bq^{k+1}\right)^{\ell}}{\left(q;q\right)_{\ell}\left(z_{2}vq^{\ell};q\right)_{\infty}}
 D_{q,u}^{j}\left(\left(\frac{u}{z_{2}};q\right)_{k}\frac{\left(uvq^{k+\ell};q\right)_{\infty}}{\left(uz_{1}q^{\ell};q\right)_{\infty}}\right)\\
 & = & \frac{\left(bq,uvq^{j+k};q\right)_{\infty}}{\left(uz_{1},vz_{2};q\right)_{\infty}}\sum_{\ell=0}^{\infty}\frac{\left(uz_{1},vz_{2};q\right)_{\ell}}{\left(q,uvq^{j+k};q\right)_{\ell}}\left(bq^{j+k+1}\right)^{\ell}\\
 & \times & \sum_{i=0}^{j\wedge k}\gauss{j}{i}\gauss{k}{i}q^{\binom{i}{2}}\left(-q^{-\ell}\right)^{i}\left(q;q\right)_{i}\\
 & \times & z_{1}^{j-i}z_{2}^{k-i}\left(\frac{uq^{i}}{z_{2}};q\right)_{k-i}\left(\frac{vq^{k}}{z_{1}};q\right)_{j-i}\left(uz_{1}q^{\ell};q\right)_{i},
\end{eqnarray*}
which gives 
\[
\begin{aligned} & \sum_{m,n=0}^{\infty}\frac{p_{m+j,n+k}\left(z_{1},z_{2};b\vert q\right)}{\left(q;q\right)_{m}\left(q;q\right)_{n}}u^{m}v^{n}\\
 & =\frac{\left(bq,uvq^{j+k};q\right)_{\infty}}{\left(uz_{1},vz_{2};q\right)_{\infty}}\sum_{\ell=0}^{\infty}\frac{\left(uz_{1},vz_{2};q\right)_{\ell}}{\left(q,uvq^{j+k};q\right)_{\ell}}\left(bq^{1+j+k}\right)^{\ell}\\
 & \times \sum_{i=0}^{\infty}\gauss{j}{i}\gauss{k}{i}\left(-q^{-\ell}\right)^{i}z_{1}^{j-i}z_{2}^{k-i}q^{\binom{i}{2}}\left(q;q\right)_{i}\left(\frac{uq^{i}}{z_{2}};q\right)_{k-i}
 \left(\frac{vq^{k}}{z_{1}};q\right)_{j-i}\left(uz_{1}q^{\ell};q\right)_{i}.
\end{aligned}
\]
which establishes \eqref{eq:eqGFpplus}.
\end{proof}

\section{Zeros}
In this section we study the zeros of the two 2$D$-$q$-Hermite polynomials and the $q$-analogue of the Zernike polynomials introduced in this paper.  Because  all polynomials factor as a function of $\theta$ times a radial function it is clear that with $z_1= z,  z_2 =\bar z$  the zeros of the polynomials investigated here as functions of $z$  lie on circles.   

Let 
\begin{eqnarray}
0 < i_1(q) < i_2(q) < \cdots,
\end{eqnarray}
be the zeros of $A_q(z)$. 

\begin{thm}
Assume that  the zeros of $H_{m,n}(z, \bar z|q)$ and of  $h_{m,n}(z, \bar z|q)$ lie on the circles  with radii 
\begin{eqnarray}
r_1(H, m,n)  >  r_{2}(H, m,n) > \cdots, \; \; \textup{and}\;\; \;  r_1(h, m,n)  >  r_{2}(h, m,n) > \cdots
\end{eqnarray}
respectively. Moreover let the zeros of $p_{m,n}(z, \bar z; b|q)$  lie on the circle 
$|z| = r_{j}(p, m,n), j=1, 2, \cdots$, ordered as 
\begin{eqnarray}
r_1(p, m,n)  >  r_{2}(p, m,n) > \cdots. 
\end{eqnarray}
 Then 
\begin{eqnarray}
\label{eqlimH}
\lim_{m,n \to \infty} r_{j}(H, m,n) &=& q^{j/2}, \; \; j=1, 2, \cdots, \\
\lim_{m,n \to \infty} q^{(m+n)/2} r_{j}(h, m,n) &=&  1/\sqrt{i_j(q)}, \; \;  j=1, 2, \cdots, 
\label{eqlimh}\\
\lim_{m,n \to \infty} r_{j}(p, m,n) &=& q^{j/2}, \; \; j=1, 2, \cdots.  
\label{eqlimp}
\end{eqnarray}
\end{thm}
\begin{proof}
The first part, \eqref{eqlimH},   follows from  \eqref{eqHmntoinf} since its left-hand side 
converges to its right-hand side on compact subsets of $\mathbb{C}$.  Similary 
\eqref{eqlimp} follows from  \eqref{eqasypmn}.  
Formula    \eqref{eqlimh} follows from Theorem \ref{thm:h-asymptotics}  since the limit in 
Theorem \ref{thm:h-asymptotics} is uniform on compact subsets of $\mathbb{C}$. 
\end{proof}

It is important to note that the support of the orthogonality measure of 
$\{H_{m,n}(z, \bar z|q)\}$ and $\{p_{m,n}(z, \bar z;b|q)\}$ coincides with the closure of the 
union of the limiting circles on which the zeros lie. This is similar to to the single variable 
case.  It is not surprising that the zeros of the Ramanujan function appear in the leading 
terms of the asymptotics of zeros of the polynomials $\{h_{m,n}(z, \bar z|q)\}$. This is 
again similar to the single variable case.

\section{Positivity Results} 
\begin{lem}
For $N\in\mathbb{N}_{0}$, $q\in\left(0,1\right)$ and $z\in\mathbb{C\backslash}\left\{ 0\right\} $,
the following matrices are positive definite
\begin{equation}
\left(\frac{H_{m,n}\left(iz,i\overline{z}\right)}{i^{m+n}}\right)_{m,n=0}^{N},\qquad\left(q^{mn}h_{m,n}\left(z,\overline{z}|q^{-1}\right)\right)_{m,n=0}^{N},\label{eq:do1}
\end{equation}
\begin{equation}
\left(\frac{q^{\frac{\left(m-n\right)^{2}}{2}}h_{m,n}\left(iz,i\overline{z}|q\right)}{i^{m+n}}\right)_{m,n=0}^{N},\qquad\left(H_{m,n}(z,\overline{z}|q^{-1}\right)_{m,n=0}^{N}.\label{eq:do2}
\end{equation}
 \end{lem}
\begin{proof}
Observe that

\begin{eqnarray*}
 &  & \frac{H_{m,n}(iz,i\overline{z}|q)}{i^{m+n}}=q^{mn}h_{m,n}\left(z,\overline{z}|q^{-1}\right)\\
 & = & \sum_{k=0}^{\infty}\frac{q^{\binom{k}{2}}(q;q)_{k}}{\left|z\right|^{2k}}\left\{ \left[\begin{array}{c}
m\\
k
\end{array}\right]_{q}z^{m}\right\} \cdot\overline{\left\{ \left[\begin{array}{c}
n\\
k
\end{array}\right]_{q}z^{n}\right\} }
\end{eqnarray*}
 and
\begin{eqnarray*}
 &  & \frac{q^{\frac{\left(m-n\right)^{2}}{2}}h_{m,n}\left(iz,i\overline{z}|q\right)}{i^{m+n}}=q^{\frac{m^{2}}{2}}q^{\frac{n^{2}}{2}}H_{m,n}\left(z,\overline{z}|q^{-1}\right)\\
 & = & \sum_{j=0}^{\infty}\frac{q^{j^{2}}\left(q;q\right)_{j}}{\left|z\right|^{2j}}\left\{ \left[\begin{array}{c}
m\\
j
\end{array}\right]_{q}q^{\frac{m^{2}}{2}-mj}z^{m}\right\} \overline{\left\{ \left[\begin{array}{c}
n\\
j
\end{array}\right]_{q}q^{\frac{n^{2}}{2}-nj}z^{n}\right\} }.
\end{eqnarray*}
 \end{proof}
\begin{thm}
For $z\neq0$ and $0<q<1$, there exist sequences $e_{n}=\left\{ e_{m}^{(n)}\right\} _{m=0}^{\infty}$
and $f_{n}=\left\{ f_{m}^{(n)}\right\} _{m=0}^{\infty}$ with $e_{m}^{(n)}=f_{m}^{(n)}=0$
for $m>n$ such that
\begin{equation}
\begin{aligned} & \int_{\mathbb{R}^{2}}\left\{ \sum_{m=0}^{j}e_{m}^{(j)}z^{m}\left(-\frac{\zeta q^{\frac{1}{2}}}{z};q\right)_{m}\right\} \overline{\left\{ \sum_{n=0}^{k}e_{n}^{(k)}z^{n}\left(-\frac{\zeta q^{\frac{1}{2}}}{z};q\right)_{n}\right\} }d\nu\left(\zeta,\bar{\zeta}\right)\\
= & \sum_{\ell=0}^{\infty}\frac{q^{\binom{\ell}{2}}(q;q)_{\ell}}{\left|z\right|^{2\ell}}\sum_{m=\ell}^{j}\left[\begin{array}{c}
m\\
\ell
\end{array}\right]_{q}e_{m}^{(j)}z^{m}\overline{\sum_{n=\ell}^{k}\left[\begin{array}{c}
n\\
\ell
\end{array}\right]_{q}e_{n}^{(k)}z^{n}}=\delta_{j,k}
\end{aligned}
\label{eq:do3}
\end{equation}
and
\begin{equation}
\begin{aligned} & \int_{\mathbb{R}^{2}}\left\{ \sum_{m=0}^{j}f_{m}^{(j)}\left(-\zeta\right)^{m}\left(\frac{zq^{\frac{1}{2}}}{\zeta};q\right)_{m}\right\} \cdot\overline{\left\{ \sum_{n=0}^{k}f_{n}^{(k)}\left(-\zeta\right)^{n}\left(\frac{zq^{\frac{1}{2}}}{\zeta};q\right)_{n}\right\} }d\mu\left(\zeta,\bar{\zeta}\right)\\
= & \sum_{\ell=0}^{\infty}\frac{q^{\ell^{2}}\left(q;q\right)_{\ell}}{\left|z\right|^{2\ell}}\left\{ \sum_{m=\ell}^{j}\left[\begin{array}{c}
m\\
\ell
\end{array}\right]_{q}q^{\frac{m^{2}}{2}-m\ell}z^{m}f_{m}^{(j)}\right\} \overline{\left\{ \sum_{n=\ell}^{k}\left[\begin{array}{c}
n\\
\ell
\end{array}\right]_{q}q^{\frac{n^{2}}{2}-n\ell}z^{n}f_{n}^{(k)}\right\} }=\delta_{j,k}
\end{aligned}
\label{eq:do4}
\end{equation}
 for $j,k=0,1,2,\dots$.\end{thm}
\begin{proof}
Let us define the following inner vector spaces 
\[
H(\mathbb{N}_{0};z,q)=\left\{ \left\{ c_{n}\right\} _{n=0}^{\infty}\bigg|c_{n}\in\mathbb{C},n\in\mathbb{N}_{0},\quad\sum_{m,n=0}^{\infty}\frac{H_{m,n}\left(iz,i\overline{z}\right)}{i^{m+n}}c_{m}\overline{c_{n}}<\infty\right\} 
\]
with 
\[
\left(\left\{ c_{n}\right\} _{n=0}^{\infty},\left\{ d_{n}\right\} _{n=0}^{\infty}\right)_{H}=\sum_{m,n=0}^{\infty}\frac{H_{m,n}\left(iz,i\overline{z}\right)}{i^{m+n}}c_{m}\overline{d_{n}},
\]
where $\left\{ c_{n}\right\} _{n=0}^{\infty},\left\{ d_{n}\right\} _{n=0}^{\infty}\in H(\mathbb{N}_{0};z,q)$
and
\[
h(\mathbb{N}_{0};z,q)=\left\{ \left\{ c_{n}\right\} _{n=0}^{\infty}\bigg|c_{n}\in\mathbb{C},n\in\mathbb{N}_{0},\quad\sum_{m,n=0}^{\infty}\frac{q^{\frac{\left(m-n\right)^{2}}{2}}h_{m,n}\left(iz,i\overline{z}|q\right)}{i^{m+n}}c_{m}\overline{c_{n}}<\infty\right\} 
\]
with
\[
\left(\left\{ c_{n}\right\} _{n=0}^{\infty},\left\{ d_{n}\right\} _{n=0}^{\infty}\right)_{h}=\sum_{m,n=0}^{\infty}\frac{q^{\frac{\left(m-n\right)^{2}}{2}}h_{m,n}\left(iz,i\overline{z}|q\right)}{i^{m+n}}c_{m}\overline{d_{n}},
\]
where $\left\{ c_{n}\right\} _{n=0}^{\infty},\left\{ d_{n}\right\} _{n=0}^{\infty}\in h(\mathbb{N}_{0};z,q)$.
Then, the vectors $\left\{ \delta_{m,n}\right\} _{m=0}^{\infty},\quad n=0,1,\dots$
are linearly independent in these spaces. For $n\in\mathbb{N}_{0}$,
let $e_{n}=\left\{ e_{m}^{(n)}\right\} _{m=0}^{\infty}$ and $f_{n}=\left\{ f_{m}^{(n)}\right\} _{m=0}^{\infty}$
be the obtained orthonormal bases from the orthogonalization process
in $H(\mathbb{N}_{0};z,q)$ and $h(\mathbb{N}_{0};z,q)$ respectively,
then it is clear that $e_{m}^{(n)}=f_{m}^{(n)}=0$ for $m>n$. Observe
that
\[
\frac{H_{m,n}\left(iz,i\overline{z}\right)}{i^{m+n}}=z^{m}\overline{z}^{n}\int_{\mathbb{R}^{2}}\left(-\frac{\zeta q^{\frac{1}{2}}}{z};q\right)_{m}\left(-\frac{\bar{\zeta}q^{\frac{1}{2}}}{\overline{z}};q\right)_{n}d\nu\left(\zeta,\bar{\zeta}\right),
\]
then, 
\begin{eqnarray*}
 &  & \left(\left\{ e_{m}^{(j)}\right\} _{m=0}^{\infty},\left\{ e_{m}^{(k)}\right\} _{m=0}^{\infty}\right)_{H}=\sum_{m,n=0}^{\infty}\frac{H_{m,n}\left(iz,i\overline{z}\right)}{i^{m+n}}e_{m}^{(j)}\overline{e_{n}^{(k)}}=\delta_{j,k}\\
 & = & \sum_{\ell=0}^{\infty}\frac{q^{\binom{\ell}{2}}(q;q)_{\ell}}{\left|z\right|^{2\ell}}\sum_{m=\ell}^{j}\left[\begin{array}{c}
m\\
\ell
\end{array}\right]_{q}e_{m}^{(j)}z^{m}\overline{\sum_{n=\ell}^{k}\left[\begin{array}{c}
n\\
\ell
\end{array}\right]_{q}e_{n}^{(k)}z^{n}}\\
 & = & \int_{\mathbb{R}^{2}}\left\{ \sum_{m=0}^{j}e_{m}^{(j)}z^{m}\left(-\frac{\zeta q^{\frac{1}{2}}}{z};q\right)_{m}\right\} \overline{\left\{ \sum_{n=0}^{k}e_{n}^{(k)}z^{n}\left(-\frac{\zeta q^{\frac{1}{2}}}{z};q\right)_{n}\right\} }d\nu\left(\zeta,\bar{\zeta}\right).
\end{eqnarray*}
Similarly, from
\[
\frac{q^{\frac{\left(m-n\right)^{2}}{2}}h_{m,n}\left(iz,i\overline{z}|q\right)}{i^{m+n}}=\int_{\mathbb{R}^{2}}\left(-\zeta\right)^{m}\left(-\bar{\zeta}\right)^{n}\left(\frac{zq^{\frac{1}{2}}}{\zeta};q\right)_{m}\left(\frac{\overline{z}q^{\frac{1}{2}}}{\bar{\zeta}};q\right)_{n}d\mu\left(\zeta,\bar{\zeta}\right)
\]
we get
\begin{eqnarray*}
 &  & \left(\left\{ f_{m}^{(j)}\right\} _{m=0}^{\infty},\left\{ f_{m}^{(k)}\right\} _{m=0}^{\infty}\right)_{h}=\sum_{m,n=0}^{\infty}\frac{q^{\frac{\left(m-n\right)^{2}}{2}}h_{m,n}\left(iz,i\overline{z}|q\right)}{i^{m+n}}f_{m}^{(j)}\overline{f_{n}^{(k)}}=\delta_{j,k}\\
 & = & \sum_{\ell=0}^{\infty}\frac{q^{\ell^{2}}\left(q;q\right)_{\ell}}{\left|z\right|^{2\ell}}\left\{ \sum_{m=\ell}^{j}\left[\begin{array}{c}
m\\
\ell
\end{array}\right]_{q}q^{\frac{m^{2}}{2}-m\ell}z^{m}f_{m}^{(j)}\right\} \overline{\left\{ \sum_{n=\ell}^{k}\left[\begin{array}{c}
n\\
\ell
\end{array}\right]_{q}q^{\frac{n^{2}}{2}-n\ell}z^{n}f_{n}^{(k)}\right\} }\\
 & = & \int_{\mathbb{R}^{2}}\left\{ \sum_{m=0}^{j}f_{m}^{(j)}\left(-\zeta\right)^{m}\left(\frac{zq^{\frac{1}{2}}}{\zeta};q\right)_{m}\right\} \cdot\overline{\left\{ \sum_{n=0}^{k}f_{n}^{(k)}\left(-\zeta\right)^{n}\left(\frac{zq^{\frac{1}{2}}}{\zeta};q\right)_{n}\right\} }d\mu\left(\zeta,\bar{\zeta}\right).
\end{eqnarray*}
 \end{proof}
\begin{rem}
If we could inverse the matrices \eqref{eq:do1} and \eqref{eq:do2},
then we can determine $e_{n},\ f_{n},\quad n=0,1,\dots$ explicitly.\end{rem}
\begin{lem}
For $z\cdot\zeta\neq0$ and $m=0,1,\dots$ we have
\begin{equation}
\zeta^{m}=\sum_{j=0}^{m}\left[\begin{array}{c}
m\\
j
\end{array}\right]_{q}q^{\binom{m-j}{2}}\left(-z\right)^{m-j}\left(-\zeta q^{1/2}/z;q\right)_{j}z^{j}\label{eq:do5}
\end{equation}
 and
\begin{equation}
\zeta^{m}=\sum_{j=0}^{m}\left[\begin{array}{c}
m\\
j
\end{array}\right]_{q}\left(zq^{1/2}\right)^{m-j}\left(\frac{zq^{1/2}}{\zeta};q\right)_{j}\zeta^{j}.\label{eq:do6}
\end{equation}
\end{lem}
\begin{proof}
From the $q$-binomial theorem we have

\[
\sum_{m=0}^{\infty}\frac{\left(-\zeta q^{1/2}/z;q\right)_{m}\left(zt\right)^{m}}{\left(q;q\right)_{m}}=\frac{\left(-\zeta tq^{1/2};q\right)_{\infty}}{\left(zt;q\right)_{\infty}},
\]
 to get
\begin{eqnarray*}
 &  & \left(-\zeta tq^{1/2};q\right)_{\infty}=\sum_{m=0}^{\infty}\frac{t^{m}}{\left(q;q\right)_{m}}q^{m^{2}/2}\zeta^{m}\\
 & = & \sum_{m=0}^{\infty}\frac{t^{m}}{\left(q;q\right)_{m}}\sum_{j=0}^{m}\left[\begin{array}{c}
m\\
j
\end{array}\right]_{q}q^{\binom{m-j}{2}}\left(-z\right)^{m-j}\left(-\zeta q^{1/2}/z;q\right)_{j}z^{j},
\end{eqnarray*}
and \eqref{eq:do5} is obtained by matching the coefficients of $t^{m}$. 

Similarly, from
\[
\frac{1}{\left(\zeta t;q\right)_{\infty}}=\frac{1}{\left(ztq^{1/2};q\right)_{\infty}}\sum_{k=0}^{\infty}\frac{\left(\zeta t\right)^{k}\left(\frac{zq^{\frac{1}{2}}}{\zeta};q\right)_{k}}{\left(q;q\right)_{k}}
\]
to get \eqref{eq:do6}. \end{proof}
\begin{cor}
Let $z\neq0$ and $0<q<1$, then for $j,k=0,1,\dots$ we have
\begin{equation}
\sum_{\ell=0}^{\infty}\frac{q^{\binom{\ell}{2}}(q;q)_{\ell}}{\left|z\right|^{2\ell}}e_{j}\left(\ell\vert q\right)e_{k}\left(\ell\vert q\right)=\frac{\left(q;q\right)_{j}\log q^{-1}}{q^{\binom{j+1}{2}}\left|z\right|^{2j}}\delta_{j,k}\label{eq:do7}
\end{equation}
 and 
\begin{equation}
\sum_{\ell=0}^{\infty}\frac{q^{\ell^{2}}\left(q;q\right)_{\ell}}{\left|z\right|^{2\ell}}f_{j}\left(\ell\vert q\right)f_{k}\left(\ell\vert q\right)=\frac{\delta_{j,k}}{\left(q^{j+1};q\right)_{\infty}\left|z\right|^{2j}},\label{eq:do8}
\end{equation}
 where
\begin{equation}
e_{j}\left(\ell\vert q\right)=\sum_{m=\ell}^{j}\left[\begin{array}{c}
m\\
\ell
\end{array}\right]_{q}\left[\begin{array}{c}
j\\
m
\end{array}\right]_{q}q^{\binom{j-m}{2}}\left(-1\right)^{j-m}\label{eq:do9}
\end{equation}
 and
\begin{equation}
f_{j}\left(\ell\vert q\right)=\sum_{m=\ell}^{j}\left[\begin{array}{c}
m\\
\ell
\end{array}\right]_{q}\left[\begin{array}{c}
j\\
m
\end{array}\right]_{q}q^{\frac{m^{2}}{2}-m\ell}\left(-q^{1/2}\right)^{j-m}.\label{eq:do10}
\end{equation}
\end{cor}
\begin{proof}
Observe that
\begin{eqnarray*}
 &  & \delta_{j,k}\frac{\left(q;q\right)_{j}\log q^{-1}}{q^{\binom{j+1}{2}}}=\int_{\mathbb{\mathbb{R}}^{2}}\zeta^{j}\overline{\zeta}^{k}d\nu\left(\zeta,\bar{\zeta}\right)\\
 & = & \int_{\mathbb{R}^{2}}\left\{ \sum_{m=0}^{j}\left[\begin{array}{c}
j\\
m
\end{array}\right]_{q}q^{\binom{j-m}{2}}\left(-z\right)^{j-m}\left(-\zeta q^{1/2}/z;q\right)_{m}z^{m}\right\} \\
 & \times & \overline{\left\{ \sum_{n=0}^{k}\left[\begin{array}{c}
k\\
n
\end{array}\right]_{q}q^{\binom{k-n}{2}}\left(-z\right)^{k-n}\left(-\zeta q^{1/2}/z;q\right)_{n}z^{n}\right\} }d\nu\left(\zeta,\bar{\zeta}\right).
\end{eqnarray*}
Let us take
\[
e_{m}^{(j)}=\frac{q^{\frac{j\left(j+1\right)}{4}}}{\sqrt{\left(q;q\right)_{j}\log q^{-1}}}\left[\begin{array}{c}
j\\
m
\end{array}\right]_{q}q^{\binom{j-m}{2}}\left(-z\right)^{j-m}
\]
to get \eqref{eq:do7} and \eqref{eq:do9}. Similarly, from
\begin{eqnarray*}
 &  & \frac{\delta_{j,k}}{\left(q^{j+1};q\right)_{\infty}}=\int_{\mathbb{\mathbb{R}}^{2}}\zeta^{j}\overline{\zeta}^{k}d\mu\left(\zeta,\bar{\zeta}\right)\\
 & = & \int_{\mathbb{R}^{2}}\left\{ \sum_{m=0}^{j}\left[\begin{array}{c}
j\\
m
\end{array}\right]_{q}\left(zq^{1/2}\right)^{j-m}\left(\frac{zq^{1/2}}{\zeta};q\right)_{m}\zeta^{m}\right\} \\
 & \times & \left\{ \overline{\sum_{n=0}^{k}\left[\begin{array}{c}
k\\
n
\end{array}\right]_{q}\left(zq^{1/2}\right)^{k-n}\left(\frac{zq^{1/2}}{\zeta};q\right)_{n}\zeta^{n}}\right\} d\mu\left(\zeta,\bar{\zeta}\right),
\end{eqnarray*}
we take 
\[
f_{m}^{(j)}=\sqrt{\left(q^{j+1};q\right)_{\infty}}\left[\begin{array}{c}
j\\
m
\end{array}\right]_{q}\left(-zq^{1/2}\right)^{j-m}
\]
to obtain \eqref{eq:do8} and \eqref{eq:do10}.\end{proof}

\bigskip

\noindent{\bf Acknowledgement}:  Part of this work was done when the authors were visiting the City University of Hong Kong and we acknowledge the  hospitality of the Department of Mathematics at City University. We also thank George Andrews, Dan Dai, and Erik Koelink for interesting 
discussions and suggestions.


\end{document}